\documentclass[11pt]{amsart}
\usepackage[latin1]{inputenc} % ±ØÐë¼ÓÕâ¸ö£¬·ñÔòÍ¶¸å±àÒë²»³É¹¦
\usepackage{listings, tcolorbox}% ±ØÐë¼ÓÕâ¸ö£¬·ñÔòÍ¶¸å±àÒë²»³É¹¦

\usepackage{amssymb, amsmath, amsthm}
\usepackage[colorlinks=true,linkcolor=blue,citecolor=red]{hyperref}
\usepackage{color}
\usepackage{graphics}
\usepackage{subfigure}
\usepackage{graphicx}  %²åÈëÍ¼Æ¬µÄºê°ü
\usepackage{float}  %ÉèÖÃÍ¼Æ¬¸¡¶¯Î»ÖÃµÄºê°ü
\usepackage{epsfig}
\usepackage{amssymb}
\usepackage{epstopdf}
\usepackage{enumerate}
\usepackage{appendix}
\usepackage{amscd}
\usepackage{mathrsfs}
\usepackage{tikz}
\usepackage{bm}
\usepackage[T1]{fontenc}

\numberwithin{equation}{section}
\newtheorem{theorem}{Theorem}[section]
\newtheorem{corollary}[theorem]{Corollary}
\newtheorem{lemma}[theorem]{Lemma}
\newtheorem{proposition}[theorem]{Proposition}
\newtheorem{assumption}[theorem]{Assumption}
\theoremstyle{definition}

\newtheorem{remark}[theorem]{Remark}
\newtheorem{RHP}[theorem]{RH Problem}

\newcommand{\R}{\mathbb{R}}
\newcommand{\C}{\mathbb{C}}


\begin{document}
	
	\title[On Cauchy problem to the spin-1 Gross-Pitaevskii equation]{On soliton resolution to Cauchy problem of the spin-1 Gross-Pitaevskii equation}
	
	%\today

	\author[S. F. Tian]{Shou-Fu Tian$^{*}$}

    \author[J. F. Tong]{Jia-Fu Tong$^{*,\dag}$}
	
	\address{Shou-Fu Tian  (Corresponding author) \newline
		School of Mathematics, China University of Mining and Technology, Xuzhou 221116, China}
\email{sftian@cumt.edu.cn, shoufu2006@126.com }
\address{Jia-Fu Tong  (Corresponding author) \newline
		School of Mathematics, China University of Mining and Technology, Xuzhou 221116, China}
\email{jftong@cumt.edu.cn}	
	
	\thanks{$^{*}$Corresponding authors(sftian@cumt.edu.cn, shoufu2006@126.com (S.F. Tian) and jftong@cumt.edu.cn (J.F. Tong)).\\
\hspace*{3ex}$^\dag$This author is contributed equally as the first author.}
	
\begin{abstract}
		{We investigate the Cauchy problem for the spin-1 Gross-Pitaevskii(GP) equation, which is a model instrumental in characterizing the soliton dynamics within spinor
Bose-Einstein condensates.  Recently, Geng $etal.$ (Commun. Math. Phys. 382, 585-611 (2021)) reported the long-time asymptotic result with error $\mathcal{O}(\frac{\log t}t)$ for the spin-1 GP equation that only exists in the continuous spectrum. The main purpose of our work is to further generalize and improve Geng's work. Compared with the previous work, our asymptotic error accuracy has been improved from $\mathcal{O}(\frac{\log t}t)$ to $\mathcal{O}(t^{-3/4})$. More importantly, by establishing two matrix valued functions, we obtained effective asymptotic errors and successfully constructed asymptotic analysis of the spin-1 GP equation based on the characteristics of the spectral problem, including two cases: (i)coexistence of discrete and continuous spectrum;
(ii)only continuous spectrum which considered by Geng's work with error $\mathcal{O}(\frac{\log t}t)$.
For the case (i),
the corresponding asymptotic approximations can be characterized with an $N$-soliton as well as an interaction term between soliton solutions and the dispersion term with
diverse residual error order $\mathcal{O}(t^{-3/4})$. For the case (ii), the corresponding asymptotic approximations can be characterized with the leading term on the
continuous spectrum and the residual error order $\mathcal{O}(t^{-3/4})$. Finally, our results confirm the soliton resolution conjecture for the spin-1 GP equation.}
\end{abstract}

	\maketitle
	\tableofcontents
	\section{Introduction}
	Nonlinear partial differential equations have been instrumental in describing the nonlinear wave dynamics across various fields, including Bose-Einstein condensates,
optical fiber communications, fluid mechanics, and plasma physics. It has been known that at the very low but critical transition temperature, a large fraction of the atoms
would condense and occupy into the same the lowest energy state \cite{boo-3,boo-2}. This novel state, proposed by Bose and Einstein in 1925, is named as
Bose-Einstein condensate (BEC). It was experimentally realized by Anderson, Ensher, Matthews, Wieman and Cornell in 1995 \cite{boo-4}. The BEC applications have been seen in
the superfluidity in the liquid helium and superconductivity in the metals \cite{boo-7,boo-5,boo-6}. The BEC can be described by means of an effective mean-field theory and
the relevant model is a classical nonlinear evolution equation, the GP equation \cite{boo-8,boo-9}. BEC can have single component and multi-component. In the single-component
BEC, the GP equation, also known as the nonlinear Schr{\"o}dinger (NLS) equation in one dimension, is the relevant dynamic model. Multi-component BECs have garnered considerable
interest due to their complex and varied dynamical behaviors \cite{boo-10,boo-11,boo-12}. Among the fully integrable models for three-component BECs, the spin-1 GP equation
\begin{equation}\label{GP}
    \begin{cases}iq_{1t}+q_{1xx}+2(|q_1|^2+2|q_0|^2)q_1+2q_0^2\bar{q}_{-1}=0,\\iq_{0t}+q_{0xx}+2(|q_1|^2+2|q_0|^2+|q_{-1}|^2)q_0+2q_1q_{-1}\bar{q}_0=0,\\
    iq_{-1t}+q_{-1xx}+2(2|q_0|^2+|q_{-1}|^2)q_{-1}+2q_0^2\bar{q}_1=0,\end{cases}\end{equation}
    where $q_{1}(x,t)$, $q_{0}(x,t)$ and $q_{-1}(x,t)$  are three potentials functions, stands out for its ability to describe soliton dynamics within spinor BECs \cite{boo-13}. In 2004, Ieda $etal$. studied
    the $N$-soliton solutions of the spin-1 GP equation \cite{boo-13}. In addition, inverse scattering transform and Darboux transform are also used to study soliton
    solutions of the spin-1 GP equation \cite{boo-27,boo-14,boo-15}. In 2017, Yan studied the half-line for the initial-boundary problems of the spin-1 GP equation \cite{boo-17}.
    In 2018, Prinari $etal$. studied soliton solutions of the spin-1 GP equation under non-zero boundary conditions \cite{boo-18}. In 2019, Yan studied finite interval
    for the initial-boundary problems of the spin-1 GP equation \cite{boo-19}. Recently, Geng $etal$. extend the $\text{Deift-Zhou's}$ nonlinear steepest descent method to
    study the long-time asymptotics for the Cauchy problem of the spin-1 GP equation under only continuous spectrum \cite{boo-1}.

    We revisit the Cauchy problem \eqref{GP} with the initial values
    \begin{equation}\nonumber
    q_0^0(x)=q_0(x,0), q_1^0(x)=q_1(x,0), q_{-1}^0(x)=q_{-1}(x,0).
    \end{equation}
    $q_0^0(x)$, $q_1^0(x)$ and $q_{-1}^0(x)$ lie in the Schwartz space
    $\mathscr{S}(\mathbb{R}) = \{f(x) \in C^{\infty}(\mathbb{R}) : \sup_{x\in\mathbb{R}}|x^{\alpha}\partial^{\beta}f(x)| < \infty,\forall\alpha,\beta \in \mathbb{N}\}$
    considered by Geng $etal$. in \cite{boo-1}, in which they derived the leading order approximation to the solution of the Cauchy problem using $\text{Deift-Zhou's}$ nonlinear
    steepest descent method
    \begin{equation}
    (q_1(x,t),q_0(x,t),q_{-1}(x,t))=\frac{\sqrt{\pi}(\delta^0)^2e^{-\frac{\pi\nu}2}e^{\frac{-3\pi i}4}}{\sqrt{t}\Gamma(-i\nu)\det\gamma^\dagger(k_0)}(\bar{\gamma}_{22}(k_0),
    -\bar{\gamma}_{21}(k_0),\bar{\gamma}_{11}(k_0))+\mathcal{O}\left(\frac{\log t}t\right),
\end{equation}
    this result provides the long-time asymptotics without solitons for Cauchy problem of the spin-1 GP equation.

    Drawing on the seminal contributions of Manakov to the study of nonlinear evolution equations asymptotic properties over extended periods \cite{boo-31}, later scholars continuously follow his step \cite{boo-2000,boo-3000,boo-4000}. The $\text{Deift-Zhou's}$ nonlinear steepest descent method was first proposed by Deift and Zhou to analyze the long-time asymptotic behavior of solutions of modified
    Korteweg-de Vries (mKdV) equations In 1993 \cite{boo-32}. Later, this method was widely applied to Camassa-Holm equation \cite{boo-33}, focusing NLS equation \cite{boo-34},
     KdV equation \cite{boo-35}, Toda equation \cite{boo-36}, Sine-Gordon equation \cite{boo-37} and so on \cite{boo-40,boo-38,boo-39}. Recently, in 2006, McLaughlin and
     Miller have further proposed the $\bar{\partial}$ steepest descent method, which combines the steepest descent method to analyze the asymptotics of orthogonal polynomials
      \cite{boo-42,boo-41}. Compared with the classical nonlinear steepest descent method, the advantage lies in improving the error accuracy on the one hand, and avoiding
      the complex norm estimation on the other hand. This method has been employed to address the asymptotic analysis of the initial value problem for NLS equations without
      soliton, ever since its introduction in \cite{boo-42}. Later on, the soliton resolution conjecture for the focusing NLS equation with rapidly decaying initial values
      was tackled using the steepest descent approach within the framework of integrable systems, as detailed in \cite{boo-23}. The defocused NLS equation with finite density
      initial values also yields similar results and \cite{boo-44} gave the asymptotic stability of the $N$-soliton solution. Furthermore, this method is also applied to
      various other nonlinear integrable systems, encompassing the modified CH equation \cite{boo-45}, the short-pulse equation \cite{boo-46}, the Fokas-Lenells equation
      \cite{boo-47}, the Wadati-Konno-Ichikawa equation \cite{boo-48,boo-49}, the defocusing nonlinear Schr{\"o}dinger equation with a nonzero background \cite{boo-70} and
      so on. However, most work was concentrated on integrable systems associated with $2\times2$ matrix spectral problems. When considering the nonlinear integrable systems
      associated with higher-order matrix scattering problems, the corresponding long-time asymptotic analysis of solutions becomes more difficult. On the one hand, the main
      part of the spectrum problem may have more than two different eigenvalues, which requires the development of Fredholm integral equations to express the related higher
      order matrix RH problem, see \cite{boo-50,boo-51,boo-52,boo-53,boo-54,boo-55}. On the other hand, it is necessary to employ some tricks to decompose the jump matrix
      into upper and lower triangles, for example, to overcome this problem by introducing a $2\times2$ \cite{boo-56,boo-57}, $3\times3$ \cite{boo-58} matrix-valued function
       $\delta_j(k)$ defined in \eqref{k1}, and more, two $2\times2$ matrix-valued functions \cite{boo-1} that satisfy the RH problem of two different matrices. Next is how
       to deal with the inexact solvability of the function $\delta_j(k)$, solve a matrix RH problem, and solve the corresponding higher order matrix model RH problem.
       Therefore, it is very meaningful and challenging to study the long-term asymptotics of the solutions of integrable nonlinear evolutionary equations related to the
       spectrum problems of higher order matrices.

    In what follows, we will study the soliton resolution conjecture of the spin-1 GP equation \eqref{GP} associated with the $4\times4$ matrix spectruml problem.
    The concept of soliton resolution conjecture denotes the phenomenon where, as $|t|\to\infty $, the solution separates into a finite number of separated solitons and
    a radiative part. The parameters of these asymptotic solitons experience slight modulations due to interactions between solitons themselves and between solitons and
    the radiation. The interesting point is that soliton resolution conjecture is better understood in integrable systems, and the solution provided by RH problem is more
    accurate than that obtained by pure analytic techniques \cite{boo-300,boo-100,boo-200}. Our analysis comprehensively details the dispersive element, which encompasses
    two distinct parts: one originating from the continuous spectrum and the other stemming from the interplay between the discrete and continuous spectrum. This
    decomposition is a core features in nonlinear wave dynamics and has been the object of many theoretical and numerical studies. The realm of nonlinear dispersion equations
    remains a dynamic and burgeoning area of inquiry \cite{boo-23,boo-61,boo-60}. Although \cite{boo-1} gives the asymptotic behavior without solitons of the solution
    for the spin-1 GP equation from the classical nonlinear steepest descent method, in this work we will further consider the more rich asymptotic behavior of the spin-1 GP equation from the perspective of $\bar{\partial}$-generalization of the Deift-Zhou's steepest descent method and $\text{Deift-Zhou's}$ nonlinear steepest descent method,
    and verify the validity of the soliton resolution conjecture.

    \begin{remark}
    The long-time asymptotic behavior of the spin-1 GP equation under non-zero boundary conditions is also being prepared.
    \end{remark}

\textbf{Our paper is arranged as follows:}\\
    In Sect. \ref{s:2}, we quickly review some basic results, especially the construction of a basic RH formalism $M(k)$ related to the Cauchy problem for the spin-1 GP
     equation \eqref{GP}.

    In Sect. \ref{s:3}, we focus on the long-time asymptotic analysis for the spin-1 GP equation in the region $\xi\neq0$ with the following steps. First of all, we
     obtain a standard RH problem for $M^{(1)}(k)$ by categorizing the poles of the RH problem for $M(k)$ in Sect. \ref{s:3.1}. Then in Sect. \ref{s:3.2}, after a
      continuous extension of the jump matrix with the $\bar{\partial}$ steepest descent method, the RH problem for $M^{(1)}(k)$ is deformed into a hybrid
      $\bar{\partial}\text{-RH}$ problem for $M^{(2)}(k)$, which can be solved by decomposing it into a pure RH problem for $M_{RHP}(k)$ and a pure
      $\bar{\partial}$-problem for $M^{(3)}(k)$. The RH problem for $M_{RHP}(k)$ can be constructed by a solvable parabolic-cylinder model, and the residual error comes
       from a small-norm RH problem for $E(k)$ described in Sect. \ref{s:3.3}. In Sect. \ref{s:3.4}, we prove the existence of the solution $M^{(3)}(k)$ and estimate
        its size. In Sect. \ref{s:3.5} and Sect. \ref{s:3.6}, we get long-time asymptotic for the spin-l GP equation for $|k-k_0|\geq a$ and $|k-k_0|<a$.

    In Sect. \ref{s:4}, we investigate the asymptotics of the solution in the region $\xi=0$ using a similar way as Sect. \ref{s:3}.

   In Sect. \ref{s.5}, we summarize the above estimates and obtain the long-time asymptotic for the spin-l GP equation.
	
	\section{Inverse scattering transform}\label{s:2}
    \subsection{\it{Jost functions}}\label{s:2.1}

    In this section, we will focus on constructing the basic RH problem of the spin-1 GP equation \eqref{GP}.
    At the beginning of this section, we fix some notations used this work. If $I$ is an interval on the real line $\mathbb{R}$, and $X$ is a Banach space,
    then $C^0(I,X)$ denotes the space of continuous functions on $I$ taking values in $X$. It is equipped with the norm
    \begin{equation}\|f\|_{C^0(I,X)}=\sup_{x\in I}\|f(x)\|_X.\end{equation}
    If the entries $f_{1}$ and $f_{2}$ are in space $X$, then we call vector $\vec{f}=(f_1,f_2)^T$ is in space $X$ with
    $\parallel\vec{f}\parallel_X\triangleq\parallel f_1\parallel_X+\parallel f_2\parallel_X$. Similarly, if every entries of matrix $A$ are in space $X$, then we
     call $A$ is also in space X. For convenience, we introduce the notations. For any matrix function $A$, we define $|A|=\sqrt{\operatorname{tr}(A^\dagger A)}$
     and $\|A(\cdot)\|_p=\||A(\cdot)|\|_p$
    We introduce same normed spaces:\\
    \begin{enumerate}[(I)]
    \item$\text{ A weighted }L^p(\mathbb{R})\text{ space is specified by}$
      \begin{equation}L^{p,s}(\mathbb{R})=\{f(x)\in L^p(\mathbb{R})||x|^sf(x)\in L^p(\mathbb{R})\}\nonumber;
      \end{equation}
    \item$\text{ A Sobolev space is defined by}$
      \begin{equation}W^{k,p}(\mathbb{R})=\left\{f(x)\in L^p(\mathbb{R})|\left.\partial^jf(x)\in L^p(\mathbb{R})\mathrm{~for~}j=1,2,...,k\right\};\right.\nonumber
      \end{equation}
    \item$\text{ A weighted Sobolev space is defined by}$
      \begin{equation}H^{k,s}(\mathbb{R})=\left\{f(x)\in L^2(\mathbb{R})|(1+|x|^s)\partial^jf(x)\in L^2(\mathbb{R}),\text{ for }j=1,...,k\right\}\nonumber.
      \end{equation}
    \end{enumerate}
    And the norm of $f(x)\in L^p(\mathbb{R})$ and $g(x)\in L^{p,s}(\mathbb{R})$ are abbreviated to $\parallel f\parallel_p$,$\parallel g\parallel_{p,s}$ respectively.

    Throughout out of this work, we use the following notations: The complex conjugate of a complex number $k$ is denoted by $\bar{k}$. For a complex-valued matrix
    $A$, $\bar{A}$ denotes the element-wise complex conjugate, $A^{T}$ denotes the transpose, and $A^{\dagger}$ denotes the conjugate transpose. $4\times4$ matrix $A$
    is represented as four blocks:
    \begin{equation}\nonumber
    A=\begin{pmatrix}A_{11}&A_{12}&A_{13}&A_{14}\\A_{21}&A_{22}&A_{23}&A_{24}\\A_{31}&A_{32}&A_{33}&A_{34}\\A_{41}&A_{42}&A_{43}&A_{44}\end{pmatrix}=(A_1,A_2,A_3,A_4)
    =(A_\mathrm{L},A_\mathrm{R})=\begin{pmatrix}A_\mathrm{UL}&A_\mathrm{UR}\\A_\mathrm{DL}&A_\mathrm{DR}\end{pmatrix},
\end{equation}
    where $A_{ij}$ represents the $(i, j)$-entry, $A_{j}$ represents the $j$-th column, $A_{L}$ represents the first two columns, $A_{R}$ represents the last two columns,
    $A_{UL}$, $A_{UR}$, $A_{DL}$, $A_{DR}$ are $2\times2$ matrices. The notation $A(k)$, $k\in(D_{1},D_{2})$, means that $A_{L}$ and $A_{R}$ hold for $k\in D_1,D_2$,
    respectively. $I_{n}$ represents the $n\times n$ identity matrix, $0_{n}$ represents the $n\times n$ 0 matrix, and
    $\mathbb{C}^{+}=\{k\in\mathbb{C}:\operatorname{Im}k\geqslant0\}$, $\mathbb{C}^{-}=\{k\in\mathbb{C}:\operatorname{Im}k\leqslant0\}$. For a vector function
    $f(x,t;k)$, $f^{(n)}(x,t;k)=\partial_k^n f(x,t;k)$, $f^{(n)}(x,t;k_{0})=\partial_{k}^{n}f(x,t;k)|_{k=k_{0}}$.

    In this subsection, next we shall derive the basic RH problem from the Cauchy problem for the spin-1 GP equation \eqref{GP}. Let us consider a $4\times4$ matrix
    Lax pair of the spin-1 GP equation:
    \begin{equation}\label{Lax0}
    \begin{aligned}\psi_{x}=(-ik\sigma_4+U)\psi,\\\psi_{t}=(-2ik^2\sigma_4+V)\psi,\end{aligned}\end{equation}
    where $\psi$ is a matrix-valued function and $k$ is the spectral parameter, $\sigma_4=\sigma_3\otimes I_2$, $¡°\otimes¡±$ represents Kronecker product.
    \begin{equation}\nonumber
    U=\begin{pmatrix}0&0&q_1&q_0\\0&0&q_0&q_{-1}\\-\bar{q}_1&-\bar{q}_0&0&0\\-\bar{q}_0&-\bar{q}_{-1}&0&0\end{pmatrix},\end{equation}
    \begin{equation}\nonumber
    V=2kU+i\sigma_4(U_x-U^2).\end{equation}
    For convenience, the matrix U is written as the block form
    \begin{equation}\nonumber
    U=\begin{pmatrix}0&q\\-q^\dagger&0\end{pmatrix},\end{equation}
    where
    \begin{equation}\nonumber
    q=\begin{pmatrix}q_{1}&q_{0}\\q_{0}&q_{-1}\end{pmatrix}.\end{equation}
    Let $\mu=\psi e^{ik\sigma_4x+2ik^2\sigma_4t}$, where $e^{\sigma_{4}}=\mathrm{diag}(e,e,e^{-1},e^{-1})$. Then we get
    \begin{equation}\label{Lax}
    \begin{aligned}\mu_{x}=-ik[\sigma_4,\mu]+U\mu,\\\mu_{t}=-2ik^2[\sigma_4,\mu]+V\mu,\end{aligned}\end{equation}
    where $[\sigma_4,\mu]=\sigma_4\mu-\mu\sigma_4$.

    Then we have the Volterra integral equation about the matrix Jost solution $\mu_\pm $ of equation \eqref{Lax}.
    \begin{equation}\label{Vol}
    \mu_\pm(k;x,t)=I_{4\times4}+\int_{\pm\infty}^xe^{ik\sigma_4(y-x)}U(k;y,t)\mu_\pm(k;y,t)e^{-ik\sigma_4(y-x)}\mathrm{d}y,\end{equation}
    with $\mu_\pm\to I_{4}\text{ as }x\to\pm\infty$.

    Let $\mu_\pm=(\mu_{\pm L},\mu_{\pm R})$.
    \begin{proposition}\label{Pro1} The matrix Jost solutions $\mu_{\pm}(k;x,t)$ have the characters:
\begin{enumerate}[(i)]
    \item $\mu_{+L}$ and $\mu_{-R}$ is analytic in the lower complex k-plane $\mathbb{C}_{-}$;

    \item $\mu_{-L}$ and $\mu_{+R}$ is analytic in the upper complex k-plane $\mathbb{C}_{+}$;

    \item $\det\mu_\pm=1$, and $\mu\pm $ satisfy the symmetry conditions $\mu_\pm^\dagger(\bar{k})=\mu_\pm^{-1}(k)$, $\mu_\pm(x,t;k)=\tau\bar{\mu}_\pm(x,t;\bar{k})\tau$,
    where $\dagger$ denotes the Hermite conjugate, $\tau=\sigma_2\otimes I_2$
        \end{enumerate}
    \end{proposition}

    \begin{proof} The $\mu_\pm$ are written as the $2\times2$ block forms $(\mu_{\pm ij})_{2\times2}$. Through simple calculations, we get
    \begin{equation}\nonumber
    e^{ik\sigma_4(y-x)}U\mu_\pm e^{-ik\sigma_4(y-x)}=\begin{pmatrix}q\mu_{\pm21}&e^{2ik(y-x)}q\mu_{\pm22}\\-e^{-2ik(y-x)}q^\dagger\mu_{\pm11}&-q^\dagger\mu_{\pm12}
    \end{pmatrix}.\end{equation}
    Then we can get the analytic properties of the Jost solutions $\mu_{\pm}$  from the exponential term in the Volterra integral equation \eqref{Vol}.
    For example $\mu_{-22}$, $Re(2ik(y-x)=-2Imk(y-x)<0$, since $y-x<0$, so $Imk<0$, this is to say $\mu_{-R}$ is analytic in the lower complex k-plane $\mathbb{C}_{-}$.

    Since \begin{equation}\begin{aligned}
    (\det\mu_{\pm})_{x}& =\mathrm{tr}[(\mathrm{adj}\mu_\pm)(\mu_\pm)_x]  \\
    &=\mathrm{tr}\{(\mathrm{adj}\mu_\pm)[-ik(\sigma_4\mu_\pm-\mu_\pm\sigma_4)+U\mu_\pm]\} \\
    &=-ik\mathrm{tr}(\mathrm{adj}\mu_\pm)\mathrm{tr}(\sigma_4\mu_\pm-\mu_\pm\sigma_4)+\mathrm{tr}(\mathrm{adj}\mu_\pm)\mathrm{tr}(U)\mathrm{tr}(\mu_\pm) \\
    &=0,
    \end{aligned}\end{equation}
    where adj$X$ is the adjoint of matrix $X$ and tr$X$ is the trace of matrix $X$, so $\det\mu_{\pm}$ is not related to $x$, which show $\det\mu_\pm=1$.
    Note that $U$ has the symmetric relationship $U^\dagger(\bar{k})=-U(k)$, $\overline{U}=\tau U\tau$, we find that $\psi_\pm^\dagger(x,t;\bar{k})=\psi_\pm^{-1}(x,t;k)$,
    $\psi_\pm(x,t;k)=\tau\bar{\psi}_\pm(x,t;\bar{k})\tau$,  $k\in\mathbb{R}$, so we can be easily verified
    \begin{equation}
    \mu_\pm^\dagger(x,t;\bar{k})=\mu_\pm^{-1}(x,t;k), \mu_\pm(x,t;k)=\tau\bar{\mu}_\pm(x,t;\bar{k})\tau.
    \end{equation}
    \end{proof}
    \subsection{The scattering data}\label{s:2.2}
    Since $\psi_{\pm}=\mu_\pm e^{-ik\sigma_4x-2ik^2\sigma_4t}$ satisfy the same differential equation \eqref{Lax0}, they are linearly dependent. There exists a $4\times4$
     scattering matrix $S\left(k\right)$ satisfies
    \begin{equation}\label{S}
    \mu_-=\mu_+e^{-ik\sigma_4x-2ik^2\sigma_4t}S(k)e^{ik\sigma_4x+2ik^2\sigma_4t},\quad\det S(k)=1.\end{equation}
    By \eqref{Vol}, evaluation \eqref{S} at $x\to+\infty,t=0$ gives
    \begin{equation}\label{S0}
    \begin{aligned}
    {S(k)}& =\lim_{x\to+\infty}e^{ikx\sigma_4}\mu_-(x,0;k)e^{-ikx\sigma_4}  \\
    &=I_{4}+\int_{-\infty}^{+\infty}e^{iky\sigma_4}U(y,0;k)\mu_-(y,0;k)e^{-iky\sigma_4}\mathrm{d}y,
    \end{aligned}\end{equation}
    which implies that $S(k)$ can be determined by the initial data.

    From the symmetry property of $\mu_\pm$, the scattering matrix $S(k)$ has symmetry condition $S^\dagger(\bar{k})=S^{-1}(k)$, $S(k)=\tau\bar{S}(\bar{k})\tau$.
    Write $S(k)$ in $2\times2$ block form
    \begin{equation}\label{215}
    S(k)=\begin{pmatrix}a(k)&-\bar{b}(\bar{k})\\b(k)&\bar{a}(\bar{k})\end{pmatrix}
\end{equation}
    then it is easy to see that $a(k)$ and $\bar{a}(\bar{k})$ can be analytically extended to $\mathbb{C}_+$ and $\mathbb{C}_-$, respectively. Moreover,
    \begin{equation}\nonumber
    a^\dagger(\bar{k})a(k)+b^\dagger(\bar{k})b(k)=I_2,\quad a^\mathrm{T}(k)b(k)=b^\mathrm{T}(k)a(k),
\end{equation}
    \begin{equation}\nonumber
    a(k)a^\dagger(\bar{k})+\bar{b}(\bar{k})b^\mathrm{T}(k)=I_2,\quad a(k)b^\dagger(\bar{k})=\bar{b}(\bar{k})a^\mathrm{T}(k).
\end{equation}

    \eqref{S0} implies
    \begin{equation}\label{218}
    a(k)=I_2+\int_{-\infty}^{+\infty}q(x,0)\mu_{-21}(x,0;k)\mathrm{d}x,
\end{equation}
    \begin{equation}\nonumber
    b(k)=-\int_{-\infty}^{+\infty}e^{-2ikx}q^\dagger(x,0)\mu_{-11}(x,0;k)\mathrm{d}x.
\end{equation}
    Suppose that $q(x,0)\in L^1(\mathbb{R})$, $a(k)$ and $b(k)$ are well defined for $k\in\mathbb{R}$, $a(k)$ can be analytically continued onto $\mathbb{C}^+$.
     Furthermore, it follows from equation \eqref{S} and \eqref{215} that $a(k)$ and $b(k)$ can be expressed in terms of $\mu_{\pm}(x,t;k)$ or $\psi_\pm(x,t;k)$:
    \begin{subequations}\label{220}
\begin{align}
&a(k)=\mu_{+\mathrm{L}}^\dagger(\bar{k})\mu_{-\mathrm{L}}(k)=\psi_{+\mathrm{L}}^\dagger(\bar{k})\psi_{-\mathrm{L}}(k), &k\in\mathbb{C}^+\cup\mathbb{R}, \label{Z1}\\
&\det[a(k)]=\det[\mu_{-\mathrm{L}}(k),\mu_{+\mathrm{R}}(k)]=\det[\psi_{-\mathrm{L}}(k),\psi_{+\mathrm{R}}(k)],\quad&k\in\mathbb{C}^+\cup\mathbb{R},\label{Z2} \\
&\bar{a}(\bar{k})=\mu_{+\mathbb{R}}^\dagger(\bar{k})\mu_{-\mathbb{R}}(k)=\psi_{+\mathbb{R}}^\dagger(\bar{k})\psi_{-\mathbb{R}}(k), &k\in\mathbb{C}^-\cup\mathbb{R}, \label{Z3}\\
&\det[\bar{a}(k)]=\det[\mu_{+\mathrm{L}}(k),\mu_{-\mathrm{R}}(k)]=\det[\psi_{+\mathrm{L}}(k),\psi_{-\mathrm{R}}(k)],\quad&k\in\mathbb{C}^-\cup\mathbb{R},\label{Z4}\\
&b(k)=e^{-2i\theta(k)}\mu_{+\mathrm{R}}^\dagger(k)\mu_{-\mathrm{L}}(k)=\psi_{+\mathrm{R}}^\dagger(k)\psi_{-\mathrm{L}}(k), &k\in\mathbb{R},\label{Z5}\\
&-\bar{b}(\bar{k})=e^{2i\theta(k)}\mu_{+\mathrm{L}}^\dagger(k)\mu_{-\mathrm{R}}(k)=\psi_{+\mathrm{L}}^\dagger(k)\psi_{-\mathrm{R}}(k), &k\in\mathbb{R}.\label{Z6}
\end{align}
\end{subequations}

    The reflection coefficient $\gamma(k)$ is then defined by
    \begin{equation}\label{221}
    \gamma(k)=b(k)a^{-1}(k),\quad k\in\mathbb{R}.
\end{equation}
    Therefore, we have $\gamma^\mathrm{T}(k)=\gamma\left(k\right)$, $\gamma(k)\gamma^\dagger(k)=\gamma^\dagger(k)\gamma(k).$

    \subsection{\it{A basic RH problem}}\label{s:2.3}
    The high-order pole solutions of the spin-1 GP equation obtained using the RH method in \cite{boo-27}. We cited some results from \cite{boo-27}.
    Combining Proposition \ref{Pro1}, we have
    \begin{equation}\nonumber
    \psi_{\pm j_4}(x,t;k)=\mathcal{G}[\bar{\psi}_{\pm j_1}(x,t;\bar{k}),\bar{\psi}_{\pm j_2}(x,t;\bar{k}),\bar{\psi}_{\pm j_3}(x,t;\bar{k})],\quad k\in\mathbb{R},
\end{equation}
    where $(j_1,j_2,j_3,j_4)$ is an even permutation of $(1,2,3,4)$ and $¡°\mathcal{G}[\cdot]¡±$ represents the generalized cross product defined in \cite{boo-62},
    for all $u_1,u_2,u_3\in\mathbb{C}^4$,
    \begin{equation}\nonumber
    \mathcal{G}[u_1,u_2,u_3]=\sum_{j=1}^4\det(u_1,u_2,u_3,e_j)e_j,
\end{equation}
    among them, $\{e_1,e_2,e_3,e_4\}$ represents the standard base of $\mathbb{R}^4$. By direct calculation, it is easy to verify the relationship between the
    conjugate matrix $\overline{(\cdot)}$ and the generalized cross product $\mathcal{G}[\cdot]$:
    \begin{equation}\label{33}
    \bar{u}=\begin{pmatrix}-\mathcal{G}^\mathrm{T}[u_2,u_3,u_4]\\\mathcal{G}^\mathrm{T}[u_1,u_3,u_4]\\-\mathcal{G}^\mathrm{T}[u_1,u_2,u_4]\\
    \mathcal{G}^\mathrm{T}[u_1,u_2,u_3]\end{pmatrix}.
\end{equation}
    \begin{lemma}\label{lemma22}
    for $u_1,u_2,u_3 \in \mathbb{C}^4$, $\mathcal{G}[u_1,u_2,u_3] = 0$ if and only if $u_1,u_2,u_3$ are linearly correlated. In addition, $\mathcal{G}[\cdot]$ is
    also multilinear and completely antisymmetric.
    \end{lemma}

     \begin{proposition}\label{prop2.3}
     If $k_{0}$ is the zero of $\det[a(k)]$ and the multiplicity is $m+1$, then there exist $m+1$ $\text{complex-valued}$ constant $2\times2$ symmetric matrices
     $\mathbf{B}_0,\mathbf{B}_1,\ldots,\mathbf{B}_m$ with $\mathbf{B}_0$ is not equal to the zero matrix, such that for each $n\in\{0,\ldots,m\}$,
    \begin{equation}\label{34}
    \frac{[\psi_{-L}(x,t;k_0)\mathrm{adj}[a(k_0)]]^{(n)}}{n!}=\sum_{\overset{j+l=n}{j,l\geqslant0}}\frac{\psi_{+R}^{(k)}(x,t;k_0)\mathbf{B}_j}{j!l!}.
\end{equation}
    In addition,
    \begin{equation}\label{35}
    \frac{\left[\psi_{-R}(x,t;\bar{k}_0)\mathrm{adj}[\bar{a}(k_0)]\right]^{(n)}}{n!}=-\sum_{\overset{j+l=n}{j,l\geqslant0}}\frac{\psi_{+L}^{(l)}(x,t;\bar{k}_0)
    \mathbf{B}_j^\dagger}{j!l!}.
\end{equation}
\end{proposition}
    \begin{proof} From the symmetric of $\psi_\pm(x,t;k)$, it follows that for $k\in\mathbb{R}$, we have
    \begin{equation}\nonumber
    \psi_{\pm\mathrm{L}}^\dagger(x,t;\bar{k})\psi_{\pm\mathrm{R}}(x,t;k)=0.
\end{equation}
    The equation can be analytically extended to the upper complex plane $\mathbb{C}^{+}$. Combining this with equation \eqref{220}, we get
    \begin{equation}\label{37}
    (\psi_{{+L}}(\bar{k}),\psi_{{-R}}(\bar{k})\bar{a}^{-1}(k))^{\dagger}(\psi_{{-L}}(k)a^{-1}(k),\psi_{{+R}}(k))=I_{4}.
\end{equation}
    Moreover, we have for $k\in\mathbb{R}$,
    \begin{equation}\nonumber
    \det(\psi_{+L}(\bar{k}),\psi_{-R}(\bar{k})\bar{a}^{-1}(k))=\det(\psi_{-L}(k)a^{-1}(k),\psi_{+R}(k))=1,
\end{equation}
    which also can be analytically continued onto the upper complex plane $\mathbb{C}^{+}$.

    We define
    \begin{equation}\nonumber
    \psi_{{-L}}(x,t;k)\mathrm{adj}[a(k)]=\left(\chi_{1}(x,t;k),\chi_{2}(x,t;k)\right).
\end{equation}
    Combine equation \eqref{33} with equation \eqref{37} to get
    \begin{equation}\label{310}
    \det[a(k)]\bar{\psi}_{{+L}}(\bar{k})=(-\mathcal{G}[\chi_{2}(k),\psi_{+3}(k),\psi_{+4}(k)],\mathcal{G}[\chi_{1}(k),\psi_{+3}(k),\psi_{+4}(k)]).
    \end{equation}
    We will prove this by induction of $n$. For the basic case $n=0$, we can see from equation \eqref{310} that the vector $\chi_1(x,t;k_0),\psi_{+3}(x,t;k_0)$
    and $\psi_{+4}(x,t;k_0)$ are linearly dependent. The fact $\mathrm{rank}(\psi_{+R}(x,t;k)) = 2$ means that there must be a non-zero complex valued constant
    vector $\alpha_{0}$ such that $\chi_1(x,t;k_0) = \psi_{+R}(x,t;k_0)\alpha_0$. Now, suppose it is true for $0\leqslant n\leqslant j-1$ that there exists the
    complex valued constant vector $\alpha_0,\alpha_1,\ldots,\alpha_{j-1}$ such that for every $n\in\{0,\ldots,j-1\}$,
    \begin{equation}\label{311}
    \frac{\chi_1^{(n)}(k_0)}{n!}=\sum_{\overset{r+s=n}{r,s\geqslant0}}\frac{\psi_{+R}^{(s)}(k_0)\alpha_r}{r!s!},
\end{equation}
    where $(x,t$)-dependence are omitted for brevity. We need to prove that this statement is also true for $n=j$. Recall that
    $\frac{\partial^l}{\partial k^l}\det[a(k)]|_{k=k_0}=0,l=0,\ldots,j$ and equation \eqref{310}, we find
    \begin{equation}\label{312}
    \sum_{\overset{{v+l+s=j}}{v,l,s\geqslant0}}\frac{j!}{v!l!s!}\mathcal{G}\left[\chi_1^{(s)}(k_0),\psi_{+3}^{(v)}(k_0),\psi_{+4}^{(l)}(k_0)\right]=0.
\end{equation}
    Substituting equation \eqref{311} into equation \eqref{312} yields

\begin{equation}\nonumber
\begin{aligned}\mathbf{0}=&\mathcal{G}\left[\chi_{1}^{(j)}(\lambda_{0}),\psi_{+3}(\lambda_{0}),\psi_{+4}(\lambda_{0})\right]\\
&\begin{aligned}&+\sum_{\substack{k+l+r+s=j\\
  r+s\neq j\\
k,l,r,s\geqslant 0}}\frac{j!}{k!l!r!s!}\mathcal{G}\left[\psi_{+\mathrm{R}}^{(s)}(\lambda_0)\alpha_r,\psi_{+3}^{(k)}(\lambda_0),\psi_{+4}^{(l)}(\lambda_0)
\right]\end{aligned}\\
=&\begin{aligned}\mathcal{G}
\left[\chi_{1}^{(j)}(\lambda_{0}),\psi_{+\mathrm{R}}(\lambda_{0})\right]&+\Big(\sum_{\substack{ k+l+r+s=j\\
  r+s\neq j\\
  k+r\neq j\\
   l+r\neq j\\
 k,l,r,s\geqslant 0}}+\sum_{\substack{ k+l+r+s=j\\
  r+s\neq j\\
  k+r= j\\
 k,l,r,s\geqslant 0}}+\sum_{\substack{ k+l+r+s=j\\
  r+s\neq j\\
  l+r= j\\
 k,l,r,s\geqslant 0}}\Big)\end{aligned}\\
&\frac{j!}{k!l!r!s!}\mathcal{G}
\left[\psi_{+\mathrm{R}}^{(s)}(\lambda_{0})\alpha_{r},\psi_{+3}^{(k)}(\lambda_{0}),\psi_{+4}^{(l)}(\lambda_{0})\right]
\end{aligned}\end{equation}

\begin{equation}\nonumber
\begin{aligned}&=\mathcal{G}\left[\chi_{1}^{(j)}(\lambda_{0}),\psi_{+\mathrm{R}}(\lambda_{0})\right]
+\sum_{\begin{array}{c}k+r=j\\k>0,r\geqslant0\end{array}}
\frac{j!}{k!r!}\mathcal{G}\left[\psi_{+\mathrm{R}}(\lambda_{0})\alpha_{r},\psi_{+3}^{(k)}(\lambda_{0}),\psi_{+4}(\lambda_{0})\right]\\
&+\sum_{\binom{l+r=j}{l>0,r\geqslant0}}\frac{j!}{l!r!}\mathcal{G}\left[\psi_{+\mathrm{R}}(\lambda_{0})\alpha_{r},\psi_{+3}(\lambda_{0}),\psi_{+4}^{(l)}(\lambda_{0})\right]\\
&\mathcal{G}\left[\chi_{1}^{(j)}(\lambda_{0}),\psi_{+\mathrm{R}}(\lambda_{0})\right]-\sum_{\begin{array}{c}k+r=j\\k>0,r\geqslant0\end{array}}
\frac{j!}{k!r!}\mathcal{G}\left[\psi_{+\mathrm{R}}^{(k)}(\lambda_{0})\alpha_{r},\psi_{+\mathrm{R}}(\lambda_{0})\right]\\
=&\mathcal{G}\left[\chi_{1}^{(j)}(\lambda_{0})-\sum_{\begin{array}{c}k+r=j\\k>0,r\geqslant0\end{array}}
\frac{j!}{k!r!}\psi_{+\mathrm{R}}^{(k)}(\lambda_{0})\alpha_{r},\psi_{+\mathrm{R}}(\lambda_{0})\right]\end{aligned}\end{equation}

    Since $\mathrm{rank}(\psi_{+R}(k_0)) = 2$, lemma \ref{lemma22} shows that there is a constant vector $\alpha_{j}$, such that
    \begin{equation}\nonumber
    \chi_1^{(j)}(k_0)-\sum_{\substack{v+r=j\\v>0,r\geqslant0}}\frac{j!}{v!r!}\psi_{+\mathrm{R}}^{(v)}(k_0)\alpha_r=\psi_{+R}(k_0)\alpha_j,
\end{equation}
    it is
    \begin{equation}\nonumber
    \chi_1^{(j)}(k_0)=\sum_{\substack{v+r=j\\v,r\geqslant0}}\frac{j!}{v!r!}\psi_{+\mathrm{R}}^{(v)}(k_0)\alpha_r.
\end{equation}
    By induction hypothesis, it is proved that the complex valued constant vectors $\alpha_0,\alpha_1,\ldots,\alpha_m$ such that for every $n\in\{0,\ldots,m\}$,
    \begin{equation}\nonumber
    \frac{\chi_1^{(n)}(k_0)}{n!}=\sum_{\substack{r+s=n\\r,s\geqslant0}}\frac{\psi_{+\mathrm{R}}^{(s)}(k_0)\alpha_r}{r!s!}.
\end{equation}
    Let $\mathbf{B}_n=(\alpha_n,\beta_n),n=0,\ldots,m,$ we obtain equation \eqref{34}.

    According to equation \eqref{37},
    \begin{equation}\nonumber
    (\psi_{-L}(k)a^{-1}(k),\psi_{+R}(k))(\psi_{+L}(\bar{k}),\psi_{-R}(\bar{k}){\bar{a}}^{-1}(k))^{\dagger}=I_{4},\quad k\in\mathbb{C}^{+}\cup\mathbb{R}.
\end{equation}
    In other form
    \begin{equation}\label{318}
    (\psi_{{-L}}(k)\mathrm{adj}[a(k)],\psi_{{+R}}(k))\begin{pmatrix}\psi_{{+L}}^{\dagger}(\bar{k})\\\mathrm{adj}[a^{{T}}(k)]\psi_{{-R}}^{\dagger}(\bar{k})
    \end{pmatrix}=\mathrm{det}[a(k)]I_{4},\quad k\in\mathbb{C}^{+}\cup\mathbb{R}.
\end{equation}
    For $n=0,\ldots,m,$ the equation \eqref{318} is evaluated at $k_0$ by differentiating both sides $n$ times we get
    \begin{equation}\nonumber
    \sum\limits_{\substack{v+l=n\\v,l\geqslant0}}\frac{1}{v!l!}(\psi_{-\mathrm{L}}(k_0)\mathrm{adj}[a(k_0)], \psi_{+\mathrm{R}}(k_0))^{(v)}
    \begin{pmatrix}\psi_{+\mathrm{L}}^\dagger(\bar{k}_0)\\\mathrm{adj}[a^\mathrm{T}(k_0)]\psi_{-R}^\dagger(\bar{k}_0)\end{pmatrix}^{(l)}=0.
\end{equation}
    When $n = 0$, it is obtained by combining equation \eqref{34}
    \begin{equation}\nonumber
    \psi_{+R}(k_0)\left[\mathbf{B}_0\psi_{+L}^\dagger(\bar{k}_0)+\mathrm{adj}[a^\mathrm{T}(k_0)]\psi_{-R}^\dagger(\bar{k}_0)\right]=0.
\end{equation}
    Since $\mathrm{rank}(\psi_{+R}(k_0))=2$, we can deduce that
    \begin{equation}\nonumber
    \mathbf{B}_0\psi_{+L}^\dagger(\bar{k}_0)+\mathrm{adj}[a^\mathrm{T}(k_0)]\psi_{-R}^\dagger(\bar{k}_0)=0,
\end{equation}
    which is exactly equation \eqref{35} for $n = 0$. By induction, we prove that equation \eqref{35} holds for $n=0,\ldots,m.$
    \end{proof}

    \begin{corollary}\label{24}
    Suppose $k_{0}$ is the zero of $\det[a(k)]$ with multiple gravity $m+1,$ then for every $n\in\{0,\ldots,m\}$:
    \begin{equation}\nonumber
    \frac{[\mu_{-L}(x,t;k_0)\mathrm{adj}[a(k_0)]]^{(n)}}{n!}=\sum_{\substack{j+v+l=n\\j,v,l\geqslant0}}\frac{\Theta^{(v)}(x,t;k_0)\mu_{+R}^{(l)}(x,t;k_0)\mathbf{B}_j}{j!v!l!},
\end{equation}
    \begin{equation}\label{323}
    \frac{\left[\mu_{-R}(x,t;\bar{k}_0)\mathrm{adj}[\bar{a}(k_0)]\right]^{(n)}}{n!}=-\sum_{\substack{j+v+l=n\\j,v,l\geqslant0}}
    \frac{\overline{\Theta^{(v)}(x,t;k_0)}\mu_{+L}^{(l)}(x,t;\bar{k}_0)\mathbf{B}_j^\dagger}{j!v!l!},
\end{equation}
    where $\Theta(x,t;k)=e^{2i\theta(x,t;k)}$, and $\mathbf{B}_0,\mathbf{B}_1,\ldots,\mathbf{B}_m$ are given in Proposition \eqref{prop2.3}.
    \end{corollary}

    \begin{proof} For simplicity, we omit the $(x,t)\text{-dependence}$. It is derived from Proposition \ref{prop2.3} that

    \begin{align}
&\frac{\mu_{-L}(k_{0})\mathrm{adj}[a(k_{0})]]^{(n)}}{n!}=\frac{[\Theta^{\frac{1}{2}}(k_{0})\psi_{-L}(k_{0})\mathrm{adj}[a(k_{0})]]}{n!}\nonumber \\
&=\sum_{\overset{r+s=n}{r,s\geqslant0}}\frac{(\Theta^{\frac12})^{(r)}(k_0)[\psi_{-L}(k_0)\mathrm{adj}[a(k_0)]]^{(s)}}{r!s!}\nonumber \\
&=\sum_{\overset{r+s=n}{r,s\geqslant0}}\sum_{\overset{j+m=s}{j,m\geqslant0}}\frac{(\Theta^{\frac12})^{(r)}(k_0)\psi_{+\mathrm{R}}^{(m)}(k_0)\mathbf{B}_j}{r!j!m!}\nonumber \\
&=\sum_{\overset{r+j+m=n}{r,j,m\geqslant0}}\frac{(\Theta^{\frac12})^{(r)}(k_0)(\Theta^{\frac12}\mu_{+\mathrm{R}})^{(m)}(k_0)\mathbf{B}_j}{r!j!m!} \nonumber\\
&=\sum_{\overset{r+j+h+l=n}{r,j,h,l\geqslant0}}\frac{(\Theta^{\frac12})^{(r)}(k_0)(\Theta^{\frac12})^{(h)}(k_0)\mu_{+\mathrm{R}}^{(l)}(k_0)\mathbf{B}_j}{r!j!h!l!}\nonumber \\
&=\sum_{\overset{j+v+l=n}{j,v,l\geqslant0}}\sum_{\overset{r+h=v}{r,h\geqslant0}}\frac{(\Theta^{\frac12})^{(r)}(k_0)(\Theta^{\frac12})^{(h)}(k_0)}
{r!h!}\frac{\mu_{+R}^{(l)}(k_0)\mathbf{B}_j}{j!l!}\nonumber \\
&=\sum_{\overset{j+v+l=n}{j,v,l\geqslant0}}\frac{\Theta^{(v)}(k_0)\mu_{+R}^{(l)}(k_0)\mathbf{B}_j}{j!v!l!}\nonumber.
\end{align}

    Similarly, we can derive equation \eqref{323}.
    \end{proof}

    Let $k_0$ be the zero of $\det[a(k)]$ with multiplicity $m+1$, then $\frac1{\det[a(k)]}$ has a Laurent series expansion at $k=k_0$,
    \begin{equation}\nonumber
    \frac1{\det[a(k)]}=\frac{a_{-m-1}}{(k-k_0)^{m+1}}+\frac{a_{-m}}{(k-k_0)^m}+\cdots+\frac{a_{-1}}{k-k_0}+O(1),\quad k\to k_0,
\end{equation}
    where $a_{-m-1}\neq0$ and $a_{-n-1}=\frac{\tilde{a}^{(m-n)}(k_0)}{(m-n)!}$, $\tilde{a}(k)=\frac{(k-k_0)^{m+1}}{\det[a(k)]}$, $n=0,\ldots,m$. In conjunction with
     Corollary \ref{24}, it follows that for every $n\in\{0,\ldots,m\}$,
    \begin{equation}\nonumber
    \begin{aligned}&\operatorname*{Res}_{k_0}(k-\bar{k}_0)^n\mu_{-L}(x,t;k)a^{-1}(k)\\&=\sum_{j+v+l+s=m-n\atop j,v,l,s\geqslant0}\frac{\tilde{a}^{(j)}(k_0)\Theta^{(l)}(x,t;k_0)
    \mu_{+\mathrm{R}}^{(s)}(x,t;k_0)\mathbf{B}_v}{j!v!l!s!},\end{aligned}
\end{equation}
    \begin{equation}\nonumber
    \begin{aligned}&\operatorname*{Res}_{\bar{k}_0}(k-\bar{k}_0)^n\mu_{-R}(x,t;k)\bar{a}^{-1}(\bar{k})\\&=-\sum_{j+v+l+s=m-n\atop j,v,l,s\geqslant0}
    \frac{\overline{\tilde{a}^{(j)}(k_0)\Theta^{(l)}(x,t;k_0)}\mu_{+L}^{(s)}(x,t;\bar{k}_0)\bar{\mathbf{B}}_v}{j!v!l!s!}.\end{aligned}
\end{equation}
    Introduce a symmetric matrix-valued polynomial of up to m degrees, expressed as:
    \begin{equation}\nonumber
    f_0(k)=\sum_{h=0}^m\sum_{j+k=h\atop j,v\geqslant0}\frac{\tilde{a}^{(j)}(k_0)\mathbf{B}_v}{j!v!}(k-k_0)^h.
\end{equation}
    It is evident that $f_0(k_0)\neq0$, therefore,
    \begin{equation}\label{329}
    \begin{aligned}
\text{Res }(k-k_0)^n\mu_{-\text{L}}(x,t;k)a^{-1}(k)& =\sum_{h+l+s=m-n\atop h,l,s\geqslant0}\frac{\Theta^{(l)}(x,t;k_0)\mu_{+\mathrm{R}}^{(s)}(x,t;k_0)f_0^{(h)}(k_0)}{h!l!s!} \\
&=\frac{[e^{2i\theta(x,t;k)}\mu_{+\mathrm{R}}(x,t;k)f_0(k)]^{(m-n)}|_{k=k_0}}{(m-n)!}.
\end{aligned}
\end{equation}
    Furthermore,
    \begin{equation}\label{330}
    \mathrm{Res~}(k-\bar{k}_0)^n\mu_{-\mathrm{R}}(x,t;k)\bar{a}^{-1}(\bar{k})=-\frac{[e^{-2i\theta(x,t;k)}\mu_{+L}(x,t;k)f_0^\dagger(\bar{k})]^{(m-n)}|_{k=\bar{k}_0}}{(m-n)!}.
\end{equation}
    We shall name $f_0(k_0),\ldots,f_0^{(m)}(k_0)$ as the residue constants at the discrete spectrum $k_0$.

\begin{assumption}\label{assumption25}
    The initial data $q_0^0{(x)}$, $q_1^0{(x)}$ and $q_{-1}^0{(x)}$ for the Cauchy problem for the spin-1 GP equation \eqref{GP} generates generic scattering data
    in the sense that:
    \begin{enumerate}[(i)]
    			\item There are no spectral singularities, i.e., there exist a constant $c>0$ such that $|a(k)|\geq c$ for any $k\in\mathbb{R}$;
    			
                \item The discrete spectrum is simple, i.e., every zero of $a(k)$ in $\mathbb{C}^+$ is simple.

    		\end{enumerate}

    \end{assumption}
    \begin{proposition}
     There exists a unique symmetric matrix-valued polynomial $f(k)$ whose degree is less than $\mathcal{N}=\sum_{j=1}^N(m_j+1)$ and has the property $f(k_j)\neq0$ such
      that when $j=1,\ldots,N,n_j=0,\ldots,m_j,$;
    \begin{equation}\label{331}
    \operatorname*{Res}_{k_j}\left(k-k_j\right)^{n_k}\mu_{-L}(x,t;k)a^{-1}(k)=\frac{[e^{2i\theta(x,t;k)}\mu_{+R}(x,t;k)f(k)]^{(m_j-n_j)}|_{k=k_j}}{(m_j-n_j)!},
\end{equation}
\begin{equation}\label{332}
    \operatorname*{Res}_{{\bar{k}_{j}}} (k-\bar{k}_{j})^{n_{k}}\mu_{-\mathrm{R}}(x,t;k)\bar{a}^{-1}(\bar{k})=
    -\frac{[e^{-2i\theta(x,t;k)}\mu_{+L}(x,t;k)f^{\dagger}(\bar{k})]^{(m_{j}-n_{j})}|_{k=\bar{k}_{j}}}{(m_{j}-n_{j})!}.
\end{equation}
    \end{proposition}
    \begin{proof} It's similar to equation \eqref{329} and \eqref{330}, for each $j\in\{1,\ldots,N\}$, a not greater than the value of $m_{j}$ symmetric matrix
    polynomial $f(k)$, including $f(k_j)\neq0$, makes when $n_j=0,\ldots,m_j$;
    \begin{equation}\label{333}
    \operatorname*{Res}_{k_j}\left(k-k_j\right)^{n_j}\mu_{-L}(x,t;k)a^{-1}(k)=\frac{[e^{2i\theta(x,t;k)}\mu_{+\mathrm{R}}(x,t;k)f(k_j)]^{(m_j-n_j)}|_{k=k_j}}{(m_j-n_j)!},
\end{equation}
\begin{equation}\label{334}
    \operatorname*{Res}_{{\bar{k}_{j}}}(k-\bar{k}_{j})^{{n_{j}}}\mu_{{-R}}(x,t;k)\bar{a}^{-1}(\bar{k})=
    -\frac{[e^{{-2i\theta(x,t;k)}}\mu_{{+L}}(x,t;k)f^{\dagger}(\bar{k_{j}})]^{{(m_{j}-n_{j})}}|_{{k=\bar{k}_{j}}}}{(m_{j}-n_{j})!}.
\end{equation}
    Using Hermite interpolation formula, there exists a unique symmetric matrix value polynomial $f(k)$ with degree less than $N$, which makes
    \begin{equation}\nonumber
    \begin{cases}f^{(n_1)}(k_1)=f_1^{(n_1)}(k_1),&n_1=0,\ldots,m_1,\\\vdots\\f^{(n_N)}(k_N)=f_N^{(n_N)}(k_N),&n_N=0,\ldots,m_N.\end{cases}
\end{equation}
    Hence, we have proven equation \eqref{331} and \eqref{332}.
    \end{proof}

    Let
    \begin{equation}\label{M}
    M(k;x,t)=\begin{cases}\left(\mu_{-L}(k)a^{-1}(k),\mu_{+R}(k)\right),&k\in\mathbb{C}_+,\\\left(\mu_{+L}(k),\mu_{-R}(k)\bar{a}^{-1}(\bar{k})\right),&k\in\mathbb{C}_-.
    \end{cases}\end{equation}

    We have $M(k;x,t)$ is analytic for $k\in\mathbb{C}\backslash\mathbb{R}$, there we only consider that the determinant of the matrix $a(k)$ has $N$ simple zeros,
     denoted by $k_1,\ldots,k_N$ , which are all in the upper half of the complex plane $\mathbb{C}^+$ and are not on the real axis. So we can get
    \begin{theorem}
     The piecewise-analytic function $M(k;x,t)$ determined by \eqref{M} satisfies the following RH problem \ref{RHP1}.
    \begin{RHP}\label{RHP1}
    		Find a matrix valued function $M(k)$ admits:
    		\begin{enumerate}[(i)]
    			\item Analyticity:~$M(k;x,t)$ is analytic in $k\in\mathbb{C}\setminus(\mathbb{R}\cup\mathcal{Z}\cup\bar{\mathcal{Z}})$, $\mathcal{Z}=\{k_j\}_{j=1}^N$;
    			\item Jump condition:
    			\begin{equation}M_+(k;x,t)=M_-(k;x,t)J(k;x,t),\quad k\in\mathbb{R},\end{equation}
    		\begin{equation}J=\begin{pmatrix}I_{2\times2}+\gamma^\dagger(\bar{k})\gamma(k)&\gamma^\dagger(\bar{k})e^{-2it\theta}\\\gamma(k)e^{2it\theta}&I_{2\times2}\end{pmatrix}
    ,\end{equation}
              where $\theta(k)=\frac xtk+2k^2,$ $\gamma(k)=b(k)a^{-1}(k)$.
                \item Residue conditions: $M(k;x,t)$ has simple poles at each point in $\mathcal{Z}\cup\bar{\mathcal{Z}}$

             \begin{equation}
    \operatorname*{Res}_{k=k_j}M(k)=\lim_{k\to k_j}M(k)\begin{pmatrix}0&0\\f(k_j)e^{2it\theta(k_j)}&0\end{pmatrix};
\end{equation}
                \begin{equation}
    \operatorname*{Res}_{k{=}\bar{k}}M(k)=\lim_{k\to \bar{k_j}}M(k)\begin{pmatrix}0&-f^\dagger(k_j)e^{-2it\theta(\bar{k_j})}\\0&0\end{pmatrix};
                 \end{equation}
    			\item Asymptotic behavior:
    \begin{align}
    				M(k;x,t)\rightarrow{I_4}~~as~~k\rightarrow\infty,
    			\end{align}
    		\end{enumerate}
    	\end{RHP}
     We have following reconstruction formula
    \begin{equation}\label{q1}
    q(x,t)=2i\lim_{k\to\infty}(kM(k;x,t))_{UR}\end{equation}
    \end{theorem}
   \begin{proof} By combining $M(k;x,t)$ and scattering relationship \eqref{S}, we can obtain the jump condition through simple calculations. According to the Vanishing
   Lemma, we found jump matrix $J(k;x,t)$ is positively definite, so the solution of the RH problem \ref{RHP1} is existent and unique. Notice that $M(k;x,t)$ has the
   asymptotic expansion
   \begin{equation}M(k;x,t)=I_{4\times4}+\frac{M_1(x,t)}k+\frac{M_2(x,t)}{k^2}+O(k^{-3}).\end{equation}
   From the asymptotic behavior of the functions $\mu_\pm(k)$ and $S(k)$, we have reconstruction formula, details for \cite{boo-27}
   \end{proof}

    The long-time asymptotic of RH problem \ref{RHP1} is affected by the growth and decay of the exponential function $e^{-2it\theta}$ appearing in the jump relation.
     So we need control the real part of $-2it\theta$. We introduce a new transform $M(k) \to M^{(1)}(k)$, which make that the $M^{(1)}(k)$ is well behaved as
      $t\to\infty$ along any characteristic line. Note that $\xi=x/t$, it is worth noting that the stationary point $k_0=-x/(4t)=-\frac{\xi}{4}$, which makes
       $\mathrm{d}\theta/\mathrm{d}k=\xi+4k=0$. To obtain asymptotic behavior of $e^{-2it\theta}$ as $t\to\infty$, we consider the real part of $-2it\theta$:
   \begin{equation}\label{sss}
   Re(2it\theta)=-8tImk(Re(k)-k_{0}). \end{equation}
   In order to use the $\bar{\partial}\text{-generalization}$ of the $\text{Deift-Zhou's}$ steepest descent method, we deform the contour of the RH problem and
    our main goal is to construct a model RH problem. After reorientation and extending, we obtain the long-time asymptotics of the solution to the Cauchy problem of
     the spin-1 GP equation \eqref{GP}.

   \centerline{\begin{tikzpicture}[scale=0.7]
\path [fill=pink] (0,4)--(0,0) to (4,0) -- (4,4);
\path [fill=pink] (0,-4)--(0,0) to (-4,0) -- (-4,-4);
\draw[-][thick](-4,0)--(-3,0);
\draw[-][thick](-3,0)--(-2,0);
\draw[-][thick](-2,0)--(-1,0);
\draw[-][thick](-1,0)--(0,0);
\draw[-][thick](0,0)--(1,0);
\draw[-][thick](1,0)--(2,0);
\draw[-][thick](2,0)--(3,0);
\draw[->][thick](3,0)--(4,0)[thick]node[right]{$Rek$};
\draw[<-][thick](0,4)[thick]node[right]{$Imk$}--(0,3);
\draw[-][thick](0,3)--(0,2);
\draw[-][thick](0,2)--(0,1);
\draw[-][thick](0,1)--(0,0);
\draw[-][thick](0,0)--(0,1);
\draw[-][thick](0,1)--(0,2);
\draw[-][thick](0,2)--(0,3);
\draw[-][thick](0,3)--(0,4);
\draw[-][thick](0,0)--(0,-1);
\draw[-][thick](0,-1)--(0,-2);
\draw[-][thick](0,-2)--(0,-3);
\draw[-][thick](0,-3)--(0,-4);
\draw[fill] (2,1.5)node[below]{\small{$|e^{it\theta(k)}|\rightarrow0$}};
\draw[fill] (-2.5,-1.5)node[below]{\small{$|e^{it\theta(k)}|\rightarrow0$}};
\draw[fill] (2,-1.5)node[below]{\small{$|e^{-it\theta(k)}|\rightarrow0$}};
\draw[fill] (-2.5,1.5)node[below]{\small{$|e^{-it\theta(k)}|\rightarrow0$}};
\draw[fill] (0,0)node[below right]{\small{$k_{0}$}};
\label{Figure 1.}
\end{tikzpicture}}
\centerline{\noindent {\small \textbf{Figure 2.3} Exponential decaying domains.}}

   \section{Asymptotic analysis in the region $\xi\neq0$}\label{s:3}

    \subsection{\it{Modifications to the basic RH problem}}\label{s:3.1}
   First of all, based on the decay regions as shown in Figure \ref{Figure 1.}, we need to decompose the jump matrix $J(k)$ into upper and lower triangular matrix.
   We therefore introduce a scalar function $\delta_1(k)$ and $\delta_2(k)$ satisfying the RH problem as follows.
   \begin{equation}\label{k1}
   \begin{cases}&\delta_{1+}(k)=\delta_{1-}(k)(I_{2\times2}+\gamma(k)\gamma^\dagger(\bar{k})),\quad k\in(-\infty,k_0),\\&\delta_1(k)\to I_{2\times2},\quad k\to\infty,
   \end{cases}\end{equation}
   and
   \begin{equation}\label{k2}
   \begin{cases}&\delta_{2+}(k)=(I_{2\times2}+\gamma^\dagger(\bar{k})\gamma(k))\delta_{2-}(k),\quad k\in(-\infty,k_0),\\&\delta_2(k)\to I_{2\times2},\quad k\to\infty.
   \end{cases}\end{equation}
   The solutions of the above two RH problems exist and are unique because of Vanishing Lemma \cite{boo-59}. From the uniqueness, we have the symmetry relations
   $\delta_j^{-1}(k)=\delta_j^\dagger(\bar{k}),(j=1,2)$. Then a simple calculation shows that for $j=1,2$,
   \begin{equation}\nonumber
   |\delta_{j+}|^2=\begin{cases}2+|\gamma(k)|^2,\quad k\in(-\infty,k_0),\\2,\quad k\in(k_0,+\infty),&\end{cases}\end{equation}
   \begin{equation}\nonumber
   \left|\delta_{j-}\right|^2=\begin{cases}2-\frac{2|\det\gamma(k)|^2+|\gamma(k)|^2}{1+|\det\gamma(k)|^2+|\gamma(k)|^2},\quad k\in(-\infty,k_0),\\2,\quad k\in(k_0,+\infty).
   &\end{cases}\end{equation}
   Hence, by the maximum principle, we have for $j=1,2$,
   \begin{equation}\nonumber
   |\delta_j(k)|\leqslant\mathrm{const}<\infty,\quad k\in\mathbb{C}.\end{equation}
   However, $\delta_1(k)$ and $\delta_2(k)$ can not be found in explicit form because they satisfy the matrix RH problems \eqref{k1} and \eqref{k2}.
   If we consider the determinants of the two RH problems, they become the same scalar RH problem
   \begin{equation}\nonumber
   \begin{cases}\det\delta_+(k)=(1+|\gamma(k)|^2+|\det\gamma(k)|^2)\det\delta_-(k),\quad k\in(-\infty,k_0),\\\det\delta(k)\to1,\quad k\to\infty.\end{cases}\end{equation}
   which can be solved by the Plemelj formula \cite{boo-59}
   \begin{equation}\nonumber
   \det\delta(k)=(k-k_0)^{i\nu}e^{\chi(k)},\end{equation}
   where
   \begin{equation}\nonumber
   \nu=-\frac1{2\pi}\log(1+\left|\gamma(k_0)\right|^2+\left|\det\gamma(k_0)\right|^2),\nonumber\end{equation}
   \begin{equation}\nonumber
   \begin{gathered}
   \chi\left(k\right) =\frac{1}{2\pi i}\Big(\int_{k_0-1}^{k_0}\log\left(\frac{1+|\gamma(\xi)|^2+|\det\gamma(\xi)|^2}{1+|\gamma(k_0)|^2+|\det\gamma(k_0)|^2}\right)
   \frac{\mathrm{d}\xi}{\xi-k} \\
   +\int_{-\infty}^{k_0-1}\log\left(1+\left|\gamma\left(\xi\right)\right|^2+\left|\det\gamma\left(\xi\right)\right|^2\right)\frac{\mathrm{d}\xi}{\xi-k} \\
   -\log(1+|\gamma(k_0)|^2+|\det\gamma(k_0)|^2)\log(k-k_0+1)\Big).\nonumber
   \end{gathered}\end{equation}
   For brevity, we denote
   \begin{equation}\nonumber
    \Delta_{k_{0}}^{+}=\{j\in\{1,\cdots,N\} |Re(k_{j})>k_{0}\}, \Delta_{k_{0}}^{-}=\{j\in\{1,\cdots,N\} |Re(k_{j})<k_{0}\},\\\end{equation}
    for the real interval $\mathcal I=[a,b]$, define
    \begin{equation}\nonumber
    \begin{gathered}
    \mathcal{Z}(\mathcal{I})=\{k_{j}\in\mathcal{Z}:Rek_{j}\in\mathcal{I}\},\\\mathcal{Z}^{-}(\mathcal{I})=\{k_{j}\in\mathcal{Z}:Rek_{j}<a\},\\
    \mathcal{Z}^{+}(\mathcal{I})=\{k_{j}\in\mathcal{Z}:Rek_{j}>b\}.
    \end{gathered}\end{equation}
   For $k_0\in\mathcal{I}$, define
   \begin{equation}\nonumber
   \begin{gathered}
    \Delta_{k_0}^-(\mathcal{I})=\{j\in\{1,2,\cdots,N\}:a\leqslant Rek_{j}<k_0\}, \\
    \Delta_{k_0}^+(\mathcal{I})=\{j\in\{1,2,\cdots,N\}:k_0<Rek_j\leqslant b\}.
   \end{gathered}\end{equation}
   Let's introduce the function
   \begin{equation}\nonumber
    T_{j}(k)=\prod_{j\in\Delta_{k_0}^-}\frac{k-\bar{k}_j}{k-k_j}\delta_{j}(k),T^{'}(k)=
    \prod_{j\in\Delta_{k_0}^-}\frac{k-\bar{k}_j}{k-k_j}\det\delta(k),T_0(k_0)=\prod_{j\in\Delta_{k_0}^-}\left(\frac{k_0-\bar{k}_j}{k_0-k_j}\right)e^{\chi(k_{0})}.
    \end{equation}
    \begin{proposition} The matrix function $T_{j}(k)$ and scalar function $T_0(k_0)$ satisfy the following properties:
    \begin{enumerate}[(i)]
    			\item $T_{j}(k)$ is analytic in $\mathbb{C}\setminus(-\infty,k_0]$;
    			\item for $\mathbb{C}\setminus(-\infty,k_0]$, $T_{j}(k)(T_{j})^\dagger(\bar{k}) = I$;
                \item for $k\in(-\infty,k_0]$, $T_{1+}(k)=T_{1-}(k)(I_{2\times2}+\gamma(k)\gamma^\dagger(\bar{k}))$,
                $T_{2+}(k)=(I_{2\times2}+\gamma^\dagger(\bar{k})\gamma(k))T_{2-}(k)$;
    			\item for $|k|\to\infty, |\arg(k)|\leqslant c<\pi,$
        \begin{equation}\nonumber
              T^{'}(k)=1+\frac ik2\sum_{j\in\Delta_\xi^-}Imk_k-\frac1{2ki\pi}\int_{-\infty}^{k_0}\log({1+|\gamma(s)|^2+|\det\gamma(s)|^2})ds+O(k^2);
           \end{equation}
                \item Along the ray $k=k_0+e^{i\phi}\mathbb{R}_+, |\phi|\leqslant c<\pi,$
     \begin{equation}\nonumber
    |T^{'}(k)-T_0(k_0)(k-k_0)^{i\nu(k_0)}|\leqslant C\parallel \gamma\parallel_{H^1(\mathbb{R})}|k-k_0|^{1/2},\quad k\to k_0.
     \end{equation}
    		\end{enumerate}

     \end{proposition}
   \begin{proof}Properties (i), (ii), (iii) and (iv) can be obtain by simple calculation from the definition of $T_{j}(k)$, $T^{'}(k)$ and $T_0(k_0)$. For (v), via fact that
   \begin{equation}\nonumber
    |(k-k_0)^{i\nu(k_0)}|=|e^{i\nu\log|k-k_0|-\nu\arg(k-k_0)}|\leqslant e^{-\pi\nu(k_0)}=\sqrt{1+|\gamma(k_0)|^2},
    \end{equation}
       it adduces that
       \begin{equation}\nonumber
    |\chi(k,k_{0})-\chi(k_{0},k_{0})|\leqslant c||\gamma(k)||_{H^{1}}|k-k_{0}|^{1/2}.
\end{equation}
    Then, the result follows promptly.
    \end{proof}

   Now we use $T_1(k)$ and $T_2(k)$ to define a new matrix function
   \begin{equation}\label{M1}
   M^{(1)}(k;x,t)=M(k;x,t)\Delta^{-1}(k),\end{equation}
   where
   \begin{equation}\Delta(k)=\begin{pmatrix}T_1(k)&0\\0&T_2^{-1}(k)\end{pmatrix}.\end{equation}
   $M^{(1)}(k;x,t)$ is a solution for the following RH problem.
    \begin{RHP}\label{RHP2}
    		Find a matrix valued function $M^{(1)}(k;x,t)$ admits:
    		\begin{enumerate}[(i)]
    			\item Analyticity:~$M^{(1)}(k;x,t)$ is analytic in $k\in\mathbb{C}\setminus(\mathbb{R}\cup\mathcal{Z}\cup\bar{\mathcal{Z}})$;
    			\item Jump condition:
    			\begin{equation}M_+^{(1)}(k;x,t)=M_-^{(1)}(k;x,t)J^{(1)}(k;x,t),\quad k\in\mathbb{R},\end{equation}
    		\begin{equation}\left.J^{(1)}(k;x,t)=
    \begin{cases}\begin{pmatrix}I_{2\times2}&T_{1-}(k)\gamma^\dagger(\bar{k})T_{2-}(k)e^{-2it\theta}\\0&I_{2\times2}\end{pmatrix}
    \begin{pmatrix}I_{2\times2}&0\\T_{2+}^{-1}(k)\gamma(k)T_{1+}^{-1}(k)e^{2it\theta}&I_{2\times2}\end{pmatrix},\quad k\in(-\infty,k_0),\\
    \begin{pmatrix}I_{2\times2}&0\\T_{2-}^{-1}(k)\rho^{\dagger}(\bar{k})T_{1-}^{-1}(k)e^{2it\theta}&I_{2\times2}\end{pmatrix}
    \begin{pmatrix}I_{2\times2}&T_{1+}(k)\rho(k)T_{2+}(k)e^{-2it\theta}\\0&I_{2\times2}\end{pmatrix},\quad k\in(k_0,+\infty);&\end{cases}\right.\end{equation}
    where
    \begin{equation}\rho(k)=\left(I_{2\times2}+\gamma^\dagger(\bar{k})\gamma(k)\right)^{-1}\gamma^\dagger(\bar{k}),\nonumber\end{equation}

    			\item Asymptotic behavior:
    			\begin{align}
    				M^{(1)}(k;x,t)\rightarrow{I_{4\times4}}~~as~~k\rightarrow\infty;
    			\end{align}
                \item Residue conditions: $M^{(1)}(k;x,t)$ has simple poles at each point in $\mathcal{Z}\cup\bar{\mathcal{Z}}$

                For $j\in\Delta_{k_0}^-$
                \begin{equation}
    \operatorname*{Res}_{k=k_j}M^{(1)}(k)=\lim_{k\to k_j}M^{(1)}(k)\begin{pmatrix}0&[({T_1}^{-1})'(k_j)]^{-1}f^{-1}(k_j)[({T_2}^{-1})'(k_j)]^{-1}e^{-2it\theta(k_j)}
    \\0&0\end{pmatrix};
                 \end{equation}
               \begin{equation}
    \operatorname*{Res}_{k{=}\bar{k}}M^{(1)}(k)=\lim_{k\to \bar{k_j}}M^{(1)}(k)\begin{pmatrix}0&0\\
    -[{T_2}'(\bar{k_j})]^{-1}(f^\dagger(k_j))^{-1}[{T_1}'(\bar{k_j})]^{-1}e^{2it\theta(\bar{k_j})}&0\end{pmatrix};
\end{equation}

                For $j\in\Delta_{k_0}^+$
                \begin{equation}
    \operatorname*{Res}_{k=k_j}M^{(1)}(k)=\lim_{k\to k_j}M^{(1)}(k)\begin{pmatrix}0&0\\{T_2}^{-1}(k_j)f(k_j){T_1}^{-1}(k_j)e^{2it\theta(k_j)}&0\end{pmatrix};
\end{equation}
                \begin{equation}
    \operatorname*{Res}_{k{=}\bar{k}}M^{(1)}(k)=\lim_{k\to \bar{k_j}}M^{(1)}(k)\begin{pmatrix}0&-T_1(\bar{k_j})f^\dagger(k_j)T_2(\bar{k_j})e^{-2it\theta(\bar{k_j})}
    \\0&0\end{pmatrix};
                 \end{equation}

    		\end{enumerate}
    	\end{RHP}
    Since $\Delta^{-1}(k)\to I$, $k\to\infty $, the relation between the solution of the spin-1 GP equation and the solution of the RH problem is
    \begin{equation}q(x,t)=2i\lim_{k\to\infty}(kM^{(1)}(k;x,t))_{UR}.\end{equation}
     \subsection{\it{Transformation to a hybrid $\bar{\partial}\text{-problem}$}}\label{s:3.2}
     Next, we make continuous extension for the jump matrix $J^{(1)}(k;x,t)$ to remove the jump from $\mathbb{R}$. Denote lines $\Sigma_{i}$ and domains
     $\Omega_{jl}$, $i = 1, 2, 3, 4$ as shown in Figure \ref{Figure 2.}. And $\Sigma^{(1)}=\Sigma_1\cup\Sigma_2\cup\Sigma_3\cup\Sigma_4$.
     The key is to construct the matrix function $R^{(2)}(k)$. We need to eliminate jumps on $\mathbb{R}$  and the new analytic jump matrix has the expected
     exponential decay along the contour $\Sigma^{(1)}$. The norm of $R^{(2)}(k)$ should be controlled to ensure that the long-time asymptotic behavior
     of $\overline{\partial}\text{-contribution}$ to the solution  is negligible. Now we introduce a new matrix $M^{(2)}(k)$ such that the jump contour of
     the RH problem \ref{RHP2} is transformed from $\mathbb{R}$ to $\Sigma^{(1)}$.
\\

     \centerline{\begin{tikzpicture}[scale=0.7]
\path [fill=pink] (-5,4)--(-4,4) to (-4,4) -- (0,0);
\path [fill=pink] (-5,4)--(-4,4) to (-4,4) -- (-4,3);
\path [fill=pink] (-4,3)--(-5,3) to (-5,3) -- (-5,4);
\path [fill=pink] (0,0)--(-4,-4) to (-4,-4) -- (-5,-4);
\path [fill=pink] (-5,-4)--(-5,0) to (-5,0) -- (-5,-4);
\path [fill=pink] (-5,-3)--(-5,3) to (-5,3) -- (-4,3);
\path [fill=pink] (-4,3)--(-4,-3) to (-4,-3) -- (-5,-3);
\path [fill=pink] (-5,-4)--(-4,-4) to (-4,-4) -- (-4,-3);
\path [fill=pink] (-4,-3)--(-5,-3) to (-5,-3) -- (-5,-4);
\path [fill=pink] (-5,4)--(-4,4) to (-4,4) -- (0,0);
\path [fill=pink] (-4,-4)--(0,0) to (0,0) -- (-4,0);
\path [fill=pink] (-4,0)--(0,0) to (0,0) -- (-4,4);
\path [fill=pink] (5,4)--(4,4) to (4,4) -- (0,0);
\path [fill=pink] (0,0)--(4,-4) to (4,-4) -- (5,-4);
\path [fill=pink] (5,-4)--(5,0) to (5,0) -- (5,4);
\path [fill=pink] (5,4)--(4,4) to (4,4) -- (4,3);
\path [fill=pink] (4,3)--(5,3) to (5,3) -- (5,4);
\path [fill=pink] (4,3)--(5,3) to (5,3) -- (5,-3);
\path [fill=pink] (5,-3)--(4,-3) to (4,-3) -- (4,3);
\path [fill=pink] (4,-3)--(5,-3) to (5,-3) -- (5,-4);
\path [fill=pink] (5,-4)--(4,-4) to (4,-4) -- (4,-3);
\path [fill=pink] (0,0)--(4,4) to (4,4) -- (4,0);
\path [fill=pink] (4,0)--(4,-4) to (4,-4) -- (0,0);
\draw [thick](-4,-4)--(4,4);
\draw [thick](-4,4)--(4,-4);
\draw[fill] (0,0)node[below]{$k_{0}$};
\draw[fill] (0,-1)node[below]{$\Omega_{5}$};
\draw[fill] (0,1.5)node[below]{$\Omega_{2}$};
\draw[fill] (1.5,-0.5)node[below]{$\Omega_{6}$};
\draw[fill] (1.5,1)node[below]{$\Omega_{1}$};
\draw[fill] (-1.5,1)node[below]{$\Omega_{3}$};
\draw[fill] (-1.5,-0.5)node[below]{$\Omega_{4}$};
\draw[fill] (-6,2)node[below]{$\mathcal {R}^{(2)}=U_{R}^{-1}$};
\draw[fill] (4,2)node[below]{$\mathcal {R}^{(2)}=W_{R}^{-1}$};
\draw[fill] (4.5,-2)node[above]{$\mathcal {R}^{(2)}=W_{L}$};
\draw[fill] (-6,-2)node[below]{$\mathcal {R}^{(2)}=U_{L}$};
\draw[fill] (4,2.5)node[below]{} circle [radius=0.05];
\draw[fill] (4,-2.5)node[below]{} circle [radius=0.05];
\draw[fill] (0,2.5)node[below]{} circle [radius=0.05];
\draw[fill] (0,-2.5)node[below]{} circle [radius=0.05];
\draw[fill] (2,3)node[below]{} circle [radius=0.05];
\draw[fill] (2,-3)node[below]{} circle [radius=0.05];
\draw[fill] (-2.5,2)node[below]{} circle [radius=0.05];
\draw[fill] (-2.5,-2)node[below]{} circle [radius=0.05];
\draw[fill] (-4,3)node[below]{} circle [radius=0.05];
\draw[fill] (-4,-3)node[below]{} circle [radius=0.05];
\draw[fill] (-2.5,2)node[below]{${k}_{j}$};
\draw[fill] (-2.5,-2)node[left]{$\bar{{k}_{j}}$};
\draw[fill] (4.5,4)node[below]{$\Sigma_{1}$};
\draw[fill] (4.5,-4)node[below]{$\Sigma_{4}$};
\draw[fill] (-4.5,4)node[below]{$\Sigma_{2}$};
\draw[fill] (-4.5,-4)node[below]{$\Sigma_{3}$};
\draw[-][thick][dashed](-6,0)--(-5,0);
\draw[-][thick][dashed](-5,0)--(-4,0);
\draw[-][thick][dashed](-4,0)--(-3,0);
\draw[->][thick][dashed](-3,0)--(-2,0);
\draw[-][thick][dashed](-2,0)--(-1,0);
\draw[-][thick][dashed](-1,0)--(0,0);
\draw[-][thick][dashed](0,0)--(1,0);
\draw[-][thick][dashed](1,0)--(2,0);
\draw[->][thick][dashed](2,0)--(3,0);
\draw[-][thick][dashed](3,0)--(4,0);
\draw[-][thick][dashed](4,0)--(5,0);
\draw[->][thick][dashed](5,0)--(6,0)[thick]node[right]{$Rek$};
\end{tikzpicture}}\label{Figure 2.}
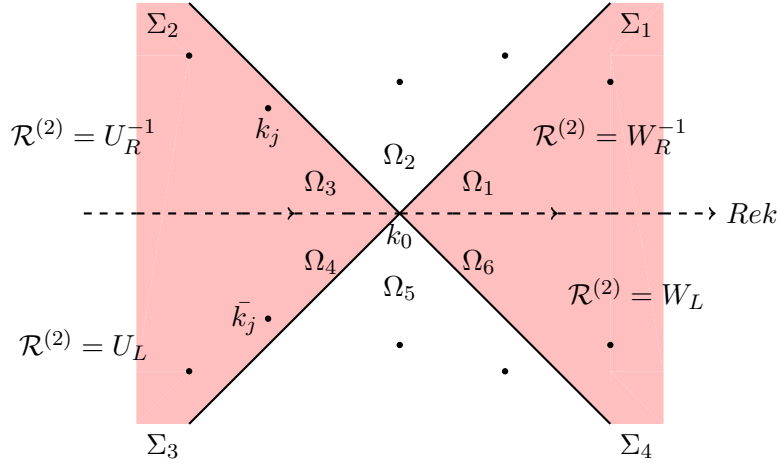
\centerline{\noindent {\small \textbf{Figure 3.2}  Definition of $R^{(2)}$ in different domains.}}

     \begin{equation}\label{M2}
     M^{(2)}(k)=M^{(1)}(k)R^{(2)}(k),\end{equation}
     where
     \begin{equation}R^{(2)}(k)=\left\{\begin{array}{cc}\left(\begin{array}{cc}I_{2\times2}&0\\-R_1e^{2it\theta}&I_{2\times2}\end{array}\right)=W_R^{-1},&k\in\Omega_1,\\
                                                        \left(\begin{array}{cc}I_{2\times2}&-R_{3}e^{-2it\theta}\\0&I_{2\times2}\end{array}\right)=U_R^{-1},&k\in\Omega_{3},\\
                                                        \left(\begin{array}{cc}I_{2\times2}&0\\R_{4}e^{2it\theta}&I_{2\times2}\end{array}\right)=U_{L},&k\in\Omega_{4},\\
                                                        \left(\begin{array}{cc}I_{2\times2}&R_6e^{-2it\theta}\\0&I_{2\times2}\end{array}\right)=W_{L},&k\in\Omega_6,\\
                                                        \left(\begin{array}{cc}I_{2\times2}&0\\0&I_{2\times2}\end{array}\right),&k\in\Omega_2\cup\Omega_5.
                                                        \end{array}\right.\end{equation}
     And the matrices $R_j(j=1,3,4,6)$ such that the following proposition.

     \begin{proposition} The matrices $R_j$ have the following boundary values:

     \begin{equation}\nonumber\left.R_1(k)=\begin{cases}T_{2+}^{-1}(k)\gamma(k)T_{1+}^{-1}(k),&k\in(k_0,\infty),\\
     T_{0}^{-1}(k_0)\gamma(k_0)T_{0}^{-1}(k_0)(k-k_0)^{-2i\nu(k_0)}(1-\chi_{\mathcal{Z}}(k)),&k\in\Sigma_1,\end{cases}\right.\end{equation}
     \begin{equation}\nonumber\left.R_3(k)=\begin{cases}T_{1+}(k)\rho(k)T_{2+}(k),&k\in(-\infty,k_0),\\
     T_{0}(k_0)\rho(k_0)T_{0}^(k_0)(k-k_0)^{2i\nu(k_0)}(1-\chi_{\mathcal{Z}}(k)),&k\in\Sigma_2,\end{cases}\right.\end{equation}
     \begin{equation}\nonumber\left.R_4(k)=\begin{cases}T_{2-}^{-1}(k)\rho^{\dagger}(\bar{k})T_{1-}^{-1}(k),&k\in(-\infty,k_0),\\
     T_{0}^{-1}(k_0)\rho^{\dagger}({k_0}^*)T_{0}^{-1}(k_0)(k-k_0)^{-2i\nu(k_0)}(1-\chi_{\mathcal{Z}}(k)),&k\in\Sigma_3,\end{cases}\right.\end{equation}
      \begin{equation}\nonumber\left.R_6(k)=\begin{cases}T_{1-}(k)\gamma^\dagger(\bar{k})T_{2-}(k),&k\in(k_0,\infty),\\
      T_{0}(k_0)\gamma^\dagger({k_0}^*)T_{0}(k_0)(k-k_0)^{2i\nu(k_0)}(1-\chi_{\mathcal{Z}}(k)),&k\in\Sigma_4.\end{cases}\right.\end{equation}
     These matrices $R_j$ are estimated as follows
    \begin{equation}\label{R111}
    |R_j(k)|\lesssim \sin^2(\arg k)+\langle Rek\rangle^{-1/2},j=1,3,4,6,\end{equation}
    \begin{equation}\label{R1}
    |\bar{\partial}R_j(k)|\lesssim|\bar{\partial}\chi_{\mathcal{Z}}(k)|+|\gamma^{\prime}(Re(k))|+|k-k_0|^{-1/2}+|k-k_0|^{-1},
    \textit{ for all }k\in\Omega_j,j=1,3,4,6,\end{equation}
    \begin{equation}\nonumber
    \bar{\partial}R_j(k)=0,\quad k\in\Omega_2\cup\Omega_5, or \operatorname{dist}(k,\mathcal{Z}\cup\overline{\mathcal{Z}})\leqslant\rho/3.
    \end{equation}
    \end{proposition}
   \begin{proof} Since the functions $R_j$, j = 1, 3, 4, 6 have the same construction, we take $R_1$ for example.
   Denote $k=k_0+\varrho e^{i\alpha}$, for $k\in\Omega_{1}$, $\varrho=|k-k_0|$, $\alpha\in[0,\pi/4]$. Under the $(\varrho,\alpha)$-coordinate, the $\bar{\partial}$-derivative
    has the following representation
   \begin{equation}\nonumber\bar{\partial}=\frac12e^{i\alpha}(\partial_\varrho+i\varrho^{-1}\partial_\alpha).\end{equation}
   Define that
   \begin{equation}\nonumber
    g_1=T_0^{-2}(k_0)T_{2+}(k)\gamma(k_0)T_{1+}(k)(k-k_0)^{-2i\nu(k_0)},\quad k\in\bar{\Omega}_1,\nonumber
   \end{equation}

    \begin{equation}\nonumber
    \begin{aligned}
    R_1(k)& =T^{-1}_{2+}(k)\{\gamma(Rek)\cos(2\varphi)+[1-\cos(2\varphi)]g_1(k)\}T^{-1}_{1+}(k)(1-\chi_{\mathcal{Z}}(k))\\
    &=T^{-1}_{2+}(k)\{g_1(k)+[\gamma(Rek)-g_1(k)]\cos(2\varphi)\}T^{-1}_{1+}(k)(1-\chi_{\mathcal{Z}}(k)),
    \end{aligned}\end{equation}
    we have
    \begin{equation}\nonumber
     \begin{aligned}
    \bar{\partial}R_{1}=& -T^{-1}_{2+}(k)[r(Rek)\cos(2\varphi)+g_1(1-\cos(2\varphi))]T^{-1}_{1+}(k)\bar{\partial}\chi_{\mathcal{Z}}(k) \\
    &+T^{-1}_{2+}(k)\left[\frac12e^{i\varphi}\gamma^{\prime}(Rek)\cos(2\varphi)-ie^{i\varphi}\frac{(r(Rek)-g_1)\sin(2\varphi)}{|k-k_0|}\right]
    T^{-1}_{1+}(k)(1-\chi_{\mathcal{Z}}(k)).
  \end{aligned}
  \end{equation}
  Thus
  \begin{equation}\nonumber
    |\bar{\partial}R_1|\leqslant c_1\bar{\partial}\chi_{\mathcal{Z}}(k)+c_2|\gamma'(Rek)|+\frac{c_3|\gamma(Rek)-g_1|}{|k-k_0|}.
  \end{equation}

   The last item of the right is rewritten as
     \begin{equation}\nonumber
    |\gamma(Rek)-g_1|\leqslant|\gamma(Rek)-\gamma(k_0)|+|\gamma(k_0)-g_1|,
    \end{equation}
    we get
    \begin{equation}\nonumber
    |\gamma(\text{Re}k)-\gamma(k_0)|=\left|\int_{k_0}^{\text{Re}k}\gamma^{\prime}(s)ds\right|\leqslant||\gamma||_{H^1}|\text{Re}k-k_0|^{1/2}\leqslant c|k-k_0|^{1/2},
    \end{equation}
    \begin{align}
    |\gamma(k_{0})-g_1(k)|&=|\gamma(k_0)-T_0^{-2}(k_0)T_{2+}(k)\gamma(k_0)T_{1+}(k)(k-k_0)^{-2i\nu(k_0)}|\nonumber\\
     &\leqslant|T_{2+}(k)||T^{-1}_{2+}(k)\gamma(k_0)T^{-1}_{1+}(k)-T_0^{-2}(k_0)\gamma(k_0)(k-k_0)^{-2i\nu(k_0)}||T_{1+}(k)|\nonumber,
\end{align}
    we note that $\tilde{a}(k)=\prod_{j\in\Delta_{k_0}^-}\frac{k-\bar{k}_j}{k-k_j}$, so $T_{j}(k)=\tilde{a}(k)\delta_{j}(k)$,
    \begin{align}
    T^{-1}_{2+}(k)\gamma(k_0)T^{-1}_{1+}(k)& =\tilde{a}^{-2}(k)\delta^{-1}_{2+}(k)\gamma(k_0)\delta^{-1}_{1+}(k)\nonumber\\
     & =\tilde{a}^{-2}(k)\delta^{-1}_{2+}(k)\gamma(k_0)[\delta^{-1}_{1+}(k)-(\det\delta(k))^{-1}]+\tilde{a}^{-2}(k)\delta^{-1}_{2+}(k)\gamma(k_0)(\det\delta(k))^{-1}\nonumber\\
     & =f_1+\tilde{a}^{-2}(k)\delta^{-1}_{2+}(k)\gamma(k_0)(\det\delta(k))^{-1}\nonumber\\
     & =f_1+\tilde{a}^{-2}(k)[\delta^{-1}_{2+}(k)-(\det\delta(k))^{-1}]\gamma(k_0)(\det\delta(k))^{-1}+\tilde{a}^{-2}(k)(\det\delta(k))^{-2}\gamma(k_0)\nonumber\\
     & =f_1+f_2+\tilde{a}^{-2}(k)(\det\delta(k))^{-2}\gamma(k_0)\nonumber.
    \end{align}
    where $f_1=\tilde{a}^{-2}(k)\delta^{-1}_{2+}(k)\gamma(k_0)[\delta^{-1}_{1+}(k)-(\det\delta(k))^{-1}]$,
    $f_2=\tilde{a}^{-2}(k)[\delta^{-1}_{2+}(k)-(\det\delta(k))^{-1}]\gamma(k_0)(\det\delta(k))^{-1}$.

    So we can get
    \begin{align}|\gamma(k_{0})-g_1(k)|& \leqslant|T_{2+}(k)||T^{-1}_{2+}(k)\gamma(k_0)T^{-1}_{1+}(k)-T_0^{-2}(k_0)\gamma(k_0)(k-k_0)^{-2i\nu(k_0)}||T_{1+}(k)|\nonumber\\
    & \leqslant|T_{2+}(k)|[|f_1|+|f_2|+|a^{2}(k)(\det\delta(k))^{-2}\gamma(k_0)-T_0^{-2}(k_0)\gamma(k_0)(k-k_0)^{-2i\nu(k_0)}|]|T_{1+}(k)|\nonumber\\
    & \leqslant c(c_1+c_2+c_3||\gamma||_{H^1}|k-k_0|^{1/2})\nonumber.\end{align}

    Then \eqref{R1} follows immediately.
    \end{proof}

    $M^{(2)}(k)$ satisfies the following mixed $\bar{\partial}\text{-RH}$ problem.
    \begin{RHP}\label{RHP3}
    		Find a matrix valued function $$M^{(2)}(k)=M^{(2)}(k;x,t)$$ admits:
    		\begin{enumerate}[(i)]
    			\item ~$M^{(2)}(k;x,t)$ is continuous in $k\in\mathbb{C}\setminus(\mathbb{R}\cup\mathcal{Z}\cup\bar{\mathcal{Z}})$;
    			\item Jump condition:
    			\begin{equation}M_+^{(2)}(k)=M_-^{(2)}(k)V^{(2)}(k),\quad k\in\Sigma^{(2)},\end{equation}
              where the jump matrix
              \begin{equation}
             V^{(2)}(k)=(R_-^{(2)})^{-1}J^{(1)}R_+^{(2)}=I+(1-\chi_\mathcal{Z}(k))\delta V^{(2)},
            \end{equation}
    		\begin{equation}
    \delta V^{(2)}(k)=\begin{cases}\begin{pmatrix}0&0\\T_0(k_0)^{-2}\gamma(k_0)(k-k_0)^{-2i\nu(k_0)}e^{2it\theta}&0\end{pmatrix},\quad k\in\Sigma_1,\\\\
    \begin{pmatrix}0&T_0(k_0)^{2}\rho(k_0)(k-k_0)^{2i\nu(k_0)}e^{-2it\theta}\\0&0\end{pmatrix},\quad k\in\Sigma_2,\\\\\begin{pmatrix}0&0\\
    T_0(k_0)^{-2}\rho^{\dagger}(\bar{k}_0)(k-k_0)^{-2i\nu(k_0)}e^{2it\theta}&0\end{pmatrix},\quad k\in\Sigma_3,\\\\
    \begin{pmatrix}0&T_0(k_0)^{2}\gamma^\dagger(\bar{k}_0)(k-k_0)^{2i\nu(k_0)}e^{-2it\theta}\\0&0\end{pmatrix},\quad k\in\Sigma_4,&\end{cases}
\end{equation}
    where
    \begin{equation}\rho(k)=\left(I_{2\times2}+\gamma^\dagger(\bar{k})\gamma(k)\right)^{-1}\gamma^\dagger(\bar{k}).\nonumber\end{equation}

    			\item Asymptotic behavior:
    			\begin{align}
    				M^{(2)}(k;x,t)\rightarrow{I_{4\times4}}~~as~~k\rightarrow\infty.
    			\end{align}
                \item $\bar{\partial}-Derivative$: for $\mathbb{C}\setminus(\Sigma^{(2)}\cup\mathcal{Z}\cup\bar{\mathcal{Z}})$ we have
                \begin{equation}
                \bar{\partial}M^{(2)}(k)=M^{(2)}(k)\bar{\partial}R^{(2)}(k),
                \end{equation}
                where
                \begin{equation}\bar{\partial}R^{(2)}(k)
                =\left\{\begin{array}{cc}\left(\begin{array}{cc}0&0\\-\bar{\partial}R_1e^{2it\theta}&0\end{array}\right),&k\in\Omega_1,\\
                                                        \left(\begin{array}{cc}0&-\bar{\partial}R_{3}e^{-2it\theta}\\0&0\end{array}\right),&k\in\Omega_{3},\\
                                                        \left(\begin{array}{cc}0&0\\\bar{\partial}R_{4}e^{2it\theta}&0\end{array}\right),&k\in\Omega_{4},\\
                                                        \left(\begin{array}{cc}0&\bar{\partial}R_6e^{-2it\theta}\\0&0\end{array}\right),&k\in\Omega_6,\\
                                                        \left(\begin{array}{cc}0&0\\0&0\end{array}\right),&k\in\Omega_2\cup\Omega_5.
                                                        \end{array}\right.\end{equation}

                \item Residue conditions: $M^{(2)}(k;x,t)$ has simple poles at each point in $\mathcal{Z}\cup\bar{\mathcal{Z}}$

                For $j\in\Delta_{k_0}^-$
                \begin{equation}
    \operatorname*{Res}_{k=k_j}M^{(2)}(k)=
    \lim_{k\to k_j}M^{(2)}(k)\begin{pmatrix}0&[({T_1}^{-1})'(k_j)]^{-1}f^{-1}(k_j)[({T_2}^{-1})'(k_j)]^{-1}e^{-2it\theta(k_j)}\\0&0\end{pmatrix};
                 \end{equation}
               \begin{equation}
    \operatorname*{Res}_{k{=}\bar{k}}M^{(2)}(k)=\lim_{k\to \bar{k_j}}M^{(2)}(k)
    \begin{pmatrix}0&0\\-[{T_2}'(\bar{k_j})]^{-1}(f^\dagger(k_j))^{-1}[{T_1}'(\bar{k_j})]^{-1}e^{2it\theta(\bar{k_j})}&0\end{pmatrix};
\end{equation}

                For $j\in\Delta_{k_0}^+$
                \begin{equation}
    \operatorname*{Res}_{k=k_j}M^{(2)}(k)=\lim_{k\to k_j}M^{(2)}(k)\begin{pmatrix}0&0\\{T_2}^{-1}(k_j)f(k_j){T_1}^{-1}(k_j)e^{2it\theta(k_j)}&0\end{pmatrix};
\end{equation}
                \begin{equation}
    \operatorname*{Res}_{k{=}\bar{k}}M^{(2)}(k)=\lim_{k\to \bar{k_j}}M^{(2)}(k)
    \begin{pmatrix}0&-T_1(\bar{k_j})f^\dagger(k_j)T_2(\bar{k_j})e^{-2it\theta(\bar{k_j})}\\0&0\end{pmatrix}.
                 \end{equation}

    		\end{enumerate}
    	\end{RHP}
    The relation between the solution of the spin-1 GP equation and the solution of the RH problem is
    \begin{equation}\nonumber
    q(x,t)=2i\lim_{k\to\infty}(kM^{(2)}(x,t,k))_{UR}.
    \end{equation}
    To solve RH problem \ref{RHP3}, we decompose RH problem \ref{RHP3} into a pure RH problem with $\overline{\partial}R^{(2)}=0$ and a pure
    $\overline{\partial}$-problem with ${\overline{\partial}}R^{(2)}\neq0$. We express the decomposition as follows:
    \begin{equation}\nonumber
    M^{(2)}(k;x,t)=\begin{cases}\overline{\partial}R^{(2)}=0\to M_{RHP}^{(2)},\\\overline{\partial}R^{(2)}\neq0\to M^{(3)}=M^{(2)}M_{RHP}^{(2)-1},\end{cases}
    \end{equation}

    \subsection{\it{Asymptotic analysis on a pure RH problem}}\label{s:3.3}
     $M_{RHP}^{(2)}$ is the pure RH part of the mixed RH problem \ref{RHP3}, that is, it has the same jump line and residue conditions as $M^{(2)}$,
     but $\overline{\partial}R^{(2)}=0$. We call it the pure RH problem, which is described as follows:
     \begin{RHP}\label{RHP4}
    		Find a matrix valued function $M_{RHP}^{(2)}$ admits:
    		\begin{enumerate}[(i)]
    			\item ~$M_{RHP}^{(2)}(k;x,t)$ is continuous in $k\in\mathbb{C}\setminus(\mathbb{R}\cup\mathcal{Z}\cup\bar{\mathcal{Z}})$;
    			\item Jump condition:
    			\begin{equation}M_{+RHP}^{(2)}(k)=M_{-RHP}^{(2)}(k)V^{(2)}(k),\quad k\in\Sigma^{(2)};\end{equation}

    			\item Asymptotic behavior:
    			\begin{align}
    				M_{RHP}^{(2)}(k;x,t)\rightarrow{I_{4\times4}}~~as~~k\rightarrow\infty;
    			\end{align}
                \item $\overline{\partial}R^{(2)}=0;$

                \item Residue conditions: With the same residue conditions as $M^{(2)}$.

    		\end{enumerate}
    	\end{RHP}

    Denote
    \begin{equation}\nonumber
    \mathcal{U}_{k_0}=\{k:|k-k_0|<\rho/2\}.
    \end{equation}

    \begin{proposition} For the jump matrix $V^{(2)}(k)$, we have the following estimate
    \begin{equation}\nonumber
    ||V^{(2)}-I||_{L^\infty(\Sigma^{(2)})}=\begin{cases}\mathcal{O}(e^{-t\rho^2}),&k\in\Sigma^{(2)}\setminus\mathcal{U}_{k_0},\\c|k-k_0|^{-1}t^{-1/2},
    &k\in\Sigma^{(2)}\cap\mathcal{U}_{k_0}.\end{cases}
    \end{equation}
    \end{proposition}
   \begin{proof} On $\Sigma_{1}$, the jump line is $k-k_0=|k-k_0|e^{i\pi/4}$, and thus follows
   \begin{equation}\nonumber
    \theta=2(k-k_0)^2-2k_0^2=2i|k-k_0|^2-2k_0^2.
    \end{equation}
    Using \eqref{R111} and \eqref{sss}, it can be obtained
    \begin{equation}\nonumber
    |R_1e^{2it\theta(k)}|\leqslant|R_1|e^{-2t\text{Im}\theta(k)}\leqslant\left(\frac{1}{2}c_1+c_2\langle\text{Re}k\rangle^{-1/2}\right)e^{-4t|k-k_0|^2}.
    \end{equation}
    Note that
    \begin{equation}\nonumber
    \langle Rek\rangle^{-1/2}=\frac1{[1+(k_0+|k-k_0|e^{i\pi/4})^2]^{1/4}},\nonumber
    \end{equation}
    so we have
    \begin{equation}\nonumber
    \begin{aligned}&\langle Rek\rangle^{-1/2}\to\frac1{(1+k_0^2)^{1/4}},\quad k\to k_0,\\&\langle Rek\rangle^{-1/2}\to0,\quad k\to\infty.\end{aligned}\nonumber
    \end{equation}
    Thus $\langle Rek\rangle^{-1/2}\leqslant c.$

    For $k\in\Sigma^{(2)}\cap U_{k_0},$
    \begin{equation}\nonumber
     \begin{gathered}
      ||V^{(2)}-I||_{L^{\infty}(\Sigma^{(2)})} \leqslant ct^{-1/2}|k-k_{0}|^{-1}(t^{1/2}|k-k_{0}|e^{-4t|k-k_{0}|^{2}}) \\
     \leqslant c|k-k_{0}|^{-1}t^{-1/2}.
     \end{gathered}
      \end{equation}
     It can be seen that within $U_{k_0}$, the jump matrix $V^{(2)}$ decays to the identity matrix point by point.
     \end{proof}

     For $k\in\Sigma\cap\{|k-k_0|\geqslant\rho/2\}$
     \begin{equation}\nonumber
    ||V^{(2)}-I||_{L^\infty(\Sigma^{(2)})}\leqslant ce^{-4t|k-k_0|^2}\leqslant ce^{-t\rho^2}.
       \end{equation}

    This proposition inspire us to construct the solution $M_{RHP}^{(2)}(k)$ of the RH problem \ref{RHP4} in following form
    \begin{equation}\nonumber
    M_{RHP}^{(2)}(k)=\begin{cases}E(k)M^{(out)}(k),&\quad k\in\mathbb{C}\setminus\mathcal{U}_{k_0},\\E(k)M^{(LC)}(k)=E(k)M^{(out)}(k)M^{(SA)}(k),
    &\quad k\in\mathcal{U}_{k_0},\end{cases}
\end{equation}
    This decomposition splits $M_{RHP}^{(2)}(k)$ into two parts: $E(k)$ is a error function, which is a solution of a small-norm RH problem.
    $M^{(out)}(k)$ reduces to a new RH problem for RH problem \ref{RHP4} as the jump conditions ignore.

    \begin{RHP}\label{RHP6}
    		Find a matrix valued function $M^{(out)}(k)$ admits:
    		\begin{enumerate}[(i)]
    			\item Analyticity: ~$M^{(out)}(k)$ is  analytical in  $k\in\mathbb{C}\setminus(\mathbb{R}\cup\mathcal{Z}\cup\bar{\mathcal{Z}})$;

    			\item Asymptotic behaviors:$M^{(out)}(k)\sim I,\quad k\to\infty,$

                \item Residue conditions: $M^{(out)}(k)$ has simple poles at each point in $k\in\mathcal{Z}\cup\bar{\mathcal{Z}}$ satisfying:
                For $j\in\Delta_{k_0}^-$
                \begin{equation}
    \operatorname*{Res}_{k=k_j}M^{(out)}(k)=\lim_{k\to k_j}M^{(out)}(k)
    \begin{pmatrix}0&[({T_1}^{-1})'(k_j)]^{-1}f^{-1}(k_j)[({T_2}^{-1})'(k_j)]^{-1}e^{-2it\theta(k_j)}\\0&0\end{pmatrix};
                 \end{equation}
               \begin{equation}
    \operatorname*{Res}_{k{=}\bar{k}}M^{(out)}(k)=\lim_{k\to \bar{k_j}}M^{(out)}(k)
    \begin{pmatrix}0&0\\-[{T_2}'(\bar{k_j})]^{-1}(f^\dagger(k_j))^{-1}[{T_1}'(\bar{k_j})]^{-1}e^{2it\theta(\bar{k_j})}&0\end{pmatrix};
\end{equation}

                For $j\in\Delta_{k_0}^+$
                \begin{equation}
    \operatorname*{Res}_{k=k_j}M^{(out)}(k)=\lim_{k\to k_j}M^{(out)}(k)
    \begin{pmatrix}0&0\\{T_2}^{-1}(k_j)f(k_j){T_1}^{-1}(k_j)e^{2it\theta(k_j)}&0\end{pmatrix};
\end{equation}
                \begin{equation}
    \operatorname*{Res}_{k{=}\bar{k}}M^{(out)}(k)=\lim_{k\to \bar{k_j}}M^{(out)}(k)
    \begin{pmatrix}0&-T_1(\bar{k_j})f^\dagger(k_j)T_2(\bar{k_j})e^{-2it\theta(\bar{k_j})}\\0&0\end{pmatrix};
                 \end{equation}

    		\end{enumerate}
    	\end{RHP}

 \centerline{\begin{tikzpicture}[scale=0.7]
\path [fill=pink] (-5,4)--(-4,4) to (-4,4) -- (0,0);
\path [fill=pink] (-5,4)--(-4,4) to (-4,4) -- (-4,3);
\path [fill=pink] (-4,3)--(-5,3) to (-5,3) -- (-5,4);
\path [fill=pink] (0,0)--(-4,-4) to (-4,-4) -- (-5,-4);
\path [fill=pink] (-5,-4)--(-5,0) to (-5,0) -- (-5,-4);
\path [fill=pink] (-5,-3)--(-5,3) to (-5,3) -- (-4,3);
\path [fill=pink] (-4,3)--(-4,-3) to (-4,-3) -- (-5,-3);
\path [fill=pink] (-5,-4)--(-4,-4) to (-4,-4) -- (-4,-3);
\path [fill=pink] (-4,-3)--(-5,-3) to (-5,-3) -- (-5,-4);
\path [fill=pink] (-5,4)--(-4,4) to (-4,4) -- (0,0);
\path [fill=pink] (-4,-4)--(0,0) to (0,0) -- (-4,0);
\path [fill=pink] (-4,0)--(0,0) to (0,0) -- (-4,4);
\path [fill=pink] (5,4)--(4,4) to (4,4) -- (0,0);
\path [fill=pink] (0,0)--(4,-4) to (4,-4) -- (5,-4);
\path [fill=pink] (5,-4)--(5,0) to (5,0) -- (5,4);
\path [fill=pink] (5,4)--(4,4) to (4,4) -- (4,3);
\path [fill=pink] (4,3)--(5,3) to (5,3) -- (5,4);
\path [fill=pink] (4,3)--(5,3) to (5,3) -- (5,-3);
\path [fill=pink] (5,-3)--(4,-3) to (4,-3) -- (4,3);
\path [fill=pink] (4,-3)--(5,-3) to (5,-3) -- (5,-4);
\path [fill=pink] (5,-4)--(4,-4) to (4,-4) -- (4,-3);
\path [fill=pink] (0,0)--(4,4) to (4,4) -- (4,0);
\path [fill=pink] (4,0)--(4,-4) to (4,-4) -- (0,0);
\draw(4,2.5) [black, line width=0.5] circle(0.3);
\draw(4,-2.5) [black, line width=0.5] circle(0.3);
%\draw(2,3) [black, line width=0.5] circle(0.3);
%\draw(2,-3) [black, line width=0.5] circle(0.3);
\draw(-2.5,2) [black, line width=0.5] circle(0.3);
\draw(-2.5,-2) [black, line width=0.5] circle(0.3);
\draw(-4,3) [black, line width=0.5] circle(0.3);
\draw(-4,3) [black, line width=0.5] circle(0.3);
\draw[fill][white] (4,2.5) circle [radius=0.3];
\draw[fill][white] (4,-2.5) circle [radius=0.3];
\draw[fill][white] (-2.5,2) circle [radius=0.3];
\draw[fill][white] (-2.5,-2) circle [radius=0.3];
\draw[fill][white] (-4,3) circle [radius=0.3];
\draw[fill][white] (-4,-3) circle [radius=0.3];
\draw [thick](-4,-4)--(4,4);
\draw [thick](-4,4)--(4,-4);
\draw[fill] (0,0)node[below]{$k_{0}$};
\draw[fill] (0,-1)node[below]{$\Omega_{5}$};
\draw[fill] (0,1.5)node[below]{$\Omega_{2}$};
\draw[fill] (1.5,-0.5)node[below]{$\Omega_{6}$};
\draw[fill] (1.5,1)node[below]{$\Omega_{1}$};
\draw[fill] (-1.5,1)node[below]{$\Omega_{3}$};
\draw[fill] (-1.5,-0.5)node[below]{$\Omega_{4}$};
\draw[fill] (-6,2)node[below]{$\mathcal {R}^{(2)}=U_{R}^{-1}$};
\draw[fill] (6,2)node[below]{$\mathcal {R}^{(2)}=W_{R}^{-1}$};
\draw[fill] (6,-2)node[below]{$\mathcal {R}^{(2)}=W_{L}$};
\draw[fill] (-6,-2)node[below]{$\mathcal {R}^{(2)}=U_{L}$};
\draw[fill] (4,2.5)node[below]{} circle [radius=0.05];
\draw[fill] (4,-2.5)node[below]{} circle [radius=0.05];
\draw[fill] (2,3)node[below]{} circle [radius=0.05];
\draw[fill] (2,-3)node[below]{} circle [radius=0.05];
\draw[fill] (0,2.5)node[below]{} circle [radius=0.05];
\draw[fill] (0,-2.5)node[below]{} circle [radius=0.05];
\draw[fill] (-2.5,2)node[below]{} circle [radius=0.05];
\draw[fill] (-2.5,-2)node[below]{} circle [radius=0.05];
\draw[fill] (-4,3)node[below]{} circle [radius=0.05];
\draw[fill] (-4,-3)node[below]{} circle [radius=0.05];
\draw[fill] (-2.5,2)node[below left]{${k}_{j}$};
\draw[fill] (-2.5,-2)node[below left]{${\bar{k}}_{j}$};
\draw[fill] (4.5,4)node[below]{$\Sigma_{1}$};
\draw[fill] (4.5,-4)node[below]{$\Sigma_{4}$};
\draw[fill] (-4.5,4)node[below]{$\Sigma_{2}$};
\draw[fill] (-4.5,-4)node[below]{$\Sigma_{3}$};
\draw[-][thick][dashed](-6,0)--(-5,0);
\draw[-][thick][dashed](-5,0)--(-4,0);
\draw[-][thick][dashed](-4,0)--(-3,0);
\draw[->][thick][dashed](-3,0)--(-2,0);
\draw[-][thick][dashed](-2,0)--(-1,0);
\draw[-][thick][dashed](-1,0)--(0,0);
\draw[-][thick][dashed](0,0)--(1,0);
\draw[-][thick][dashed](1,0)--(2,0);
\draw[->][thick][dashed](2,0)--(3,0);
\draw[-][thick][dashed](3,0)--(4,0);
\draw[-][thick][dashed](4,0)--(5,0);
\draw[->][thick][dashed](5,0)--(6,0)[thick]node[right]{$Rek$};
\end{tikzpicture}}\label{Figure 3.}
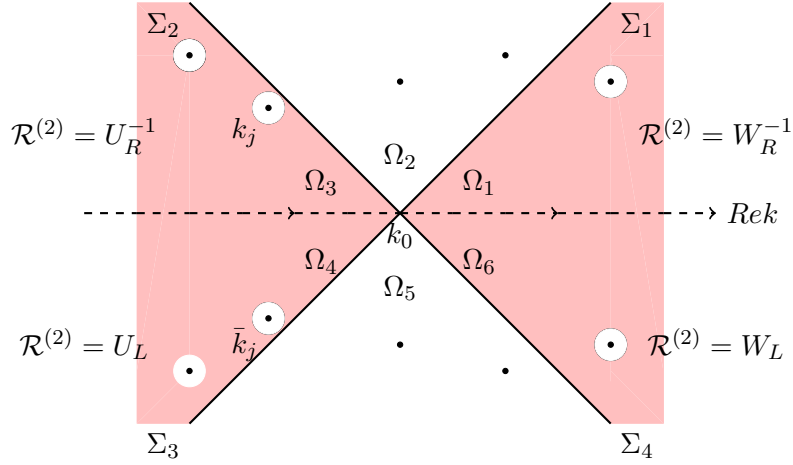
\noindent {\small \textbf{Figure 3.3}   Jump matrix $V^{(2)}$, $\bar{\partial}$ derivative of pink region: $\bar{\partial}R^{(2)}\neq0$.
$\bar{\partial}$-derivative of the white region: $\bar{\partial}R^{(2)}=0$.}

    We will establish the non reflective case of RH problem \ref{RHP3} as RH problem \ref{RHP6} to indicate that it approximates the finite sum of
    soliton solutions in this section. On the basis of the original RH problem \ref{RHP1}, the existence and uniqueness of the RH problem \ref{RHP6}
    solution have been proven in this section.

    Here, the main distribution to the RH problem $M^{(out)}(k)$ is from the soliton solutions corresponding to the scattering data
    \begin{equation}\nonumber
    \sigma_d^{(out)}=\left\{(k_j,\tilde{c}_j),\quad k_j\in\mathcal{Z}\right\}_{k=1}^{2N},\quad\widetilde{c}_j(k_j)={T_2}^{-1}(k_j)f(k_j){T_1}^{-1}(k_j).
    \end{equation}
    Next we build a outer model RH problem and show that its solution can be approximated with a finite sum of solitons. We first recall the RH problem
    corresponding to the matrix function $M^{(out)}(k;\sigma_d^{out}){:}$
    \begin{RHP}\label{RHP7}
    		Find a matrix valued function $M^{(out)}(k;\sigma_d^{out})$ admits:
    		\begin{enumerate}[(i)]
    			\item Analyticity: ~$M^{(out)}(k;\sigma_d^{out})$ is analytical in $k\in\mathbb{C}\setminus(\mathbb{R}\cup\mathcal{Z}\cup\bar{\mathcal{Z}})$;

    			\item Asymptotic behaviors:$M^{(out)}(k;\sigma_d^{out})\sim I,\quad k\to\infty$;

                \item Residue conditions: $M^{(out)}(k;\sigma_d^{out})$ has simple poles at each point in $k\in\mathcal{Z}\cup\bar{\mathcal{Z}}$ satisfying:
                For $j\in\Delta_{k_0}^-$
                \begin{equation}
    \operatorname*{Res}_{k=k_j}M^{(out)}(k;\sigma_d^{out})=\lim_{k\to k_j}M^{(out)}(k;\sigma_d^{out})
    \begin{pmatrix}0&[({T_1}^{-1})'(k_j)]^{-1}f^{-1}(k_j)[({T_2}^{-1})'(k_j)]^{-1}e^{-2it\theta(k_j)}\\0&0\end{pmatrix};
                 \end{equation}
               \begin{equation}
    \operatorname*{Res}_{k{=}\bar{k}}M^{(out)}(k;\sigma_d^{out})=\lim_{k\to \bar{k_j}}M^{(out)}(k;\sigma_d^{out})
    \begin{pmatrix}0&0\\-[{T_2}'(\bar{k_j})]^{-1}(f^\dagger(k_j))^{-1}[{T_1}'(\bar{k_j})]^{-1}e^{2it\theta(\bar{k_j})}&0\end{pmatrix};
\end{equation}

                For $j\in\Delta_{k_0}^+$
                \begin{equation}
    \operatorname*{Res}_{k=k_j}M^{(out)}(k;\sigma_d^{out})=\lim_{k\to k_j}M^{(out)}(k;\sigma_d^{out})
    \begin{pmatrix}0&0\\{T_2}^{-1}(k_j)f(k_j){T_1}^{-1}(k_j)e^{2it\theta(k_j)}&0\end{pmatrix};
\end{equation}
                \begin{equation}
    \operatorname*{Res}_{k{=}\bar{k}}M^{(out)}(k;\sigma_d^{out})=\lim_{k\to \bar{k_j}}M^{(out)}(k;\sigma_d^{out})
    \begin{pmatrix}0&-T_1(\bar{k_j})f^\dagger(k_j)T_2(\bar{k_j})e^{-2it\theta(\bar{k_j})}\\0&0\end{pmatrix};
                 \end{equation}

    		\end{enumerate}
    	\end{RHP}
    In order to show the existence and uniqueness of solution corresponding to the above RH problem, we need to study the existence and uniqueness of
    RH problem \ref{RHP1} in the reflectionless case. In this special case, $M(k; x,t)$ has no contour, the RH problem \ref{RHP1} reduces to the following RH problem.
    \begin{RHP}\label{RHP8}
    		Given scattering data $\sigma_d=\{(k_j,f(k_j)\}_{j=1}^N$ and $\mathcal{Z}=\left\{k_j\right\}_{j=1}^{N}$. Find a matrix-valued function
    $M(k;x,t|\sigma_d)$ with following condition:
    		\begin{enumerate}[(i)]
    			\item Analyticity: ~$M(k;x,t|\sigma_d)$ is analytical in $k\in\mathbb{C}\setminus(\mathbb{R}\cup\mathcal{Z}\cup\bar{\mathcal{Z}})$;

    			\item Asymptotic behaviors:$M(k;x,t|\sigma_d)\sim I,\quad k\to\infty$;

                \item Residue conditions: $M(k;x,t|\sigma_d)$ has simple poles at each point in $k\in\mathcal{Z}\cup\bar{\mathcal{Z}}$ satisfying:

                \begin{equation}
    \operatorname{Res}\limits_{k=k_j}M(k;x,t|\sigma_d)=\lim\limits_{k\to k_j}M(k;x,t|\sigma_d)N_j, N_j=
    \begin{pmatrix}0&0\\f(k_j)e^{2it\theta(k_{j})}&0\end{pmatrix};
\end{equation}
                  \begin{equation}
    \operatorname{Res}_{k=k_j^*}M(k;x,t|\sigma_d)=\lim_{k\to k_j^*}M(k;x,t|\sigma_d)\widetilde{N}_j,\widetilde{N}_j=
    \begin{pmatrix}0&[f(k_j)e^{2it\theta(k_{j})}]^\dagger\\0&0\end{pmatrix}.
\end{equation}

    		\end{enumerate}
    	\end{RHP}
\begin{proposition} Given scattering data $\sigma_d=\{(k_j,f(k_j)\}_{j=1}^N$ and $\mathcal{Z}=\left\{k_j\right\}_{j=1}^{N}$he RH problem has unique solutions.
    \begin{equation}\nonumber
    q_{sol}(x,t;\sigma_d)=2i\lim_{k\to\infty}(kM(k;x,t|\sigma_d)_{UR}.
    \end{equation}
    \end{proposition}
   \begin{proof}  The uniqueness of the solution can be guaranteed by \text{Liouville's theorem}. As for the RH problem
    \begin{equation}\nonumber
    M_+(k;x,t|\sigma_d)=M_-(k;x,t|\sigma_d)V(k),
    \end{equation}
    which can be regularized by subtracting any pole contributions and the leading order asymptotics at infinity:
    \begin{equation}\nonumber
    \begin{aligned}\mathcal{M}(x,t;k)=M(k;x,t|\sigma_d)-I_{4}-\sum_{j=1}^{N}\left(\frac{\mathrm{Res~}_{k=k_j}M(k;x,t|\sigma_d)}{k-k_{j}}
    +\frac{\mathrm{Res~}_{k=k_j^*}M(k;x,t|\sigma_d)}{k-k_j^*}\right).\end{aligned}
\end{equation}
    Consequently, the piecewise holomorphic function $\mathcal{M}(x,t;k)$ satisfies
    \begin{equation}\nonumber
    \begin{aligned}&\mathcal{M}_{+}(x,t;k)-\mathcal{M}_{-}(x,t;k)=M_{-}(k;x,t|\sigma_d)(V(k)-I_{4}),&&k\in\mathbb{R},\\
    &\mathcal{M}(x,t;k)\to0,&&k\to\infty.\end{aligned}
\end{equation}

    Using Sokhotski-Plemelj formula, we have
    \begin{equation}\nonumber
    \mathcal{M}(x,t;k)=\frac1{2\pi i}\int_\mathbb{R}\frac{M(\zeta;x,t|\sigma_d)(V(\zeta)-I_4)}{\zeta-k}\mathrm{d}\zeta.
\end{equation}
    Moreover when $V=I_4$ we can solve the RH problem \ref{RHP8} using the system of algebraic-integral equations
    \begin{equation}\nonumber
    M_{UL}(k|\sigma_d)=I_{2}+\sum_{j=1}^{N}\frac{e^{2i\theta(k_{j})}M_{UR}(k_{j})f(k_{j})}{k-k_{j}},
\end{equation}
    \begin{equation}\nonumber
    M_{UR}(k|\sigma_d)=-\sum_{j=1}^{N}\frac{e^{-2i\theta(\bar{k_{j}})}M_{UL}(\bar{k_{j}}|\sigma_d)f^\dagger(k_{j})}{k-\bar{k_{j}}}.
\end{equation}
    Thus
    \begin{equation}\nonumber
    q_{sol}(x,t;\sigma_d)=2i\lim_{k\to\infty}(kM(k;x,t|\sigma_d)_{UR}=-2i\sum_{j=1}^{N}{[e^{{-2i\theta(\bar{k_{j}})}}M_{UL}(\bar{k_{j}})f^{{\dagger}}(k_j)]}.
    \end{equation}
    Let
    \begin{equation}\nonumber
    \mathbf{h}(k)=e^{2i\theta(k)}f(k),\quad\mathbf{F}(k)=-2i\mathbf{M}_{\mathrm{UL}}(k)\mathbf{h}^\dagger(\bar{k}),
\end{equation}
    \begin{equation}\nonumber
    \mathbf{G}(k)=\mathbf{F}(k)+2\mathbf{ih}^\dagger(\bar{k})+
    \sum_{j=1}^{N}\sum_{l=1}^{N}(\frac{F(\bar{k_{l}})\mathbf{h}(k_{j})\mathbf{h}^{\dagger}(\bar{k})}{(k-k_{j})(k_{j}-k_{j})}).
\end{equation}
    In the reflectionless case, the solution to the spin-1 GP equation (1.1) can be expressed as follows:
    \begin{equation}\label{1000}
    q=\sum_{j=1}^N{\mathbf{F}(x,t;\bar{k}_j)},
    \end{equation}
    where $\mathbf{F}(x,t;\bar{k}_j)$ is the solution to the following algebraic system
    \begin{equation}\label{ds}
    \begin{cases}
    \mathbf{G}(x,t;\bar{k}_1)&=0,\\\mathbf{G}(x,t;\bar{k}_2)&=0,\\&\vdots\\\mathbf{G}(x,t;\bar{k}_N)&=0.\end{cases}
\end{equation}
   We verify the existence and uniqueness of the solution to the algebraic system \eqref{ds}. The proof can be shown in a similar way as the reference \cite{boo-27}.
    They have proof that the determinant of the matrix $a(k)$ has N zeros, denoted by $k_1,\ldots,k_N$ , which are all in the upper half of the complex plane
    $\mathbb{C}^{+}$ and are not on the real axis. The zeros have multiplicities $m_1+1,\ldots,m_N+1$, respectively.
   \begin{equation}\nonumber
    \begin{aligned}
\mathbf{G}^\mathrm{T}(x,t;k)=& \mathbf{F}^\mathrm{T}(x,t;k)+2i\mathbf{h}^\dagger(x,t;\bar{k}) \\
&\begin{aligned}&+\sum_{k=1}^{N}\sum_{l=1}^{N}\left.
\left(\frac{\mathbf{h}^{\dagger}(x,t;\bar{k})\mathbf{h}(x,t;k_j)\mathbf{F}^{\mathrm{T}}(x,t;\bar{k_{l}})}{(k-k_j)(k_j-\bar{k_{j}})}\right).\right.\end{aligned}
\end{aligned}
\end{equation}
    Define $2N\times 2N$ matrix $\mathbf{X}$, let
    \begin{equation}\nonumber
    \mathbf{X}=\begin{pmatrix}\mathbf{X}_{11}&\mathbf{X}_{12}&\cdots\mathbf{X}_{1N}\\\mathbf{X}_{21}&\mathbf{X}_{22}&\cdots\mathbf{X}_{2N}\\
    \vdots\\\mathbf{X}_{N1}&\mathbf{X}_{N2}&\cdots\mathbf{X}_{NN}\end{pmatrix},\quad\mathbf{X}_{jk}=\begin{pmatrix}\mathbf{X}_{jk}^{11}&\mathbf{X}_{jk}^{12}\\
    \mathbf{X}_{jk}^{21}&\mathbf{X}_{jk}^{22}\end{pmatrix},
\end{equation}
    where $\mathbf{X}_{jk}$ is a $2\times2$ matrix for , we call $\mathbf{X}_{jk}^{rs}$ the $(r_j,s_k)\text{-entry}$ of $\mathbf{X}$.
    \begin{equation}\nonumber
    \mathbf{H}=\mathbf{H}(x,t)=-2i\left(\mathbf{h}(k_1),\mathbf{h}(k_1),\cdots\mathbf{h}(k_N)\right)^\dagger.
\end{equation}
    Define $2N\times 2N$ matrix-valued functions of $\left(x,t\right)$:
    \begin{equation}\nonumber
    \mathbf{A}=\mathbf{A}(x,t),
    \end{equation}
    where
    \begin{equation}\nonumber
    \mathbf{A}_{jk}=-i\frac{\mathbf{h}^{\dagger}(k_j)}{\bar{k_j}-k_k},
\end{equation}
    So \eqref{1000} can become
    \begin{equation}\nonumber
    q(x,t)=\alpha\left(I_{2N}+\mathbf{A}(x,t)\mathbf{\bar{A}}(x,t)\right)^{-1}\mathbf{H}(x,t),
\end{equation}
    \begin{equation}\nonumber
    \alpha=(I_2,I_2\cdots I_2)_{2\times2N}.
\end{equation}
    According to Lemma \ref{lemma3.5}, the algebraic system
    \begin{equation}\label{ds2}
    \begin{pmatrix}I_{2N}+\mathbf{A}(x,t)\mathbf{\bar{B}}(x,t)\end{pmatrix}\tilde{\mathbf{F}}(x,t)=\mathbf{H}(x,t),
\end{equation}
    with the unknown column vector
    \begin{equation}\nonumber
    \tilde{\mathbf{F}}=\tilde{\mathbf{F}}(x,t)=\left(\mathbf{F}(\bar{k}_1),\ldots,\mathbf{F}(\bar{k}_N)\right)^\mathrm{T},
\end{equation}
    has a unique solution, we can now prove the existence and uniqueness of the RH problem.
    \end{proof}

    \begin{lemma}\label{lemma3.5}
     The solution to the algebraic system \eqref{ds2} exists uniquely.
    \end{lemma}
   \begin{proof} Define a $2N\times2N$ matrix-valued function $\mathbf{\hat{A}}(x,t)$ with $(j,l)\text{-entry}$,
   \begin{equation}\nonumber
    \hat{\mathbf{A}}_{jl}=\begin{cases}{[\mathbf{h}(x,t;k_j)]^\dagger},&j=l,\\\mathbf{0},&\text{otherwise,}\end{cases}
\end{equation}
    and a $2N\times2N$ matrix ${\tilde{\mathbf{A}}}$ with $(j,l)\text{-entry}$,
    \begin{equation}\nonumber
    \mathbf{A}_{jl}=-i\frac{1}{\bar{k_{j}}-k_{j}}I_{2}.
\end{equation}
    Indeed, upon performing direct calculations, we have discovered
    \begin{equation}\nonumber
    \hat{\mathbf{A}}(x,t)\tilde{\mathbf{A}}=\mathbf{A}(x,t).
\end{equation}
    Define
    \begin{equation}\nonumber
    h_{j}(y)=e^{-i\overline{k}_{j}y},
\end{equation}
    and an inner product $\langle\cdot,\cdot\rangle{:}$
    \begin{equation}\nonumber
    \langle h_{j}(y),h_{l}(y)\rangle=\int_0^{+\infty}h_{j}(y)\bar{h_{l}(y)}\mathrm{d}y.
\end{equation}
    So
    \begin{equation}\nonumber
    \tilde{\mathbf{A}}_{jl}=\langle h_{j},h_{l}\rangle I_{2},
\end{equation}
    which means $\tilde{\mathbf{A}}$ is a positive definite Hermitian matrix, we can get
    \begin{equation}\nonumber
    \begin{aligned}\mathbf{A}\bar{\mathbf{A}}&=(\hat{\mathbf{A}}\tilde{\mathbf{A}})\overline{(\hat{\mathbf{A}}(x,t)\tilde{\mathbf{A}})}\\
    &=\hat{\mathbf{A}}(x,t)\tilde{\mathbf{A}}\hat{\mathbf{A}}^\dagger(x,t),\end{aligned}
\end{equation}
    we can conclude that $\hat{\mathbf{A}}(x,t)\tilde{\mathbf{A}}\hat{\mathbf{A}}^\dagger(x,t)$ is Hermitian positive definite matrices. Therefore
    \begin{equation}\nonumber
    \begin{aligned}
\det\left(I_{2N}+\mathbf{A}(x,t)\mathbf{\bar{A}}(x,t)\right)& =\det\left(I_{2N}+
\hat{\mathbf{A}}(x,t)\tilde{\mathbf{A}}\hat{\mathbf{A}}^{\dagger}(x,t)\bar{\tilde{\mathbf{A}}}\right) \\
&=\det\left(I_{2N}+\hat{\mathbf{A}}(x,t)\tilde{\mathbf{A}}\hat{\mathbf{A}}^{\dagger}(x,t)\mathbf{C}^{2}\right) \\
&=\det\left(I_{2N}+\mathbf{C}\hat{\mathbf{A}}(x,t)\tilde{\mathbf{A}}\hat{\mathbf{A}}^{\dagger}(x,t)\mathbf{C}\right)>1.
\end{aligned}
\end{equation}
    Using $\text{Cramer's Rule,}$, we can now prove the existence and uniqueness of the solution to the algebraic system \eqref{ds2}.
    \end{proof}

    Introducing symbols $\begin{aligned}\Delta\subseteq\{1,2,\cdots,N\},&&\bigtriangledown=\Delta^c=\{1,2,\cdots,N\}\setminus\Delta,\end{aligned}$ define
    \begin{equation}\label{ak}
    a_\Delta=\prod_{j\in\Delta}\frac{k-k_j}{k-\bar{k}_j},\quad a_\bigtriangledown=\prod_{j\in\bigtriangledown}\frac{k-k_j}{k-\bar{k}_j},
\end{equation}
    and make another transformation
    \begin{equation}\nonumber
    M^\Delta(k|\sigma_d^\Delta)=M(k|\sigma_d)a_\Delta(k)^{\sigma_4},
\end{equation}
    then $M^\Delta(k|\sigma_d^\Delta)$ satisfies the following non-reflective RH problem.
\begin{RHP}\label{RHP9}
    		Find a matrix valued function $M^\Delta(k|\sigma_d^\Delta)$ admits:
    		\begin{enumerate}[(i)]
    			\item Analyticity: ~$M^\Delta(k|\sigma_d^\Delta)$ is analytical in $k\in\mathbb{C}\setminus(\mathbb{R}\cup\mathcal{Z}\cup\bar{\mathcal{Z}})$;

    			\item Asymptotic behaviors:$M^\Delta(k|\sigma_d^\Delta)\sim I,\quad k\to\infty$;

                \item Residue conditions: $M^\Delta(k|\sigma_d^\Delta)$ has simple poles at each point in $k\in\mathcal{Z}\cup\bar{\mathcal{Z}}$ satisfying:
                For $j\in\Delta_{k_0}^-$
                \begin{equation}
    \operatorname*{Res}_{k=k_j}M^{(out)}(k;\sigma_d^{out})=\lim_{k\to k_j}M^{(out)}(k;\sigma_d^{out})\begin{pmatrix}0&[(a_\Delta)'(k_j)]^{-2}f^{-1}(k_j)e^{-2it\theta(k_j)}
    \\0&0\end{pmatrix};
                 \end{equation}
               \begin{equation}
    \operatorname*{Res}_{k{=}\bar{k}}M^{(out)}(k;\sigma_d^{out})=\lim_{k\to \bar{k_j}}M^{(out)}(k;\sigma_d^{out})\begin{pmatrix}0&0\\
    -[(a_\Delta)'(\bar{k_j})]^{-2}(f^\dagger(k_j))^{-1}e^{2it\theta(\bar{k_j})}&0\end{pmatrix};
\end{equation}

                For $j\in\Delta_{k_0}^+$
                \begin{equation}
    \operatorname*{Res}_{k=k_j}M^{(out)}(k;\sigma_d^{out})=\lim_{k\to k_j}M^{(out)}(k;\sigma_d^{out})\begin{pmatrix}0&0\\
    {a_\Delta}^{2}(k_j)f(k_j)e^{2it\theta(k_j)}&0\end{pmatrix};
\end{equation}
                \begin{equation}
    \operatorname*{Res}_{k{=}\bar{k}}M^{(out)}(k;\sigma_d^{out})=\lim_{k\to \bar{k_j}}M^{(out)}(k;\sigma_d^{out})
    \begin{pmatrix}0&-[(a_\Delta)'(\bar{k_j})]^{-2}f^\dagger(k_j)e^{-2it\theta(\bar{k_j})}\\0&0\end{pmatrix}.
                 \end{equation}

    		\end{enumerate}
    	\end{RHP}

     \begin{proposition} The RH problem $\ref{RHP9}$ has a unique solution for the non-reflective scattering data
     $\sigma_d^\Delta=\{k_j,{a_\Delta}^{2}(k_j)f(k_j)\}_{j=1}^N$. and
     \begin{equation}\label{qsol}
    q_{sol}(x,t;\sigma_d^\Delta)=2i\lim_{k\to\infty}[kM^\Delta(k|\sigma_d^\Delta)]_{UR}=2i\lim_{k\to\infty}[kM(k|\sigma_d)]_{UR}=q_{sol}(x,t;\sigma_d).
\end{equation}
\end{proposition}

    We note that $M^{(out)}(k|\sigma_d^{out})$ is a reflected soliton solution, and the reflection mainly comes from $T(k_j)$.
    In order to combine $M^{(out)}(k|\sigma_d^{out})$ with the non-reflective scattering data $\sigma_d^\Delta=\{k_j,{a_\Delta}^{2}(k_j)f(k_j)\}_{j=1}^N$
    corresponds to soliton solutions, we in the definition $\eqref{ak}$, take $\Delta=\Delta_{k_0}^-$, then
    \begin{equation}\nonumber
    a_{\Delta_{k_0}}(k)=\prod_{j\in\Delta_{k_0}^-}\frac{k-k_j}{k-\bar{k}_j},\quad a_{\Delta_{k_0}^+}(k)=\prod_{j\in\Delta_{k_0}^+}\frac{k-k_j}{k-\bar{k}_j},
\end{equation}
    and
    \begin{equation}\nonumber
    T(k)=a_{\Delta_{k_0}^-}(k)^{-1}\delta(k),
\end{equation}
    \begin{equation}\nonumber
    T_{j}(k_j)^{-2}=a_{\Delta{k_0}^-}(k_j)^2\delta_{j}(k_j)^{-2},\quad(1/T_{j})^{\prime}(k_j)^{-2}=a_{\Delta_{k_0}^-}^{\prime}(k_j)^{-2}\delta_{j}(k_j)^2.
\end{equation}
    Therefore, the above RH problem $\ref{RHP7}$ can be rewritten as follows.
    \begin{RHP}\label{RHP10}
    		Find a matrix valued function $M^{(out)}(k;\sigma_d^{out})$ admits:
    		\begin{enumerate}[(i)]
    			\item Analyticity: ~$M^{(out)}(k;\sigma_d^{out})$ is analytical in $k\in\mathbb{C}\setminus(\mathbb{R}\cup\mathcal{Z}\cup\bar{\mathcal{Z}})$;

    			\item Asymptotic behaviors:$M^{(out)}(k;\sigma_d^{out})\sim I,\quad k\to\infty$;

                \item Residue conditions: $M^{(out)}(k;\sigma_d^{out})$ has simple poles at each point in $k\in\mathcal{Z}\cup\bar{\mathcal{Z}}$ satisfying:
                For $j\in\Delta_{k_0}^-$
                \begin{equation}
    \operatorname*{Res}_{k=k_j}M^{(out)}(k;\sigma_d^{out})=\lim_{k\to k_j}M^{(out)}(k;\sigma_d^{out})
    \begin{pmatrix}0&[(a_\Delta)'(k_j)]^{-2}\delta_{1}(k_j)f^{-1}(k_j)\delta_{2}(k_j)e^{-2it\theta(k_j)}\\0&0\end{pmatrix};
                 \end{equation}
               \begin{equation}
    \operatorname*{Res}_{k{=}\bar{k}}M^{(out)}(k;\sigma_d^{out})=\lim_{k\to \bar{k_j}}M^{(out)}(k;\sigma_d^{out})
    \begin{pmatrix}0&0\\-[(a_\Delta)'(\bar{k_j})]^{-2}\delta_{2}(\bar{k_j})(f^\dagger(k_j))^{-1}\delta_{1}(\bar{k_j})e^{2it\theta(\bar{k_j})}&0\end{pmatrix};
\end{equation}

                For $j\in\Delta_{k_0}^+$
                \begin{equation}
    \operatorname*{Res}_{k=k_j}M^{(out)}(k;\sigma_d^{out})=\lim_{k\to k_j}M^{(out)}(k;\sigma_d^{out})
    \begin{pmatrix}0&0\\{a_\Delta}^{2}(k_j)\delta_{2}(k_j)^{-1}f(k_j)\delta_{1}(k_j)^{-1}e^{2it\theta(k_j)}&0\end{pmatrix};
\end{equation}
                \begin{equation}
    \operatorname*{Res}_{k{=}\bar{k}}M^{(out)}(k;\sigma_d^{out})=\lim_{k\to \bar{k_j}}M^{(out)}(k;\sigma_d^{out})
    \begin{pmatrix}0&-[(a_\Delta)'(\bar{k_j})]^{-2}\delta_{1}(\bar{k_j})f^\dagger(k_j)\delta_{2}(\bar{k_j})e^{-2it\theta(\bar{k_j})}\\0&0\end{pmatrix}.
                 \end{equation}

    		\end{enumerate}
    	\end{RHP}

    Note that
    \begin{equation}\nonumber
    M^{(out)}(k|\sigma_d^{\mathrm{out}})=M(k|\sigma_d^{\Delta_{k_0}^-})\begin{pmatrix}\delta_1(k)&0\\0&\delta_2^{-1}(k)\end{pmatrix}|_{\gamma(k)=0,k\in\mathbb{R}^-}.
\end{equation}

     \begin{proposition} The RH problem \ref{RHP10} has a unique solution, and the $N$-soliton solution of the spin-1 GP equation with reflective
      scattering data satisfies the $N$-soliton solution with non-reflective scattering data
    \begin{equation}\nonumber
    q_{sol}(x,t;\sigma_d^{\mathrm{out}})=q_{sol}(x,t;\sigma_d^{\Delta_{k_0}^-}).
\end{equation}
\end{proposition}
    \begin{proof}
    \begin{equation}\nonumber
    \begin{gathered}
q_{sol}(x,t;\sigma_{d}^{out}) =2i\lim_{k\to\infty}[kM^{(out)}(k|\sigma_d^{out})]_{UR}
=2i\lim_{k\to\infty}[kM^{\Delta_{k_0}^-}(k|\sigma_d^{\Delta_{k_0}})\delta(k)^{\sigma_3}]_{UR} \\
=2i\lim_{k\to\infty}[kM^{\Delta_{k_0}^-}(k|\sigma_d^{\Delta_{k_0}^-})]_{UR}=q_{sol}(x,t,\sigma_d^{\Delta_{k_0}^-}).
\end{gathered}
\end{equation}
\end{proof}
    We now consider the long-time behavior of soliton solutions. Not all discrete spectrum have contribution as $t\to\infty.$ Give pairs points
     $x_1\leq x_2\in\mathbb{R}$ and velocities $v_1\leq v_2\in\mathbb{R}$, we define a cone
    \begin{equation}\nonumber
    S(x_1,x_2,v_1,v_2)=\{(x,t):x=x_0+vt, x_0\in[x_1,x_2], v\in[v_1,v_2]\}.
    \end{equation}
    Denote $\mathcal{I}=[-\frac{v_1}4,-\frac{v_2}4],$ $\Delta_{\mathcal{I}}^-=\{j:Rek_j<-v_2/4\}, \Delta_{\mathcal{I}}^+=\{j:Rek_j>-v_1/4\},$ then we
     have the following proposition

    \centerline{\begin{tikzpicture}[scale=0.8]
\path [fill=pink] (-1,3)--(0,0) to (2,0) -- (3,3);
\path [fill=pink] (-1,-3)--(0,0) to (2,0) -- (3,-3);
\draw[-][thick](-4,0)--(-3,0);
\draw[-][thick](-3,0)--(-2,0);
\draw[-][thick](-2,0)--(-1,0);
\draw[-][thick](-1,0)--(0,0);
\draw[-][thick](0,0)--(1,0);
\draw[-][thick](1,0)--(2,0);
\draw[-][thick](2,0)--(3,0);
\draw[-][thick](3,0)--(4,0);
\draw[->][thick](4,0)--(5,0)[thick]node[right]{$x$};
\draw[<-][thick](-2,3)[thick]node[right]{$t$}--(-2,2);
\draw[-][thick](-2,2)--(-2,1);
\draw[-][thick](-2,1)--(-2,0);
\draw[-][thick](-2,0)--(-2,-1);
\draw[-][thick](-2,-1)--(-2,-2);
\draw[-][thick](-2,-2)--(-2,-3);
\draw[fill] (0,0) circle [radius=0.08];
\draw[fill] (2,0) circle [radius=0.08];
\draw[fill] (-0.5,0)node[below]{$x_{2}$};
\draw[fill] (2.5,0)node[below]{$x_{1}$};
\draw[fill] (3.5,3)node[above]{$x=v_{2}t+x_{2}$};
\draw[fill] (3,-3)node[below]{$x=v_{1}t+x_{2}$};
\draw[fill] (-1,-3)node[below]{$x=v_{2}t+x_{1}$};
\draw[fill] (-2,3)node[above]{$x=v_{2}t+x_{1}$};
\draw[fill] (1,2)node[below]{$S$};
\draw[-][thick](-1,3)--(0,0);
\draw[-][thick](3,3)--(2,0);
\draw[-][thick](-1,-3)--(0,0);
\draw[-][thick](3,-3)--(2,0);
\end{tikzpicture}}\label{Figure 4.}
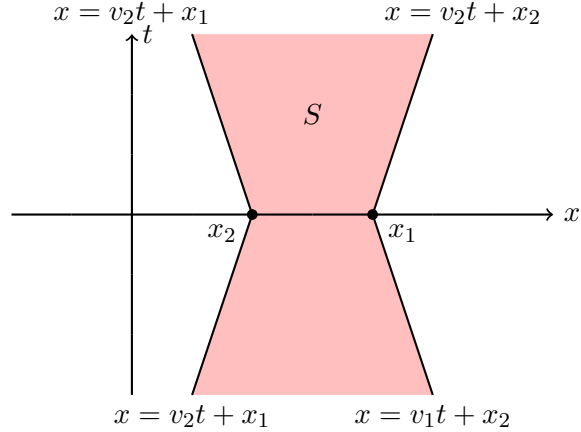
\centerline{\noindent {\small \textbf{Figure 3.4} Space-time $S(x_{1},x_{2},v_{1},v_{2})$.}}

    \begin{proposition} Given scattering  data $\sigma_d^\Delta=\{k_j,{a_\Delta}^{2}(k_j)\delta_{2}(k_j)^{-1}f(k_j)\delta_{1}(k_j)^{-1}\}_{j=1}^N$.
    At $t\to+\infty $ with $(x,t)\in S(x_1,x_2,v_1,v_2)$, we have
    \begin{equation}\nonumber
    M^{\Delta_{k_0}}(k|\sigma_d^{\Delta_{k_0}})=(I+O(e^{-8\mu t}))M^{\Delta_{\mathcal{I}}}(k|\widehat{\sigma}_d(\mathcal{I})),
\end{equation}
    where $\mu=\mu(\mathcal{I})=\min_{k_j\in\mathcal{Z}\setminus\mathcal{Z}(\mathcal{I})}\{Im(k_j)\mathrm{dist}(Rek_j,\mathcal{I})\}$.

\end{proposition}

\centerline{\begin{tikzpicture}[scale=0.8]
\path [fill=pink] (1.5,3)--(-1.5,3) to (-1.5,-3) -- (1.5,-3);
\draw[-][thick](-4,0)--(-3,0);
\draw[-][thick](-3,0)--(-2,0);
\draw[-][thick](-2,0)--(-1,0);
\draw[-][thick](-1,0)--(0,0);
\draw[-][thick](0,0)--(1,0);
\draw[-][thick](1,0)--(2,0);
\draw[-][thick](2,0)--(3,0);
\draw[->][thick](3,0)--(4,0)[thick]node[right]{$Rek$};
\draw[-][thick](-1.5,3)--(-1.5,2);
\draw[-][thick](-1.5,2)--(-1.5,1);
\draw[-][thick](-1.5,1)--(-1.5,0);
\draw[-][thick](-1.5,0)--(-1.5,-1);
\draw[-][thick](-1.5,-1)--(-1.5,-2);
\draw[-][thick](-1.5,-2)--(-1.5,-3);
\draw[-][thick](1.5,3)--(1.5,2);
\draw[-][thick](1.5,2)--(1.5,1);
\draw[-][thick](1.5,1)--(1.5,0);
\draw[-][thick](1.5,0)--(1.5,-1);
\draw[-][thick](1.5,-1)--(1.5,-2);
\draw[-][thick](1.5,-2)--(1.5,-3);
\draw[fill] (1.5,0) circle [radius=0.05];
\draw[fill] (-1.5,0) circle [radius=0.05];
\draw[fill] (2,0)node[below]{$-\frac{v_{1}}{4}$};
\draw[fill] (-2,0)node[below]{$-\frac{v_{2}}{4}$};
\draw[fill] (3,1)node[below]{$k_{5}$} circle [radius=0.05];
\draw[fill] (3,-1)node[below]{$\bar{k_{5}}$} circle [radius=0.05];
\draw[fill] (-1,-1)node[below]{$\bar{k_{6}}$} circle [radius=0.05];
\draw[fill] (0.5,0.5)node[below]{$k_{7}$} circle [radius=0.05];
\draw[fill] (0.5,-0.5)node[below]{$\bar{k_{7}}$} circle [radius=0.05];
\draw[fill] (-1,1)node[below]{$k_{6}$} circle [radius=0.05];
\draw[fill] (2,3)node[below]{$k_{2}$} circle [radius=0.05];
\draw[fill] (2,-3)node[below]{$\bar{k_{2}}$} circle [radius=0.05];
\draw[fill] (0.5,2.5)node[below]{$k_{1}$} circle [radius=0.05];
\draw[fill] (0.5,-2.5)node[below]{$\bar{k_{1}}$} circle [radius=0.05];
\draw[fill] (-0.5,2.8)node[below]{$k_{3}$} circle [radius=0.05];
\draw[fill] (-0.5,-2.8)node[below]{$\bar{k_{3}}$} circle [radius=0.05];
\draw[fill] (-3,1.5)node[below]{$k_{4}$} circle [radius=0.05];
\draw[fill] (-3,-1.5)node[below]{$\bar{k_{4}}$} circle [radius=0.05];
\end{tikzpicture}}\label{Figure 3.5}
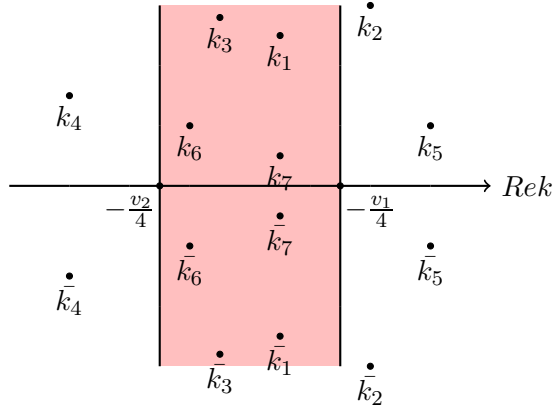
\centerline{\noindent {\small \textbf{Figure 3.5} For fixed $v_{1}<v_{2}$, $I=\left[-\frac{v_{2}}{4},-\frac{v_{1}}{4}\right]$.}}

    \begin{proof}
    As for $t>0$, $(x,t)\in S(x_1,x_2,v_1,v_2)$, we obtain
    \begin{equation}\nonumber
    -v_2/4<k_0+x_0/(4t)<-v_1/4.
    \end{equation}
    Since $x_1<x_0<x_2,$ we know that $x_0/(4t)\to0,$ and $t\to\infty,$ so for sufficiently large $t$, we have $-v_2/4\leqslant k_0\leqslant-v_1/4.$

    In RH problem \ref{RHP9}, let $\Delta=\Delta_{k_0}^-$, then $\bigtriangledown=\Delta_{k_0}^+$.

    $\begin{pmatrix}i\end{pmatrix}$ For $j\in\Delta_{k_0}^-=\Delta_{\mathcal{I}}^-\cup\Delta_{k_0}^-(\mathcal{I}),$ we have
    \begin{equation}\nonumber
    N_j^{\Delta_{k_0}^-}=\begin{pmatrix}0&[(a_\Delta)'(k_j)]^{-2}f^{-1}(k_j)e^{-2it\theta(k_j)}\\0&0\end{pmatrix}, j\in\Delta_{k_0}^-.
\end{equation}
    Particularly, for $j\in\Delta_{\mathcal{I}}\longleftrightarrow Rek_j<-v_2/4,$
    \begin{equation}\nonumber
    \begin{aligned}
-Im(k_{j})Re(k_{j}+v/4)& \geqslant\min_{k_{j}\in\mathcal{Z}\setminus\mathcal{Z}(\mathcal{I})}\{Im(k_{j})\mathrm{dist}(Rek_{j},\mathcal{I})\} \\
&=\mu>-\min_{k_{j}\in\mathcal{Z}^{-}(\mathcal{I})}\{Im(k_{j})(Rek_{j}+v_{2}/4)\}>0.
\end{aligned}
\end{equation}
    Therefore
    \begin{equation}\nonumber
    |[(a_\Delta)'(k_j)]^{-2}f^{-1}(k_j)e^{-2it\theta(k_j)}|=|f^{-1}(k_j)||e^{2x_{0}Im(k_{j})}e^{8tIm(k_{j})Rek_{j}+v/2)}|=\mathcal{O}(e^{-8\mu t}).
\end{equation}
    For $j\in\Delta_{k_0}^-(\mathcal{I})\Longleftrightarrow-v_2/4\leqslant Re(k_j)\leqslant k_0,$
    \begin{equation}\nonumber
    |[(a_\Delta)'(k_j)]^{-2}f^{-1}(k_j)e^{-2it\theta(k_j)}|\leqslant ce^{-8t\text{Im}(k_j)(\text{Re}(k_j)-k_0)}=\mathcal{O}(1).
\end{equation}

    $\begin{pmatrix}ii\end{pmatrix}$ For $j\in\Delta_{k_0}^+=\Delta_{\mathcal{I}}^+\cup\Delta_{k_0}^+(\mathcal{I}),$ we can get
    \begin{equation}\nonumber
    N_j^{\Delta_{k_0}^+}=\begin{pmatrix}0&0\\{a_{\Delta_{k_0}^+}}^{2}(k_j)f(k_j)e^{2it\theta(k_j)}&0\end{pmatrix}, j\in\Delta_{k_0}^+.
\end{equation}
    Particularly, for $j\in\Delta_{\mathcal{I}}^+\longleftrightarrow Rek_j>-v_1/4,$
    \begin{equation}\nonumber
    \begin{aligned}
Im(k_{j})Re(k_{j}+v/4)& \geqslant\min_{k_{j}\in\mathcal{Z}\setminus\mathcal{Z}(\mathcal{I})}\{Im(k_{j})\mathrm{dist}(Rek_{j},\mathcal{I})\} \\
&=\mu>\min_{k_{j}\in\mathcal{Z}^{-}(\mathcal{I})}\{Im(k_{j})(Rek_{j}+v_{1}/4)\}>0.
\end{aligned}
\end{equation}
    Therefore
    \begin{equation}\nonumber
    |{a_{\Delta_{k_0}^+}}^{2}(k_j)f(k_j)e^{2it\theta(k_j)}|=|{a_{\Delta_{k_0}^+}}^{2}(k_j)||f^{-1}(k_j)||e^{2x_{0}Im(k_{j})}e^{-8tIm(k_{j})Rek_{j}+v/2)}|
    =\mathcal{O}(e^{-8\mu t}).
\end{equation}
    For $j\in\Delta_{k_0}^+(\mathcal{I})\Longleftrightarrow k_0\leqslant Re(k_j)\leqslant-v_2/4,$
    \begin{equation}\nonumber
    |{a_{\Delta_{k_0}^+}}^{2}(k_j)f(k_j)e^{2it\theta(k_j)}|\leqslant ce^{-8t\text{Im}(k_j)(\text{Re}(k_j)-k_0)}=\mathcal{O}(1).
\end{equation}
    Thus for $t\to\infty $, and $(x,t)\in S(x_1,x_2,v_1,v_2)$, we have

    \begin{equation}\nonumber
    ||N_j^{\Delta_{\mathcal{I}}^\pm}||=\begin{cases}\mathcal{O}(1)&k_j\in\mathcal{Z}(\mathcal{I}),\\
    \mathcal{O}(e^{-8\mu t})&k_j\in\mathcal{Z}\setminus\mathcal{Z}(\mathcal{I}).\end{cases}
\end{equation}
    For each discrete spectrum $k_j \in \mathcal{Z}\setminus\mathcal{Z}(\mathcal{I})$, make disks $D_{k}$ of sufficiently small radius so that they do not
     intersect each other and define functions
    \begin{equation}\nonumber
    \omega(k)=\begin{cases}I_4-\frac{1}{k-k_j}N_j^{\Delta_{\mathcal{I}}^\pm},&\quad k\in D_j,\\\\I_4-\frac{1}{k-\bar{k}_j}\bar{N}_j^{\Delta_{\mathcal{I}}^\pm},
    &\quad k\in\overline{D}_j,\\\\I_4,&\quad\text{otherwise.}\end{cases}
\end{equation}
    Then we introduce a new transformation to convert the poles $k_j \in \mathcal{Z}\setminus\mathcal{Z}(\mathcal{I})$ to jumps which will decay to identity
    matrix exponentially,
    \begin{equation}\label{MM}
    \widehat{M}^{\Delta_{k_0}^\pm}(k|\sigma_d^{\Delta_{k_0}^\pm})=M^{\Delta_{k_0}^\pm}(k|\sigma_d^{\Delta_{k_0}^\pm})\omega(k).
\end{equation}
    Direct calculation shows that
    \begin{equation}\nonumber
    \widehat{M}_+^{\Delta_{k_0}^\pm}(k|\sigma_d^{\Delta_{k_0}^\pm})=\widehat{M}_-^{\Delta_{k_0}^\pm}(k|\widehat{\sigma}_d)\widehat{V}(k),\quad k\in\widehat{\Sigma}
    =\cup_{k_l\in\mathcal{Z}\setminus\mathcal{I}(\mathcal{I})}\partial D_j\cup\partial\bar{D}_j,
\end{equation}
    where jump matrices $\widehat{V}(k)$ satisfy
    \begin{equation}\nonumber
    ||\widehat{V}(k)-I||_{L^\infty(\widehat{\Sigma})}=\mathcal{O}(e^{-8\mu t}).
\end{equation}
    As in RH problem \ref{RHP9} again, make $\Delta=\Delta_\mathcal{I}$, then the $M^{\Delta_{\mathcal{I}}}(k|\widehat{\sigma}_d(\mathcal{I}))$
    and $\widehat M^{\Delta_{\mathcal{I}}^{\pm}}(k|\widehat{\sigma}_d)$ have the same poles and residue conditions in $\Delta_{\mathcal{I}}$, so
    \begin{equation}\nonumber
    \mathcal{E}(k)=\widehat{M}^{\Delta_{k_0}^\pm}(k|\sigma_d^{\Delta_{k_0}^\pm})\left[M^{\Delta_\mathcal{I}}(k|\widehat{\sigma}_d(\mathcal{I}))\right]^{-1},
\end{equation}
    has no poles and satisfies jump condition
    \begin{equation}\label{MMM}
    \mathcal{E}_+(k)=\mathcal{E}_-(k)V_\mathcal{E},
\end{equation}
    where $V_{\mathcal{E}}=M^{\Delta_\mathcal{I}}\widehat{V}[M^{\Delta_\mathcal{I}}]^{-1}\sim\widehat{V}$ satisfies
    \begin{equation}\nonumber
    ||V_{\mathcal{E}}(k)-I||_{L^{\infty}(\widehat{\Sigma})}=\mathcal{O}(e^{-8\mu t}).
\end{equation}
    Based on the properties of small norm RH problem, we know that $\mathcal{E}(k)$ exists and
    \begin{equation}\nonumber
    \mathcal{E}(k)=I+(O)\left(e^{-8\mu t}\right),\quad t\to+\infty.
\end{equation}
    Finally, according to \eqref{MM} and \eqref{MMM}, we obtain the following conclusion
    \begin{equation}\nonumber
    M^{\Delta_{k_0}^\pm}(k|\widehat{\sigma}_d)=(I+\mathcal{O}(e^{-8\mu t}))M^{\Delta_{\mathcal{I}}}(k|\widehat{\sigma}_d(\mathcal{I})).
\end{equation}
\end{proof}

    \begin{corollary} Suppose that $q_{sol}$ is the soliton solutions corresponding to the scattering data $\sigma_d^\Delta=\{k_j,{a_\Delta}^{2}(k_j)f(k_j)\}_{j=1}^N$
    of the spin-1 GP equation, as $(x,t)\in S(x_1,x_2,v_1,v_2)$, and $t\to+\infty$,
    \begin{equation}\label{775}
    q_{sol}(x,t;\sigma_d^{out})=q_{sol}(x,t;\sigma_d^{\Delta_{k_0}^\pm})=q_{sol}(x,t;\widehat{\sigma}_d(\mathcal{I}))+\mathcal{O}(e^{-8\mu t}).
\end{equation}
\end{corollary}

\subsection{\it{Asymptotic analysis on a pure $\bar{\partial}\text{-problem}$}}\label{s:3.4} 	
    Now we consider the long time asymptotics behavior of $M^{(3)}(k;x,t)$. Note that
    \begin{equation}\label{MM3}
    M^{(3)}(k)=M^{(2)}(k)(M_{RHP}^{(2)})^{-1}(k).
   \end{equation}
   \begin{RHP}\label{RHP5}
    		Find a matrix valued function $M^{(3)}(k)=M^{(3)}(k;x,t)$ admits:
    		\begin{enumerate}[(i)]
    			\item ~$M^{(3)}(k;x,t))$ is continuous in $k\in\mathbb{C}\setminus(\mathbb{R}\cup\mathcal{Z}\cup\bar{\mathcal{Z}})$;
    			\item $\bar{\partial}M^{(3)}(k)=M^{(3)}(k)W^{(3)}(k), k\in\mathbb{C}$;

    			\item $M^{(3)}(k)\sim I,\quad k\to\infty$.

    		\end{enumerate}
    	\end{RHP}
    where
    \begin{equation}\nonumber
    W^{(3)}=M_{RHP}^{(2)}(k)\bar{\partial}R^{(2)}M_{RHP}^{(2)}(k)^{-1},
    \end{equation}
    \begin{equation}\nonumber
    W^{(3)}(k)=\left\{\begin{array}{ccc}M_{RHP}^{(2)}(k)\left(\begin{array}{cc}0&0\\-\bar{\partial}R_{1}e^{2it\theta}&0\end{array}\right)M_{RHP}^{(2)}(k)^{-1},
    &k\in\Omega_{1},\\\\M_{RHP}^{(2)}(k)\left(\begin{array}{cc}0&-\bar{\partial}R_{3}e^{-2it\theta}\\0&0\end{array}\right)M_{RHP}^{(2)}(k)^{-1},&k\in\Omega_{3},\\\\
    M_{RHP}^{(2)}(k)\left(\begin{array}{cc}0&0\\\bar{\partial}R_{4}e^{2it\theta}&0\end{array}\right)M_{RHP}^{(2)}(k)^{-1},&k\in\Omega_{4},\\\\
    M_{RHP}^{(2)}(k)\left(\begin{array}{cc}0&\bar{\partial}R_{6}e^{-2it\theta}\\0&0\end{array}\right)M_{RHP}^{(2)}(k)^{-1},&k\in\Omega_{6},\\\\
    \left(\begin{array}{cc}0&0\\0&0\end{array}\right),&k\in\Omega_{2}\cup\Omega_{5}.\end{array}\right.
    \end{equation}

    The $\bar{\partial}$-problem of $M^{(3)}$ is equivalent to
    \begin{equation}\label{M3}
    M^{(3)}(k)=I-\frac1\pi\iint_\mathbb{C}\frac{M^{(3)}(s)W^{(3)}(s)}{s-k}dA(s),
\end{equation}
    where $dA(s)$ is the Lebesgue measure. Further, we write the equation \eqref{M3} in operator form
    \begin{equation}\nonumber
    (I-S)M^{(3)}(k)=I,
\end{equation}
    where $S$ is the Cauchy operator
    \begin{equation}\nonumber
    S[f](k)=-\frac{1}{\pi}\iint\limits_{\mathbb{C}}\frac{M^{(3)}(s)W^{(3)}(s)}{s-k}dA(s).
\end{equation}
    \begin{proposition}
    For large time t,
    \begin{equation}\nonumber
    ||S||_{L^\infty\to L^\infty}\leqslant ct^{-1/4},
\end{equation}
    which implies that the operator $(I-S)^{-1}$ is invertible and the solution of pure $\bar{\partial}$-problem exists and is unique.
    \end{proposition}
    \begin{proof} We only give the proof of $k\in\Omega_{1}$. For any $f\in L^{\infty}$,
    \begin{equation}\nonumber
    \begin{aligned}
|S(f)|& \leqslant\frac{1}{\pi}\iint_{\Omega_{11}}\frac{|f(s)M_{RHP}^{(2)}(k)\bar{\partial}R^{(2)}M_{RHP}^{(2)}(k)^{-1}|}{|s-k|}\mathrm{d}A(s) \\
&\leqslant\|f\|_{L^{\infty}}\frac{1}{\pi}\iint_{\Omega_{11}}\frac{|\bar{\partial}Re^{-2it\theta}|}{|s-k|}\mathrm{d}A(s)\lesssim I_{1}+I_{2}+I_{3}+I_{4},
\end{aligned}
\end{equation}
    where
    \begin{equation}\nonumber
    I_1=\iint\limits_{\Omega_{1}}\frac{|\bar{\partial}\chi_{\mathcal{Z}}e^{-2it\theta}|}{|s-k|}\mathrm{d}A(s),
    I_2=\iint\limits_{\Omega_{1}}\frac{|\gamma'(\text{Re}s)e^{-2it\theta}|}{|s-k|}\mathrm{d}A(s),\nonumber
\end{equation}
    \begin{equation}\nonumber
    I_3=\iint\limits_{\Omega_{1}}\frac{|s-k_0|^{-\frac{1}{2}}|e^{-2it\theta}|}{|s-k|}\mathrm{d}A(s),
    I_4=\iint\limits_{\Omega_{1}}\frac{|s-k_0|^{-1}|e^{-2it\theta}|}{|s-k|}\mathrm{d}A(s).\nonumber
\end{equation}
    Denote $s=k_{1}+u+i v$, $k=\alpha+i\eta $, and $Re(2it\theta)=-8tuv$ we have
    \begin{equation}\nonumber
    I_1\leqslant\intop_0^\infty\intop_v^\infty\frac{|\bar{\partial}\chi_{\mathcal{Z}}|}{|s-k|}e^{2\xi tv}
    \mathrm{d}u\mathrm{d}v\leqslant\intop_0^\infty e^{-8tv^2}\|\bar{\partial}\chi_{\mathcal{Z}}\|_{L^2}\|(s-k)^{-1}\|_{L^2}\mathrm{d}v,
\end{equation}
    where
    \begin{equation}\nonumber
    \begin{aligned}
\|(s-k)^{-1}\|_{L^{2}(v,\infty)}& \leqslant\left(\int_{-\infty}^\infty\frac1{(k_1+u-\alpha)^2+(v-\eta)^2}\mathrm{d}u\right)^{1/2} \\
&=\left(\frac1{|v-\eta|}\intop_{-\infty}^\infty\frac1{1+y^2}\mathrm{d}y\right)^{1/2}=\left(\frac\pi{|v-\eta|}\right)^{1/2},
\end{aligned}
\end{equation}
    with $y=\frac{k_1+u-\alpha}{v-\eta}$. Thus
    \begin{equation}\nonumber
    I_1\lesssim\intop_0^\infty\frac{e^{-8tv^2}}{\sqrt{|v-\eta|}}\mathrm{d}v\lesssim t^{-1/4}.\nonumber
\end{equation}
    The estimate of $I_2$ is just the same as $I_1$ because $\gamma^{\prime}(k)\in L^2(\mathbb{R})$, so we get $I_2\lesssim t^{-1/4}$.

    Finally, we deal with $I_3$. We first give the proof of the estimates as follows: for $2<p<\infty $ and $q$ satisfying $\frac1p+\frac1q=1$ we have
    \begin{equation}\nonumber
    \begin{aligned}
&\left\|\left|s-k\right|^{-1}\right\|_{L^q(\nu,+\infty)}=\left(\int_{\Omega_1}\left|s-k\right|^{-q}\mathrm{d}u\right)^{\frac1q} \\
&=\left(\int_{\Omega_1}\left((k_1+u-x)^2+(\nu-y)^2\right)^{-\frac q2}(\nu-y)d\frac{k_1+u-x}{\nu-y}\right)^{\frac1q} \\
&=\left(\int_{\Omega_1}\left(\left(\frac{k_1+u-x}{\nu-y}\right)^2+1\right)^{-\frac q2}\left(\nu-y\right)^{1-q}d\frac{k_1+u-x}{\nu-y}\right)^{\frac1q} \\
&=\left\{\int_{k_0}^{+\infty}\left[\left(\frac{k_1+u-x}{\nu-y}\right)^2+1\right]^{-q/2}d\left(\frac{k_1+u-x}{\nu-y}\right)\right\}^{1/q}|\nu-y|^{1/q-1} \\
&\lesssim|\nu-y|^{1/q-1},
\end{aligned}
\end{equation}
    and
    \begin{equation}\label{89}
    \begin{aligned}
\left\|\frac1{\sqrt{|s-k_0|}}\right\|_{L^p}
&=\left(\int_v^\infty\frac1{|u+iv|^{p/2}}du\right)^{1/p}=\left(\int_v^\infty\frac1{(u^2+v^2)^{p/4}}du\right)^{1/p} \\
&=v^{1/p\boldsymbol{-}1/2}\left(\int_1^\infty\frac1{(1+x^2)^{p/4}}dx\right)^{1/p}\leqslant cv^{1/p\boldsymbol{-}1/2}.
\end{aligned}
\end{equation}
    Therefore, by Cauchy-Schwarz inequality,
    \begin{equation}\nonumber
    \begin{aligned}
I_{3}& \leqslant c\int_{0}^{\infty}e^{-8tv^{2}}dv\int_{v}^{\infty}\frac{1}{|k-k_{0}|^{1/2}|s-k|}du \\
&\leqslant c\int_0^\infty e^{-8tv^2}\bigg|\bigg|\frac{1}{\sqrt{|s-k_0|}}\bigg|\bigg|_{L^p}\bigg|\bigg|\frac{1}{s-k_0}\bigg|\bigg|_{L^q}dv \\
&\begin{aligned}\leqslant c_3\int_0^\infty e^{-8tv^2}v^{1/p-1/2}|v-\beta|^{1/q-1}dv\lesssim t^{-1/4}.\end{aligned}
\end{aligned}
\end{equation}
    For $I_{4}$
    \begin{equation}\label{3157}
    I_{4}=\int_{a}^{1}\int_{v}^{1}\frac{1}{|s-k|}\frac{1}{\sqrt{u^{2}+v^{2}}}e^{-8tuv}dudv,
\end{equation}
    we find that $\frac{1}{\sqrt{u^{2}+v^{2}}}$ is square-integrable on $\left[a,1\right]$. We can then argue as $I_{1}$ to conclude that $I_4\lesssim{t}^{-1/4}$.
    \end{proof}

    Therefore, the solution of $\bar{\partial}$-problem is unique and obeys the formula \eqref{M3}. Then we have the following asymptotic estimation of $M^{(3)}(k)$.

    \begin{proposition}  As $k\to\infty $, The solution $M^{(3)}(k)$ of $\bar{\partial}-problem$ admits Laurent expansion:
    \begin{equation}\nonumber
    M^{(3)}(k)=I+\frac1kM_1^{(3)}+\mathcal{O}(k^{-2}),
\end{equation}
    where $M_1^{(3)}$ is a $k$-independent coefficient with
    \begin{equation}\nonumber
    M_1^{(3)}=\frac{1}{\pi}\iint_CM^{(3)}(s)W^{(3)}(s)dA(s).
\end{equation}
    $M_1^{(3)}$ satisfies
    \begin{equation}\label{813}
    |M_1^{(3)}|\lesssim t^{-3/4}.
\end{equation}
\end{proposition}
    \begin{proof}   $M_{RHP}^{(2)}$ is bounded beyond the poles on $\Omega_1'=\Omega_1\cap\operatorname{supp}(1-\chi_{\mathcal{Z}})$, therefore
    \begin{equation}\nonumber
    \begin{aligned}
|M_{1}^{(3)}|& \leqslant\frac{1}{\pi}\iint\limits_{\Omega_{1}}|M^{(3)}(s)M_{RHP}^{(2)}(s)\overline{\partial}R^{(2)}M_{RHP}^{(2)}(s)^{-1}|dA(s) \\
&\leqslant\frac{1}{\pi}||M^{(3)}||_{L^{\infty}(\Omega)}||M_{RHP}^{(2)}||_{L^{\infty}(\Omega^{\prime})}||(M_{RHP}^{(2)})^{-1}||_{L^{\infty}(\Omega^{\prime})}
\iint\limits_{\Omega}|\overline{\partial}Re^{2it\theta}|dA(s) \\
&\leqslant C\left(\underset{\Omega_1}{\operatorname*{\iint}}|\bar{\partial}\chi_{\mathcal{Z}}(s)|e^{-8tvu}dA\right.+\iint\limits_{\Omega_1}|\gamma'(u)|e^{-8tvu}dA(s)\\
&+\iint\limits_{\Omega_1}\frac{1}{|s-k_0|^{1/2}}e^{-8tvu}dA+\iint\limits_{\Omega_1}\frac{1}{|s-k_0|}e^{-8tvu}dA(s) \\
&\leqslant C(I_5+I_6+I_7+I_8).
\end{aligned}
\end{equation}
    Use the Cauchy-Schwarz inequality
    \begin{equation}\nonumber
    \begin{aligned}
|I_{5}|& \leqslant\int_{0}^{\infty}||\bar{\partial}\chi_{\mathcal{Z}}||_{L_{u}^{2}(v,\infty)}\left(\int_{v}^{\infty}e^{-8tuv}du\right)^{1/2}dv \\
&\leqslant ct^{-1/2}\int_0^\infty\frac{e^{-4tv^2}}{\sqrt{v}}\leqslant t^{-1/4}\int_0^\infty\frac{e^{-4w^2}}{\sqrt{w}}dw\leqslant c_5t^{-3/4}.
\end{aligned}
\end{equation}
    In a similar way to $I_{5}$, it can be shown that $I_6\leqslant c_5t^{-3/4}$, and $I_8\leqslant c_5t^{-3/4}$.

    Similar to the previous $I_{3}$ proof, for $2 < p < 4$, using the Holder inequality and \eqref{89},
    \begin{equation}\nonumber
    \int_{v}^\infty e^{-8tuv}|s-k_0|^{-1/2}du\leqslant cv^{1/p-1/2}\left(\int_v^\infty e^{-8qtuv}du\right)^{1/q},
\end{equation}
    where $1/p+1/q=1$, $2<p<4$, thus
    \begin{equation}
    \begin{aligned}\nonumber
    I_{7}&\leqslant\int_0^\infty v^{1/p-1/2}\left(\int_v^\infty e^{-4tqtuv}du\right)^{1/q}dv=\int_0^\infty v^{1/p-1/2}(qtv)^{-1/q}e^{-4tv^2}dv \\\\
   & \leqslant ct^{-1/q}\int_0^\infty v^{2/p-3/2}e^{-4tv^2}dv\leqslant ct^{-3/4}\int_0^\infty w^{2/p-3/2}e^{-4w^2}dw\leqslant ct^{-3/4},
    \end{aligned}
\end{equation}
\end{proof}

    \subsection{\it{Long-time asymptotics for the spin-1 GP equation for Region-1: $|k-k_{0}|\geq a$}}\label{s:3.5}
In this subsection, we construct the long-time asymptotics of equation \eqref{GP}. Inverting the sequence of transformations \eqref{M1}, \eqref{M2}, \eqref{MM3}
    and (6.4), we have
    \begin{equation}\nonumber
    M=M^{(1)}\Delta^{-1}(k)=M^{(2)}R^{(2)^{-1}}\Delta^{-1}(k)=M^{(3)}M_{RHP}^{(2)}R^{(2)^{-1}}\Delta^{-1}(k)
\end{equation}
    In particular, if we consider $k\to\infty $ in the vertical direction $k\in\Omega_{2},\Omega_{5}$, then we have $R^{(2)}=I,$, thus
    \begin{equation}\nonumber
    M=\left(I+\frac{M_1^{(3)}}k+\cdots\right)\left(I+\frac{M_1^{(out)}}k+\cdots\right)\left(I+\frac{{\Delta_1}^{-1}(k)}k+\cdots\right),
\end{equation}
    we can get
    \begin{equation}\nonumber
    M_1=M_1^{(out)}+M_1^{(3)}+{\Delta_1}^{-1}(k).
\end{equation}
    Then, it can be obtained by the reconstruction formula \eqref{q1} and the estimation \eqref{813}
    \begin{equation}\nonumber
    q(x,t)=2i(M_1^{(out)})_{UR}+\mathcal{O}(t^{-3/4}).
\end{equation}
    Since
    \begin{equation}\nonumber
    2i(M_1^{(out)})_{12}=q_{sol}(x,t;\sigma_d^{out}).
\end{equation}
    So
    \begin{equation}\label{360}
    q(x,t)=q_{sol}(x,t;\sigma_d^{out})+\mathcal{O}(t^{-3/4}).
\end{equation}

    \subsection{\it{Long-time asymptotics for the spin-1 GP equation for Region-2: $|k-k_{0}|<a$}}\label{s:3.6} 	
    For $|k|\in[0,a)$, $a<\rho/2$, $\int_{v}^{1}\frac{1}{u^{2}+v^{2}}du=\frac{1}{v}\arctan(\frac{1}{v})-\frac{1}{v}\arctan(1)$, which is not square-integrable
    on $\left[0,a\right]$, so $I_4$ defined in \eqref{3157} cannot be appropriately scaled in the $\bar{\partial}$ steepest descent method. We need to
    reconsider $I_4$ and we find that $I_4$ is caused by the error between $\delta_{j}$ and $\det\delta$ in the $\bar{\partial}$ steepest descent method.
    Here, we use the $\text{Deift-Zhou's}$ nonlinear steepest descent method to handle it. According to Corollary 3.2 in the \cite{boo-1}, as $t\to\infty,$
    the solution $q(x, t)$ for the Cauchy problem of the spin-1 GP equation \eqref{GP} is
    \begin{equation}\nonumber
    q(x,t)=\frac i{\sqrt{2t}}\left(M_1^0\right)_{UR}+\mathcal{O}\left(\frac{\log t}t\right),
\end{equation}
    where $M^0(k;x,t)$ is the solution of the following RH problem

    \begin{RHP}\label{RHP11}
    		Find a matrix valued function $M^0(k;x,t)$ admits:
    		\begin{enumerate}[(i)]
    			\item Analyticity:~$M^0(k;x,t)$ is analytic in $k\in\mathbb{C}\setminus(\Sigma_0)$, ;
    			\item Jump condition:
    			\begin{equation}M_+^0(k;x,t)=M_-^0(k;x,t)J^0(k;x,t),\quad k\in\Sigma_0,\end{equation}
    			\item Asymptotic behavior:
    \begin{align}
    				M^0(k;x,t)\to I_{4\times4},&k\to\infty,
    			\end{align}
    		\end{enumerate}
    	\end{RHP}

    $J^0=(b_-^0)^{-1}b_+^0=(I_{4\times4}-\omega_-^0)^{-1}(I_{4\times4}+\omega_+^0)$ and
    $\Sigma_0=\{k=ue^{\frac{\pi i}4}:u\in\mathbb{R}\}\cup\{k=ue^{-\frac{\pi i}4}:u\in\mathbb{R}\}$
    \begin{equation}\nonumber
    \begin{aligned}\omega^0=\omega_+^0=\begin{cases}\begin{pmatrix}0&-(\delta^0)^2k^{2i\nu}e^{-\frac12ik^2}\gamma^\dagger(k_0)\\
    0&0\end{pmatrix},&k\in\Sigma_0^1,\\\begin{pmatrix}0&(\delta^0)^2k^{2i\nu}e^{-\frac12ik^2}(I_{2\times2}+\gamma^\dagger(k_0)\gamma(k_0))^{-1}\gamma^\dagger(k_0)\\
    0&0\end{pmatrix},&k\in\Sigma_0^3,\end{cases}\end{aligned}
\end{equation}
    \begin{equation}\nonumber
    \omega^0=\omega_-^0=\begin{cases}\begin{pmatrix}0&0\\-(\delta^0)^{-2}k^{-2i\nu}e^{\frac12ik^2}\gamma(k_0)&0\end{pmatrix},&k\in\Sigma_0^2,\\
    \begin{pmatrix}0&0\\(\delta^0)^{-2}k^{-2i\nu}e^{\frac12ik^2}\gamma(k_0)(I_{2\times2}+\gamma^\dagger(k_0)\gamma(k_0))^{-1}&0\end{pmatrix},&k\in\Sigma_0^4.\end{cases}
\end{equation}
    This RH problem \ref{RHP11} can be transformed into a model problem, from which the explicit expression of $M_1^0$ can be obtained through the standard
    parabolic cylinder function. So according to the results of Theorem 1.1. of \cite{boo-1}, we have
    \begin{equation}\label{45}
    q(x,t)=t^{-1/2}\frac{\sqrt{\pi}(\delta^0)^2e^{-\frac{\pi\nu}{2}}e^{\frac{-3\pi i}{4}}}{\Gamma(-i\nu)\det-\gamma(k_0)}\begin{pmatrix}-\gamma_{22}(k_{0})&\gamma_{12}(k_{0})
    \\\gamma_{21}(k_{0})&-\gamma_{11}(k_{0})\end{pmatrix}+\mathcal{O}(\frac{\log t}{t})=t^{-1/2}g+\mathcal{O}(\frac{\log t}{t}),
\end{equation}
     $\Gamma(\cdot)$ is the Gamma function and the $\gamma_{ij}(k)$ is the $(i,j)$ entry of the matrix-valued function $\gamma(k)$ defined in \eqref{221}, and
    \begin{equation}\nonumber
    \begin{aligned}
\delta^{0}& =e^{2itk_{0}^{2}}(8t)^{-\frac{i\nu}{2}}e^{\chi(k_{0})},\quad k_{0}=-\frac{x}{4t}, \\
&\nu=-\frac1{2\pi}\log(1+\left|\gamma\left(k_0\right)\right|^2+\left|\det\gamma\left(k_0\right)\right|^2), \\
\chi\left(k_{0}\right)& =\frac1{2\pi i}\Big[\int_{k_0-1}^{k_0}\log\left(\frac{1+|\gamma(\xi)|^2+|\det\gamma(\xi)|^2}{1+|\gamma(k_0)|^2+|\det\gamma(k_0)|^2}\right)
\frac{\mathrm{d}\xi}{\xi-k_0} \\
&+\int_{-\infty}^{k_0-1}\log\Big(1+|\gamma(\xi)|^2+|\det\gamma(\xi)|^2\Big) \frac{\mathrm{d}\xi}{\xi-k_0}\Big].
\end{aligned}
\end{equation}

\section{Asymptotic analysis in the region $\xi=0$}\label{s:4}
    In this section, we will focus on the long-time asymptotic behavior of solution to the spin-1 GP equation \eqref{GP} in $\xi\to0$, as $t\to\infty$.
    We will soon prove in this region, the solution decay like $O(t^{-3/4})$, which is same as the leading term in the oscillating region.

    As for $t\to\infty,$ we obtain
    \begin{equation}\nonumber
     k_0={-\frac{x}{4t}}\to0,\quad\mathrm{as}\quad t\to+\infty,
\end{equation}
    therefore, we get the signature table of the phase function $Re(i\theta)$ in Figure. \ref{Figure 1.}.

    \begin{RHP}
    		Find a matrix valued function $M(k)$ admits:
    		\begin{enumerate}[(i)]
    			\item Analyticity:~$M(k;x,t)$ is analytic in $k\in\mathbb{C}\setminus(\mathbb{R}\cup\mathcal{Z}\cup\bar{\mathcal{Z}})$, $\mathcal{Z}=\{k_j\}_{j=1}^N$;
    			\item Jump condition:
    			\begin{equation}\nonumber
    M_+(k;x,t)=M_-(k;x,t)J(k;x,t),\quad k\in\mathbb{R},\end{equation}
    		\begin{equation}\nonumber
    J=\begin{pmatrix}I_{2\times2}+\gamma^\dagger(\bar{k})\gamma(k)&\gamma^\dagger(\bar{k})e^{-2it\theta}\\\gamma(k)e^{2it\theta}&I_{2\times2}\end{pmatrix},\end{equation}
              where $\theta(k)=\frac xtk+2k^2,$ $\gamma(k)=b(k)a^{-1}(k)$.
                \item Residue conditions: $M(k;x,t)$ has simple poles at each point in $\mathcal{Z}\cup\bar{\mathcal{Z}}$
                 \begin{equation}\nonumber
    \operatorname*{Res}_{k=k_j}M(k)=\lim_{k\to k_j}M(k)\begin{pmatrix}0&0\\f(k_j)e^{2it\theta(k_j)}&0\end{pmatrix};
\end{equation}
                \begin{equation}\nonumber
    \operatorname*{Res}_{k{=}\bar{k}}M(k)=\lim_{k\to \bar{k_j}}M(k)\begin{pmatrix}0&-f^\dagger(k_j)e^{-2it\theta(\bar{k_j})}\\0&0\end{pmatrix};
                 \end{equation}

    			\item Asymptotic behavior:

    			\begin{align}\nonumber
    				M(k;x,t)\rightarrow{I}~~as~~k\rightarrow\infty.
    			\end{align}
    		\end{enumerate}
    	\end{RHP}

    \begin{equation}\nonumber
   M^{(1)}(k;x,t)=M(k;x,t)\Delta^{-1}(k).\end{equation}
    \begin{RHP}
    		Find a matrix valued function $M^{(1)}(k;x,t)$ admits:
    		\begin{enumerate}[(i)]
    			\item Analyticity:~$M^{(1)}(k;x,t)$ is analytic in $k\in\mathbb{C}\setminus(\mathbb{R}\cup\mathcal{Z}\cup\bar{\mathcal{Z}})$;
    			\item Jump condition:
    			\begin{equation}\nonumber
    M_+^{(1)}(k;x,t)=M_-^{(1)}(k;x,t)J^{(1)}(k;x,t),\quad k\in\mathbb{R},\end{equation}
    		\begin{equation}\nonumber
    \left.J^{(1)}(k;x,t)=
    \begin{cases}\nonumber\begin{pmatrix}I_{2\times2}&T_{1-}(k)\gamma^\dagger(\bar{k})T_{2-}(k)e^{-2it\theta}\\0&I_{2\times2}\end{pmatrix}\begin{pmatrix}I_{2\times2}&0\\
    T_{2+}^{-1}(k)\gamma(k)T_{1+}^{-1}(k)e^{2it\theta}&I_{2\times2}\end{pmatrix},\quad k\in(-\infty,0),\\\begin{pmatrix}I_{2\times2}&0\\
    T_{2-}^{-1}(k)\rho^{\dagger}(\bar{k})T_{1-}^{-1}(k)e^{2it\theta}&I_{2\times2}\end{pmatrix}\begin{pmatrix}I_{2\times2}&T_{1+}(k)\rho(k)T_{2+}(k)e^{-2it\theta}\\
    0&I_{2\times2}\end{pmatrix},\quad k\in(0,+\infty);&\end{cases}\right.\end{equation}
    where
    \begin{equation}\nonumber\rho(k)=\left(I_{2\times2}+\gamma^\dagger(\bar{k})\gamma(k)\right)^{-1}\gamma^\dagger(\bar{k}),\nonumber\end{equation}

    			\item Asymptotic behavior:
    			\begin{align}\nonumber
    				M^{(1)}(k;x,t)\rightarrow{I_{4\times4}}~~as~~k\rightarrow\infty.
    			\end{align}
                \item Residue conditions: $M^{(1)}(k;x,t)$ has simple poles at each point in $\mathcal{Z}\cup\bar{\mathcal{Z}}$

                For $j\in\Delta_{0}^-$
                \begin{equation}\nonumber
    \operatorname*{Res}_{k=k_j}M^{(1)}(k)=\lim_{k\to k_j}M^{(1)}(k)\begin{pmatrix}0&[({T_1}^{-1})'(k_j)]^{-1}f^{-1}(k_j)[({T_2}^{-1})'(k_j)]^{-1}e^{-2it\theta(k_j)}\\
    0&0\end{pmatrix};
                 \end{equation}
               \begin{equation}\nonumber
    \operatorname*{Res}_{k{=}\bar{k}}M^{(1)}(k)=\lim_{k\to \bar{k_j}}M^{(1)}(k)\begin{pmatrix}0&0\\
    -[{T_2}'(\bar{k_j})]^{-1}(f^\dagger(k_j))^{-1}[{T_1}'(\bar{k_j})]^{-1}e^{2it\theta(\bar{k_j})}&0\end{pmatrix};
\end{equation}

                For $j\in\Delta_{0}^+$
                \begin{equation}\nonumber
    \operatorname*{Res}_{k=k_j}M^{(1)}(k)=\lim_{k\to k_j}M^{(1)}(k)\begin{pmatrix}0&0\\{T_2}^{-1}(k_j)f(k_j){T_1}^{-1}(k_j)e^{2it\theta(k_j)}&0\end{pmatrix};
\end{equation}
                \begin{equation}\nonumber
    \operatorname*{Res}_{k{=}\bar{k}}M^{(1)}(k)=\lim_{k\to \bar{k_j}}M^{(1)}(k)\begin{pmatrix}0&-T_1(\bar{k_j})f^\dagger(k_j)T_2(\bar{k_j})e^{-2it\theta(\bar{k_j})}
    \\0&0\end{pmatrix}.
                 \end{equation}

    		\end{enumerate}
    	\end{RHP}

    \begin{equation}\nonumber
     M^{(2)}(k)=M^{(1)}(k)R^{(2)}(k).\end{equation}
     \begin{RHP}
    		Find a matrix valued function $$M^{(2)}(k)=M^{(2)}(k;x,t)$$ admits:
    		\begin{enumerate}[(i)]
    			\item $M^{(2)}(k;x,t)$ is continuous in $k\in\mathbb{C}\setminus(\mathbb{R}\cup\mathcal{Z}\cup\bar{\mathcal{Z}})$;
    			\item Jump condition:
    			\begin{equation}\nonumber
    M_+^{(2)}(k)=M_-^{(2)}(k)V^{(2)}(k),\quad k\in\Sigma^{(2)},\end{equation}
              where the jump matrix
              \begin{equation}\nonumber
             V^{(2)}(k)=(R_-^{(2)})^{-1}J^{(1)}R_+^{(2)}=I+(1-\chi_Z(k))\delta V^{(2)},
            \end{equation}
    		\begin{equation}\nonumber
    \delta V^{(2)}(k)=\begin{cases}\begin{pmatrix}0&0\\T_0(0)^{-2}\gamma(0)(k)^{-2i\nu(0)}e^{2it\theta}&0\end{pmatrix},\quad k\in\Sigma_1,\\\\
    \begin{pmatrix}0&T_0(0)^{2}\rho(0)(k)^{2i\nu(0)}e^{-2it\theta}\\0&0\end{pmatrix},\quad k\in\Sigma_2,\\\\
    \begin{pmatrix}0&0\\T_0(0)^{-2}\rho^{\dagger}(0)(k)^{-2i\nu(0)}e^{2it\theta}&0\end{pmatrix},\quad k\in\Sigma_3,\\\\
    \begin{pmatrix}0&T_0(0)^{2}\gamma^\dagger(0)(k)^{2i\nu(0)}e^{-2it\theta}\\0&0\end{pmatrix},\quad k\in\Sigma_4,&\end{cases}
\end{equation}
    where
    \begin{equation}\rho(k)=\left(I_{2\times2}+\gamma^\dagger(\bar{k})\gamma(k)\right)^{-1}\gamma^\dagger(\bar{k}),\nonumber\end{equation}

    			\item Asymptotic behavior:
    			\begin{align}\nonumber
    				M^{(2)}(k;x,t)\rightarrow{I_{4\times4}}~~as~~k\rightarrow\infty,
    			\end{align}
                \item $\bar{\partial}-Derivative$: for $\mathbb{C}\setminus(\Sigma^{(2)}\cup\mathcal{Z}\cup\bar{\mathcal{Z}})$ we have
                \begin{equation}\nonumber
                \bar{\partial}M^{(2)}(k)=M^{(2)}(k)\bar{\partial}R^{(2)}(k),
                \end{equation}
                where
                \begin{equation}\nonumber\bar{\partial}R^{(2)}(k)
                =\left\{\begin{array}{cc}\left(\begin{array}{cc}0&0\\-\bar{\partial}R_1e^{2it\theta}&0\end{array}\right),&k\in\Omega_1,\\
                                                        \left(\begin{array}{cc}0&-\bar{\partial}R_{3}e^{-2it\theta}\\0&0\end{array}\right),&k\in\Omega_{3},\\
                                                        \left(\begin{array}{cc}0&0\\\bar{\partial}R_{4}e^{2it\theta}&0\end{array}\right),&k\in\Omega_{4},\\
                                                        \left(\begin{array}{cc}0&\bar{\partial}R_6e^{-2it\theta}\\0&0\end{array}\right),&k\in\Omega_6,\\
                                                        \left(\begin{array}{cc}0&0\\0&0\end{array}\right),&k\in\Omega_2\cup\Omega_5.
                                                        \end{array}\right.\end{equation}

                \item Residue conditions: $M^{(2)}(k;x,t)$ has simple poles at each point in $\mathcal{Z}\cup\bar{\mathcal{Z}}$

                For $j\in\Delta_{0}^-$
                \begin{equation}\nonumber
    \operatorname*{Res}_{k=k_j}M^{(2)}(k)=\lim_{k\to k_j}M^{(2)}(k)
    \begin{pmatrix}0&[({T_1}^{-1})'(k_j)]^{-1}f^{-1}(k_j)[({T_2}^{-1})'(k_j)]^{-1}e^{-2it\theta(k_j)}\\0&0\end{pmatrix};
                 \end{equation}
               \begin{equation}\nonumber
    \operatorname*{Res}_{k{=}\bar{k}}M^{(2)}(k)=\lim_{k\to \bar{k_j}}M^{(2)}(k)
    \begin{pmatrix}0&0\\-[{T_2}'(\bar{k_j})]^{-1}(f^\dagger(k_j))^{-1}[{T_1}'(\bar{k_j})]^{-1}e^{2it\theta(\bar{k_j})}&0\end{pmatrix};
\end{equation}

                For $j\in\Delta_{0}^+$
                \begin{equation}\nonumber
    \operatorname*{Res}_{k=k_j}M^{(2)}(k)=\lim_{k\to k_j}M^{(2)}(k)
    \begin{pmatrix}0&0\\{T_2}^{-1}(k_j)f(k_j){T_1}^{-1}(k_j)e^{2it\theta(k_j)}&0\end{pmatrix};
\end{equation}
                \begin{equation}\nonumber
    \operatorname*{Res}_{k{=}\bar{k}}M^{(2)}(k)=\lim_{k\to \bar{k_j}}M^{(2)}(k)
    \begin{pmatrix}0&-T_1(\bar{k_j})f^\dagger(k_j)T_2(\bar{k_j})e^{-2it\theta(\bar{k_j})}\\0&0\end{pmatrix};
                 \end{equation}

    		\end{enumerate}
    	\end{RHP}
     \begin{equation}\nonumber
    M^{(2)}(k;x,t)=\begin{cases}\overline{\partial}R^{(2)}=0\to M_{RHP}^{(2)},\\
    \overline{\partial}R^{(2)}\neq0\to M^{(3)}=M^{(2)}M_{RHP}^{(2)-1},\end{cases}
    \end{equation}
    \begin{RHP}\nonumber
    		Find a matrix valued function $M_{RHP}^{(2)}$ admits:
    		\begin{enumerate}[(i)]
    			\item ~$M_{RHP}^{(2)}(k;x,t)$ is continuous in $k\in\mathbb{C}\setminus(\mathbb{R}\cup\mathcal{Z}\cup\bar{\mathcal{Z}})$;
    			\item Jump condition:
    			\begin{equation}\nonumber
    M_{+RHP}^{(2)}(k)=M_{-RHP}^{(2)}(k)V^{(2)}(k),\quad k\in\Sigma^{(2)},\end{equation}

    			\item Asymptotic behavior:
    			\begin{align}\nonumber
    				M_{RHP}^{(2)}(k;x,t)\rightarrow{I_{4\times4}}~~as~~k\rightarrow\infty,
    			\end{align}
                \item $\overline{\partial}R^{(2)}=0$

                \item Residue conditions: With the same residue conditions as $M^{(2)}$,

    		\end{enumerate}
    	\end{RHP}

     \begin{RHP}\nonumber
    		Find a matrix valued function $M^{(3)}(k)=M^{(3)}(k;x,t)$ admits:
    		\begin{enumerate}[(i)]
    			\item ~$M^{(3)}(k;x,t))$ is continuous in $k\in\mathbb{C}\setminus(\mathbb{R}\cup\mathcal{Z}\cup\bar{\mathcal{Z}})$;
    			\item $\bar{\partial}M^{(3)}(k)=M^{(3)}(k)W^{(3)}(k), k\in\mathbb{C},$

    			\item $M^{(3)}(k)\sim I,\quad k\to\infty,$

    		\end{enumerate}
    	\end{RHP}

    similar to the previous analysis, it can be concluded that the phase function of the spin-1 GP equation only has one steady-state phase point,
    so as $t\to+\infty$, there will be no collision of steady-state phase points. After our research,
    \begin{equation}\nonumber
    M=M^{(1)}\Delta^{-1}(k)=M^{(2)}R^{(2)^{-1}}\Delta^{-1}(k)=M^{(3)}M_{RHP}^{(2)}R^{(2)^{-1}}\Delta^{-1}(k)
\end{equation}
    we found that this situation is included in the first scenario we considered, the same leading term and the errors are all $O(t^{-3/4})$.
    \begin{equation}\nonumber
    q(x,t)=q_{sol}(x,t;\sigma_d^{out})+t^{-1/2}g+\mathcal{O}(t^{-3/4}).
\end{equation}
    \begin{equation}\nonumber
    g=\frac{\sqrt{\pi}(\delta^0)^2e^{-\frac{\pi\nu}{2}}e^{\frac{-3\pi i}{4}}}{\Gamma(-i\nu)\det(-\gamma(0))}
    \begin{pmatrix}-\gamma_{22}(0)&\gamma_{12}(0)\\\gamma_{21}(0)&-\gamma_{11}(0)\end{pmatrix}
\end{equation}

\section{\it{Long-time asymptotic behaviors for the spin-1 GP equation}}\label{s.5}

The main purpose of this section is to give the long-time asymptotic behaviors of the spin-1 GP equation .
Inverting the sequence of transformations \eqref{M1}, \eqref{M2} and \eqref{MM3}, we have
    \begin{equation}\nonumber
    M=M^{(1)}\Delta^{-1}(k)=M^{(2)}R^{(2)^{-1}}\Delta^{-1}(k)=M^{(3)}M_{RHP}^{(2)}R^{(2)^{-1}}\Delta^{-1}(k)
\end{equation}
    In particular, if we consider $k\to\infty $ in the vertical direction $k\in\Omega_{2},\Omega_{5}$, then we have $R^{(2)}=I,$, thus
    \begin{equation}\nonumber
    M=\left(I+\frac{M_1^{(3)}}k+\cdots\right)\left(I+\frac{M_1^{(out)}}k+\cdots\right)\left(I+\frac{{\Delta_1}^{-1}(k)}k+\cdots\right),
\end{equation}
    we can get
    \begin{equation}\nonumber
    M_1=M_1^{(out)}+M_1^{(3)}+{\Delta_1}^{-1}(k).
\end{equation}
    Then, it can be obtained by the reconstruction formula \eqref{q1} and the estimation \eqref{813}
    \begin{equation}\nonumber
    q(x,t)=2i(M_1^{(out)})_{UR}+\mathcal{O}(t^{-3/4}).
\end{equation}
    Since
    \begin{equation}\nonumber
    2i(M_1^{(out)})_{12}=q_{sol}(x,t;\sigma_d^{out}).
\end{equation}
    So
    \begin{equation}\label{360}
    q(x,t)=q_{sol}(x,t;\sigma_d^{out})+\mathcal{O}(t^{-3/4}).
\end{equation}

    Combining \eqref{360} and \eqref{45}, we can get the long-time asymptotic of the spin-1 GP equation under the cases of coexistence of discrete and continuous spectrum.

The principal results of the this work are now stated as follows.
     \begin{theorem}\label{theorem1}
     Let $q_0(x,t), q_1(x,t) , q_{-1}(x,t)$ be the solution of \eqref{GP} corresponding to initial data $q_0^0(x), q_1^0(x) , q_{-1}^0(x) \in \mathscr{S}(\mathbb{R})$ which
      satisfies Assumption \ref{assumption25}. Let $\left\{\gamma(k),\{(k_j,f(k_{j})\}_{j=1}^N\right\}$ denote the scattering data generated from
      $q_0^0(x)$, $q_1^0(x)$, $q_{-1}^0(x)$. Fix $x_1,x_2,v_1,v_2\in\mathbb{R}$ with $x_{1}\leq x_{2}$ and $v_2>v_1>0$. Let $I=[-v_2/4,-v_1/4]$  and $\xi=x/t$. The soliton
      solution of \eqref{GP} denote by $q_{{sol}}(x,t;\sigma_d^\pm(\mathcal{I}))$ with modulating reflectionless scattering data $\sigma_d^\pm(\mathcal{I})$ define
      in \eqref{12}. Then, as $|t|\to\infty$ in the cone $S(x_1,x_2,v_1,v_2)$ defined in \eqref{SS}, the solution of the Cauchy problem for the spin-1 GP equation \eqref{GP}
      satisfies the following asymptotic formulae:
    \begin{equation}
    q(x,t)=\begin{pmatrix}q_1(x,t)&q_0(x,t)\\q_0(x,t)&q_{-1}(x,t)\end{pmatrix}=q_{sol}(x,t;\widehat{\sigma}_d(\mathcal{I}))+t^{-1/2}g+\mathcal{O}(t^{-3/4}),
\end{equation}
    where $q_{sol}(x,t;\widehat{\sigma}_d(\mathcal{I}))$ defined in \eqref{775}, and
    \begin{equation}\nonumber
    g=\frac{\sqrt{\pi}(\delta^0)^2e^{-\frac{\pi\nu}{2}}e^{\frac{-3\pi i}{4}}}{\Gamma(-i\nu)\det(-\gamma(k_0))}\begin{pmatrix}-\gamma_{22}(k_{0})&\gamma_{12}(k_{0})\\
    \gamma_{21}(k_{0})&-\gamma_{11}(k_{0})\end{pmatrix},
\end{equation}
where $\gamma(k)$ defined in \eqref{221}, $\Gamma(\cdot)$ is the Gamma function and
\begin{equation}\nonumber
    \begin{aligned}
\delta^{0}& =e^{2itk_{0}^{2}}(8t)^{-\frac{i\nu}{2}}e^{\chi(k_{0})},\quad k_{0}=-\frac{x}{4t}, \\
&\nu=-\frac1{2\pi}\log(1+|\gamma(k_{0})|^{2}+|\det\gamma(k_{0})|^{2}), \\
\chi\left(k_{0}\right)& =\frac{1}{2\pi i}\Big[\int_{k_{0}-1}^{k_{0}}\log\left(\frac{1+|\gamma(\xi)|^{2}+|\det\gamma(\xi)|^{2}}{1+
|\gamma(k_{0})|^{2}+|\det\gamma(k_{0})|^{2}}\right)\frac{\mathrm{d}\xi}{\xi-k_{0}} \\
&+\int_{-\infty}^{k_{0}-1}\log\Big(1+|\gamma(\xi)|^{2}+|\det\gamma(\xi)|^{2}\Big) \frac{\mathrm{d}\xi}{\xi-k_{0}}\Big].
\end{aligned}
\end{equation}
\begin{equation}\label{12}
    \sigma_d^\pm(\mathcal{I})=\{(k_j,f^\pm(\mathcal{I})):k_j\in\mathcal{Z}(\mathcal{I})\},
\end{equation}
\begin{equation}\label{SS}
    S(x_1,x_2,v_1,v_2):=\left\{(x,t)\in\mathbb{R}^2:x=x_0+vt\text{ with }x_0\in[x_1,x_2], v\in[v_1,v_2]\right\}.
\end{equation}
    \end{theorem}

If the reference cone $S$ does not correspond to any of the soliton speeds, then $|\xi-\operatorname{Re}k_j|\geq c>0$ for all $(x,t)\in S$ and $j=1,\ldots,N$, $q_{sol}$ is
    each identically zero. So we have

    \begin{theorem}\label{theorem2}
      In the no-soliton case, $v_1,v_2$ are chosen in Theorem \ref{theorem1} such that $N(I)=0$, $M^{\Delta_{\mathcal{I}}}(k|\widehat{\sigma}_d(\mathcal{I}))\equiv I$,
      and $q_{sol}(x,t;\widehat{\sigma}_d(\mathcal{I}))\equiv0$, the asymptotic behavior of the solution reduces to
      \begin{equation}
    q(x,t)=\begin{pmatrix}q_1(x,t)&q_0(x,t)\\q_0(x,t)&q_{-1}(x,t)\end{pmatrix}=t^{-1/2}\frac{\sqrt{\pi}(\delta^0)^2e^{-\frac{\pi\nu}{2}}e^{\frac{-3\pi i}{4}}}
    {\Gamma(-i\nu)\det(-\gamma(k_0))}\begin{pmatrix}-\gamma_{22}(k_{0})&\gamma_{12}(k_{0})\\\gamma_{21}(k_{0})&-\gamma_{11}(k_{0})\end{pmatrix}+\mathcal{O}(t^{-3/4}),
    \end{equation}
    where $M^{\Delta_{\mathcal{I}}}(k|\widehat{\sigma}_d(\mathcal{I}))$ solves RH problem \ref{RHP1} with scattering data $\widehat{\sigma}_d(\mathcal{I})$.
    \end{theorem}

\section*{Acknowledgments}\label{s:10} 	
    This work was supported by the National Natural Science Foundation of China under Grant No.  12371255, Xuzhou Basic Research Program Project under
    Grant No. KC23048,  the 333 Project in Jiangsu Province, the Fundamental Research Funds for the Central Universities of CUMT under Grant No. 2024ZDPYJQ1003.\\

  \textbf{Data availibility}: The data which supports the findings of this study is available within the article.\\

  \textbf{Declarations}\\

\textbf{Conflict of interest}: The authors declare no conflict of interest.

\appendix{}

\section{Model problem}\label{appendix}

    In this section, we focus on $M_1^0$ and $M^0(k;x,t)$. The $M^0(k;x,t)$ is the solution of the RH problem
    \begin{equation}\label{373}
    \begin{cases}M_+^0(k;x,t)=M_-^0(k;x,t)J^0(k;x,t),&k\in\Sigma_0,\\M^0(k;x,t)\to I_{4\times4},&k\to\infty,\end{cases}
\end{equation}
    $J^0=(b_-^0)^{-1}b_+^0=(I_{4\times4}-\omega_-^0)^{-1}(I_{4\times4}+\omega_+^0)$. In particular, we have
    \begin{equation}\nonumber
    M^0(k)=I_{4\times4}+\frac{M_1^0}k+\mathcal{O}(k^{-2}),\quad k\to\infty,
\end{equation}
    \begin{equation}\label{103}
    q(x,t)=\frac i{\sqrt{2t}}\left(M_1^0\right)_{UR}+\mathcal{O}\left(\frac{\log t}t\right).
\end{equation}

    The RH problem can be transformed into a model problem, and an explicit expression for $M_1^0$ can be obtained through the standard parabolic
    cylinder function. For this purpose, we introduce
    \begin{equation}\nonumber
    \Psi(k)=H(k)k^{i\nu\sigma_4}e^{-\frac14ik^2\sigma_4},\quad H(k)=(\delta^0)^{-\sigma_4}M^0(k)(\delta^0)^{\sigma_4},
\end{equation}
where $\delta^0=e^{2itk_0^2}(8t)^{-\frac{iv}{2}}e^{\chi(k_0)}$.

    It is easy to see from \eqref{373} that
    \begin{equation}\label{104}
    \Psi_+(k)=\Psi_-(k)v(k_0),\quad v=e^{\frac14ik^2\sigma_4}k^{-i\nu\sigma_4}(\delta^0)^{-\sigma_4}J^0(k)(\delta^0)^{\sigma_4}k^{i\nu\sigma_4}e^{-\frac14ik^2\sigma_4}.
\end{equation}
    For $k\in\Sigma_0^1,\Sigma_0^2,\Sigma_0^3,\Sigma_0^4$, the jump matrix is independent of k, so
    \begin{equation}\label{105}
    \frac{\mathrm{d}\Psi_+(k)}{\mathrm{d}k}=\frac{\mathrm{d}\Psi_-(k)}{\mathrm{d}k}v(k_0).
\end{equation}
    By \eqref{104} and \eqref{105}, we obtain
    \begin{equation}\nonumber
    \frac{\mathrm{d}\Psi_+(k)}{\mathrm{d}k}+\frac12ik\sigma_4\Psi_+(k)=\left(\frac{\mathrm{d}\Psi_-(k)}{\mathrm{d}k}+\frac12ik\sigma_4\Psi_-(k)\right)v(k_0).
\end{equation}
    Then $(\mathrm{d}\Psi/\mathrm{d}k+\frac12ik\sigma_4\Psi)\Psi^{-1}$ has no jump discontinuity along each of the four rays. We have
    \begin{equation}\nonumber
    \begin{aligned}
\left(\frac{\mathrm{d}\Psi(k)}{\mathrm{d}k}+\frac12ik\sigma_4\Psi(k)\right)\Psi^{-1}(k)& =\frac{\mathrm{d}H(k)}{\mathrm{d}k}H^{-1}(k)-\frac{ik}{2}H(k)\sigma_4H^{-1}(k) \\
&+\frac{i\nu}kH(k)\sigma_4H^{-1}(k)+\frac12ik\sigma_4 \\
&=\mathcal{O}(k^{-1})+\frac i2(\delta^0)^{-\sigma_4}[\sigma_4,M_1^0](\delta^0)^{\sigma_4}.
\end{aligned}
\end{equation}
    By the $\text{Liouville's}$ Theorem we can get
    \begin{equation}\label{109}
    \frac{\mathrm{d}\Psi(k)}{\mathrm{d}k}+\frac12ik\sigma_4\Psi(k)=\beta\Psi(k),
\end{equation}
    where
    \begin{equation}\nonumber
    \beta=\frac i2(\delta^0)^{-\sigma_4}[\sigma_4,M_1^0](\delta^0)^{\sigma_4}=\begin{pmatrix}0&\beta_{12}\\\beta_{21}&0\end{pmatrix}.
\end{equation}
    Moreover,
    \begin{equation}\label{1011}
    (M_1^0)_{12}=-i(\delta^0)^2\beta_{12}.
\end{equation}
    The RH problem \ref{373} shows that
    \begin{equation}\nonumber
    \sigma_4(M^0(\bar{k}))^\dagger\sigma_4=(M^0(k))^{-1},
\end{equation}
    which implies that $\beta_{12}=\beta_{21}^\dagger $. Set
    \begin{equation}\nonumber
    \Psi(k)=\begin{pmatrix}\Psi_{11}(k)&\Psi_{12}(k)\\\Psi_{21}(k)&\Psi_{22}(k)\end{pmatrix},
\end{equation}
    $\Psi_{ij}(k)(i,j=1,2)$ are all $2\times2$ matrices. From \eqref{109} and its differential we obtain
    \begin{equation}\label{1013}
    \frac{\mathrm{d}^2\Psi_{11}(k)}{\mathrm{d}k^2}+\left[(\frac12i+\frac14k^2)I_{2\times2}-\beta_{12}\beta_{21}\right]\Psi_{11}(k)=0,
\end{equation}
    \begin{equation}\nonumber
    \beta_{12}\Psi_{21}(k)=\frac{\mathrm{d}\Psi_{11}(k)}{\mathrm{d}k}+\frac12ik\Psi_{11}(k),
\end{equation}
    \begin{equation}\nonumber
    \frac{\mathrm{d}^2\beta_{12}\Psi_{22}(k)}{\mathrm{d}k^2}+\left[(-\frac12i+\frac14k^2)I_{2\times2}-\beta_{12}\beta_{21}\right]\beta_{12}\Psi_{22}(k)=0,
\end{equation}
    \begin{equation}\nonumber
    \Psi_{12}(k)=(\beta_{12}\beta_{21})^{-1}\left(\frac{\mathrm{d}\beta_{12}\Psi_{22}(k)}{\mathrm{d}k}-\frac12ik\beta_{12}\Psi_{22}(k)\right).
\end{equation}
    For the convenience, we assume that the $2\times2$ matrices $\beta_{12}$ and $\beta_{12}\beta_{21}$ have the forms
    \begin{equation}\nonumber
    \beta_{12}=\begin{pmatrix}A&B\\C&D\end{pmatrix},\quad\beta_{12}\beta_{21}=\begin{pmatrix}\tilde{A}&\tilde{B}\\\tilde{C}&\tilde{D}\end{pmatrix}.
\end{equation}
    Set $\Psi_{11} = (\Psi_{11}^{(ij)})_{2\times2}.$ We consider that the $(1,1)$ and $(2,1)$ terms of equation \eqref{1013}
    \begin{equation}\label{1018}
    \frac{\mathrm{d}^2\Psi_{11}^{(11)}(k)}{\mathrm{d}k^2}+(\frac12i+\frac14k^2)\Psi_{11}^{(11)}(k)-\tilde{A}\Psi_{11}^{(11)}(k)-\tilde{B}\Psi_{11}^{(21)}(k)=0,
\end{equation}
    \begin{equation}\nonumber
    \frac{\mathrm{d}^2\Psi_{11}^{(21)}(k)}{\mathrm{d}k^2}+(\frac12i+\frac14k^2)\Psi_{11}^{(21)}(k)-\tilde{C}\Psi_{11}^{(11)}(k)-\tilde{D}\Psi_{11}^{(21)}(k)=0.
\end{equation}
    If $s$ satisfies $\tilde{B}\tilde{C}=(s-\tilde D)(s-\tilde{A})$, then \eqref{1018} becomes
    \begin{equation}\nonumber
    \frac{\mathrm{d}^2}{\mathrm{d}k^2}[\tilde{C}\Psi_{11}^{(11)}(k)+(s-\tilde{A})\Psi_{11}^{(21)}(k)]+(\frac12i+\frac14k^2-s)[\tilde{C}\Psi_{11}^{(11)}(k)+(s-\tilde{A})\Psi_{11}^{(21)}(k)]=0.
\end{equation}
    Obviously, we can transform the above equation into the $\text{Weber's}$ equation through a simple variable transformation. As is well known,
    the standard parabolic cylinder functions $D_{a}(\zeta)$ and $D_{a}(-\zeta)$ constitute the fundamental solution set of the $\text{Weber's}$ equation
    \begin{equation}\nonumber
    \frac{\mathrm{d}^2g(\zeta)}{\mathrm{d}\zeta^2}+\left(a+\frac12-\frac{\zeta^2}4\right)g(\zeta)=0,\nonumber
\end{equation}
    whose general solution representation
    \begin{equation}\nonumber
    g(\zeta)=C_1D_a(\zeta)+C_2D_a(-\zeta),\nonumber
\end{equation}
    where $C_{1}$ and $C_{2}$ are two arbitrary constants. Set $a=is$,
    \begin{equation}\label{1021}
    \tilde{C}\Psi_{11}^{(11)}(k)+(s-\tilde{A})\Psi_{11}^{(21)}(k)=c_1D_a(e^{\frac{\pi i}{4}}k)+c_2D_a(e^{-\frac{3\pi i}{4}}k),
\end{equation}
    where $c_{1}$ and $c_{2}$ are constants. First, the solution $c_1D_a(e^\frac{\pi i}4k)+c_2D_a(e^\frac{3\pi i}4k)$ is nontrivial, otherwise the
    large $k$ expansion of $\Psi(k)$ is false. Besides, notice that as $k\to\infty $,
    \begin{equation}\label{1022}
    \Psi_{11}(k)\to k^{i\nu}e^{-\frac14ik^2}I_{2\times2}.
\end{equation}
    where the parabolic-cylinder function $D_{a}(\zeta)$ has a asymptotic expansion
    \begin{equation}\label{1023}
    \begin{aligned}&D_{a}(\zeta)=\begin{cases}\zeta^ae^{-\frac{\zeta^2}{4}}(1+O(\zeta^{-2})),&|\arg\zeta|<\frac{3\pi}{4},\\
    \zeta^ae^{-\frac{\zeta^2}{4}}(1+O(\zeta^{-2}))-\frac{\sqrt{2\pi}}{\Gamma(-a)}e^{a\pi i+\frac{\zeta^2}{4}}\zeta^{-a-1}(1+O(\zeta^{-2})),
    &\frac{\pi}{4}<\arg\zeta<\frac{5\pi}{4},\\\zeta^ae^{-\frac{\zeta^2}{4}}(1+O(\zeta^{-2}))-
    \frac{\sqrt{2\pi}}{\Gamma(-a)}e^{-a\pi i+\frac{\zeta^2}{4}}\zeta^{-a-1}(1+O(\zeta^{-2})),&-\frac{5\pi}{4}<\arg\zeta<-\frac{\pi}{4},\end{cases}\end{aligned}
\end{equation}
    as $\zeta\to\infty $,  where $\Gamma(\cdot)$ is the Gamma function. Calculate by substituting \eqref{1022} and \eqref{1023} into \label{1021},
    we get that $c_1=\tilde{C}k^{i\nu-a}e^{\frac{-a\pi i}4}$ and $c_{2}=0$. Meanwhile, $\Psi_{11}^{(11)}(k)$ and $\Psi_{11}^{(21)}(k)$ satisfy
    asymptotic expansion \eqref{1022} and are therefore not linearly correlated. From equation \label{1021}, it can be seen that the coefficient
    of $\Psi_{11}^{(21)}(k)$ is unique, that is, $s$ is unique. Based on the definition of $s$, we obtain $\tilde{B}=\tilde{C}=0$. Therefore, we
    assume that $\beta_{12}\beta_{21}=\mathrm{diag}(d_1,d_2)$ and \eqref{1013} becomes
    \begin{equation}\nonumber
    \frac{\mathrm{d}^2}{\mathrm{d}k^2}\begin{pmatrix}\Psi_{11}^{(11)}&\Psi_{11}^{(12)}\\\Psi_{11}^{(21)}&\Psi_{11}^{(22)}\end{pmatrix}+(\frac12i+\frac14k^2)
    \begin{pmatrix}\Psi_{11}^{(11)}&\Psi_{11}^{(12)}\\\Psi_{11}^{(21)}&\Psi_{11}^{(22)}\end{pmatrix}
    -\begin{pmatrix}d_1\Psi_{11}^{(11)}&d_1\Psi_{11}^{(12)}\\d_2\Psi_{11}^{(21)}&d_2\Psi_{11}^{(22)}\end{pmatrix}=0.
\end{equation}
    It can be seen that $\Psi_{11}^{(11)}(k)$, $\Psi_{11}^{(21)}(k)$, $\Psi_{11}^{(12)}(k)$ and $\Psi_{11}^{(22)}(k)$ satisfy the same equation,
    respectively. Set $\tilde{a}=id_1$, similar to \label{1021}, $\Psi_{11}^{(12)}(k)$ can be expressed as a linear combination of
    $D_{\tilde{a}}(e^{\frac{\pi i}4}k)$ and $D_{\tilde{a}}(e^{-\frac{3\pi i}4}k)$. Notice that $\Psi_{11}^{(12)}\to0$ as $k\to\infty $ and the
    asymptotic expansion \eqref{1023}, $\Psi_{11}^{(12)}=0$. A similar computation shows that $\Psi_{11}^{(12)}=0$. Then $\Psi_{11}=0$ is a diagonal
    matrix and set $a_{1}=id_{1},a_{2}=id_{2}$, we have
    \begin{equation}\nonumber
    \Psi_{11}^{(11)}=c_1^{(1)}D_{a_1}(e^{\frac{\pi i}4}k)+c_2^{(1)}D_{a_1}(e^{-\frac{3\pi i}4}k),
\end{equation}
    \begin{equation}\nonumber
    \Psi_{11}^{(22)}=c_1^{(2)}D_{a_2}(e^{\frac{\pi i}4}k)+c_2^{(2)}D_{a_2}(e^{-\frac{3\pi i}4}k),
\end{equation}
    among them, $c_1^{(j)},c_2^{(j)}(j=1,2)$ are constants. Similar analysis can be applied to $\Psi_{22}(k)$, and we have
    \begin{equation}\nonumber
    A\Psi_{22}^{(11)}=c_1^{(3)}D_{-a_1}(e^{-\frac{\pi i}{4}}k)+c_2^{(3)}D_{-a_1}(e^{\frac{3\pi i}{4}}k),
\end{equation}
    \begin{equation}\nonumber
    D\Psi_{22}^{(22)}=c_1^{(4)}D_{-a_2}(e^{-\frac{\pi i}4}k)+c_2^{(4)}D_{-a_2}(e^{\frac{3\pi i}4}k),
\end{equation}
    among them, $c_{1}^{(j)},c_{2}^{(j)}(j=3,4)$ are constants. Next, we first consider the case when $\arg k\in (-\frac\pi4,\frac\pi4)$. Note that when $k\to\infty $,
    \begin{equation}\nonumber
    \Psi_{11}(k)k^{-i\nu}e^{\frac{ik^2}4}\to I_{2\times2},\quad\Psi_{22}(k)k^{i\nu}e^{-\frac{ik^2}4}\to I_{2\times2}.\nonumber
\end{equation}
    Then we have
    \begin{equation}\nonumber
    \begin{aligned}&\Psi_{11}^{(11)}(k)=\Psi_{11}^{(22)}(k)=e^{\frac{\pi\nu}{4}}D_{a_{1}}(e^{\frac{\pi i}{4}}k),\quad a_{1}=a_{2}=i\nu,\\
    &\Psi_{22}^{(11)}(k)=\Psi_{22}^{(22)}(k)=e^{\frac{\pi\nu}{4}}D_{-a_{1}}(e^{-\frac{\pi i}{4}}k).\end{aligned}
\end{equation}
    In addition, the parabolic cylinder function can be obtained
    \begin{equation}\nonumber
    \frac{\mathrm{d}D_a(\zeta)}{\mathrm{d}\zeta}+\frac\zeta2D_a(\zeta)-aD_{a-1}(\zeta)=0.
\end{equation}
    Then we have
    \begin{equation}\nonumber
    \Psi_{21}(k)=\beta_{12}^{-1}a_1e^{\frac{\pi\nu}{4}}e^{\frac{\pi i}{4}}D_{a_1-1}(e^{\frac{\pi i}{4}}k).\nonumber
\end{equation}
    For $\arg k\in(\frac\pi4,\frac{3\pi}4)$ and $k\to\infty $,
    \begin{equation}\nonumber
    \Psi_{11}(k)k^{-i\nu}e^{\frac{ik^2}4}\to I_{2\times2},\quad\Psi_{22}(k)k^{i\nu}e^{-\frac{ik^2}4}\to I_{2\times2}.\nonumber
\end{equation}
    We get
    \begin{equation}\nonumber
    \begin{aligned}&\Psi_{11}^{(11)}(k)=\Psi_{11}^{(22)}(k)=e^{-\frac{3\pi\nu}4}D_{a_1}(e^{-\frac{3\pi i}4}k),\quad a_1=a_2=i\nu,\\
    &\Psi_{22}^{(11)}(k)=\Psi_{22}^{(22)}(k)=e^{\frac{\pi\nu}4}D_{-a_1}(e^{-\frac{\pi i}4}k),\end{aligned}
\end{equation}
    which imply
    \begin{equation}\nonumber
    \Psi_{21}(k)=\beta_{12}^{-1}a_1e^{-\frac{3\pi\nu}4}e^{-\frac{3\pi i}4}D_{a_1-1}(e^{-\frac{3\pi i}4}k).\nonumber
\end{equation}
    Along the ray $\arg k=\frac\pi4$, we can infer
    \begin{equation}\nonumber
    \Psi_+(k)=\Psi_-(k)\begin{pmatrix}I_{2\times2}&0\\-\gamma(\bar{k}_0)&I_{2\times2}\end{pmatrix}.
\end{equation}
    Notice the $(2,1)$ entry of the RH problem,
    \begin{equation}\nonumber
    \beta_{12}^{-1}a_1e^{\frac{\pi(i+\nu)}4}D_{a_1-1}(e^{\frac{\pi i}4}k)=
    e^{\frac{\pi\nu}4}D_{-a_1}(e^{\frac{3\pi i}4}k)-\gamma(\bar{k}_0)+\beta_{12}^{-1}a_1e^{-\frac{\pi(3\nu+3i)}4}D_{a_1-1}(e^{-\frac{3\pi i}4}k).
\end{equation}
    The parabolic-cylinder function satisfies
    \begin{equation}\nonumber
    D_a(\zeta)=\frac{\Gamma(a+1)}{\sqrt{2\pi}}\left(e^{\frac12a\pi i}D_{-a-1}(i\zeta)+e^{-\frac12a\pi i}D_{-a-1}(-i\zeta)\right).
\end{equation}
    We can decompose $D_{-a_1}(e^{\frac{3\pi i}4}k)$ into $D_{a_1-1}(e^{\frac{\pi i}4}k)$ and $D_{a_1-1}(e^{\frac{\pi i}4}k)$. By separating
     the coefficients of the two independent functions
    \begin{equation}\label{1033}
    \left.\beta_{12}=\frac{\nu\sqrt{2\pi}e^{-\frac{\pi\nu}{2}}e^{\frac{3\pi i}{4}}}{\Gamma(-i\nu+1)\det(-\gamma({\bar{k}_0}))}
    \left(\begin{array}{cc}-\gamma_{22}(\bar{k}_0)&\gamma_{21}^*(\bar{k}_0)\\\gamma_{12}^*(\bar{k}_0)&-\gamma_{11}^*(\bar{k}_0)\end{array}\right.\right).
\end{equation}

	\bibliographystyle{plain}

\end{document}